\documentclass[11pt,twoside,ngerman,UKenglish]{article}

\usepackage[UKenglish]{babel}
\usepackage[utf8]{inputenc}
\usepackage[T1]{fontenc}
\usepackage{charter}

\usepackage{array}
\usepackage{amssymb}
\usepackage{amsmath,mathabx}
\usepackage{amsthm}
\usepackage{relsize}

\usepackage{graphicx}
\usepackage{subfigure}
\usepackage{wrapfig}
\usepackage{caption}

\usepackage{charter}
\usepackage{a4wide}
\usepackage{makeidx}
\usepackage{textcomp}
\usepackage{ulem} 
\usepackage{geometry}
\usepackage{tikz}
\usepackage{hyperref}
\usepackage{sverb}

\usepackage{listings}
\usepackage{enumerate}
\usepackage{bbold}
\usepackage{dsfont}
\usepackage{xifthen}

\usepackage[nobysame
           ,numeric
           ,initials
           ]{amsrefs}

\newcommand{\IR}{\mathds{R}}
\newcommand{\IN}{\mathds{N}}

\newcommand{\IP}{\mathds{P}}
\newcommand{\IE}{{\mathbb{E}}}
\newcommand{\E}{\mathds{E}}

\newcommand*{\cA}{\mathcal{A}}

\newcommand*{\cF}{\mathcal{F}}
\newcommand*{\cG}{\mathcal{G}}

\newcommand*{\cM}{\mathcal{M}}

\newcommand{\1}{\mathbb{1}}

\renewcommand{\phi}{\varphi}
\renewcommand{\epsilon}{\varepsilon}

\newcommand{\qvar}[2]{\langle #1 , #1 \rangle _{#2}}

\newcommand{\SEP}{\text{SEP}}

\newcommand{\esssup}{\operatorname{ess\hspace{0.5mm}sup}}
\newcommand{\Id}{\operatorname{Id}}

\renewcommand{\d}{\operatorname{d} \hspace{-0.5mm}}
\newcommand{\Dt}{\operatorname{D}^t}
\newcommand{\Dw}{\operatorname{D}^w}
\newcommand{\diff}[3][]{\ifthenelse{\isempty{#1}}{\partial_{#2} #3}{\partial_{#1} \partial_{#2} #3}}

\newcommand{\p}[3]{#1^{(#2)}_{#3}}
\newcommand{\Xz}[2]{X^{(2), (#1) }_{ #2 }}
\newcommand{\Xd}[2]{X^{(3), (#1) }_{ #2 }}

\newcommand{\X}[2]{\p{X}{#1}{#2}}
\newcommand{\x}[1]{x^{(#1)}}
\renewcommand{\u}[2]{\ifthenelse{\isempty{#2}}{u^{(#1)}}{u^{(#1)}_{#2}}} 
\newcommand{\Z}[2]{\p{\tilde{Z}}{#1}{#2}}
\newcommand{\Zo}[2]{\p{Z}{#1}{#2}}
\newcommand{\W}[1]{\widetilde{W}_{#1}}
\newcommand{\Y}[2]{\p{Y}{#1}{#2}}

\newcommand{\ds}[0]{\,\mathrm{d}}
\newcommand{\dn}[0]{\partial}

\newcommand{\fbsde}[0]{}

\renewcommand{\IP}{\mathbf{P}}

\newtheoremstyle{thm}
{10pt}
{18pt}
{\itshape}
{}
{\bf}
{}
{\newline }
{}

\newtheorem{theorem}{Theorem}[section]
\newtheorem{lemma}[theorem]{Lemma}
\newtheorem{proposition}[theorem]{Proposition}

\newtheorem{assumption}[theorem]{Assumption}

\newtheorem{definition}[theorem]{Definition}
\newtheorem{example}[theorem]{Example}
\newtheorem{remarks}[theorem]{Remark}

\numberwithin{equation}{section}
\numberwithin{figure}{section}
\numberwithin{table}{section}
 \usepackage[nodayofweek]{datetime}

\begin{document}
\selectlanguage{UKenglish}
\author{Stefan Ankirchner
\footnote{University of Jena (Germany), s.ankirchner@uni-jena.de},
Stefan Engelhardt
\footnote{S.~Engelhardt was partially supported by the German Exchange DAAD (Nr. 57369588) and acknowledges the hospitality of the University of Edinburgh.}
\footnote{University of Jena (Germany), engelhardt.stefan@uni-jena.de}, 
Alexander Fromm
\footnote{A. Fromm acknowledges support from the \emph{German Research Foundation} through the project AN 1024/4-1.}
\footnote{University of Jena (Germany), alexander.fromm@uni-jena.de}, 
Gon\c{c}alo dos Reis
\footnote{G. dos Reis acknowledges support from the \emph{Funda{\c c}$\tilde{\text{a}}$o para a Ci$\hat{e}$ncia e a Tecnologia} (Portuguese Foundation for Science and Technology) through the project UID/MAT/00297/2019 (Centro de Matem\'atica e Aplica\c c$\tilde{\text{o}}$es CMA/FCT/UNL).}
\footnote{University of Edinburgh (UK) and Centro de Matem\'atica e Aplica\c c$\tilde{\text{o}}$es/FCT/UNL (PT), G.dosReis@ed.ac.uk}
}
\title{The Skorokhod embedding problem for inhomogeneous diffusions}

\date{ \currenttime, \ddmmyyyydate\today}

\maketitle
\renewcommand*{\thefootnote}{\arabic{footnote}}

\begin{abstract} 
We solve the Skorokhod embedding problem for a class of stochastic processes satisfying an inhomogeneous stochastic differential equation (SDE) of the form $\d A_t =\mu (t, A_t) \d t + \sigma(t, A_t) \d W_t$. We provide sufficient conditions guaranteeing that for a given probability measure $\nu$ on $\IR$ there exists a bounded stopping time $\tau$ and a real $a$ such that the solution $(A_t)$ of the SDE with initial value $a$ satisfies $A_\tau \sim \nu$. We hereby distinguish the cases where $(A_t)$ is a solution of the SDE in a weak or strong sense. Our construction of embedding stopping times is based on a solution of a fully coupled forward-backward SDE. We use the so-called method of decoupling fields for verifying that the FBSDE has a unique solution. Finally, we sketch an algorithm for putting our theoretical construction into practice and illustrate it with a numerical experiment. 
\end{abstract}
{\bf Keywords:} Skorokhod embedding, decoupling fields, FBSDE.
\vspace{0.3cm}

\noindent
{\bf 2010 AMS subject classifications:}\\
Primary: 60G40. Secondary: 60H10, 60J25


%
%

%
%
%
\section{Introduction}

Let $\nu$ be a probability measure on $\IR$, let $\mu, \sigma: [0,\infty) \times \IR \to \IR$ be continuous in both arguments and let $(A_t)_{t \ge 0}$ be a stochastic process satisfying the inhomogeneous stochastic differential equation (SDE) 
\begin{equation}\label{sde:base}
\d A_t =\mu (t, A_t) \d t + \sigma(t, A_t) \d W_t,
\end{equation}
where $W$ is a Brownian motion. In this article we consider the Skorokhod embedding problem (SEP) for $\nu$ in $(A_t)$. More precisely, we provide sufficient conditions on $\mu$, $\sigma$ and $\nu$ guaranteeing the existence of a stopping time $\tau$ and a real number $a$ such that the solution of the SDE \eqref{sde:base}, in a weak or strong sense, with initial condition $A_0 = a$ satisfies $A_\tau \sim \nu$. 

We solve the embedding problem by reducing it to the forward-backward stochastic differential equation (FBSDE)
\begin{equation}
\label{fbsde:base}
\begin{array}{rl}
\X{1}{s} =& x^{(1)} + W_s \\
\X{2}{s} =& x^{(2)} + \int _0 ^s \frac{Z_r^2}{\sigma^2 ( \X{2}{r}, Y_r + \X{3}{r} ) } dr \\
\X{3}{s} =& x^{(3)} + \int_0 ^s \mu (\X{2}{r}, Y_r + \X{3}{r}) \frac{Z_r^2}{\sigma^2 ( \X{2}{r}, Y_r + \X{3}{r} ) } dr \\
Y_s =& g(\X{1}{1} ) - \X{3}{1} - \int _s ^1 Z_r \d W_r 
\end{array}
\end{equation}
for $s \in [0,1]$ and $(x^{(1)},x^{(2)},x^{(3)}) \in \IR^3$, where $g$ is a real function chosen such that $g(W_1) \sim \nu$. 
Notice that the FBSDE \eqref{fbsde:base} is fully coupled, i.e.\ the second and third forward equation depend on the solution components $Y$ and $Z$ of the backward equation; and, vice versa, the backward equation depends on the forward components $X^{(1)}$ and $X^{(3)}$. 

It is a longstanding challenge to find conditions guaranteeing that a fully coupled FBSDE possesses a solution. Sufficient conditions are provided e.g.\ in \cite{Ma1994}, \cite{Pardoux1999}, \cite{ma:yon:99}, \cite{peng:wu:99}, \cite{Delarue2002}, \cite{ma:wu:zhang:15} (see also references therein). The method of decoupling fields, developed in \cite{Fromm15} (see also the precursor articles \cite{ma:yin:zhan:12}, \cite{Fromm2013} and \cite{ma:wu:zhang:15}), is convenient for determining whether a solution exists. A decoupling field describes the functional dependence of the backward part $Y$ on the forward component $X$. The decoupling field for the particular FBSDE \eqref{fbsde:base} is, roughly speaking, a function $u$ such that for all $s \in [0,1]$
\begin{align}\label{dec condi}
u(s, \X{1}{s}, \X{2}{s}, \X{3}{s} ) =& Y_s. 
\end{align} 
Under some nice conditions on the parameters of the FBSDE, there exists a maximal non-vanishing interval possessing a solution triplet $(X,Y,Z)$ and a decoupling field with nice regularity properties. The method of decoupling fields consists in analyzing the dynamics of the decoupling field's gradient in order to determine whether the FBSDE has a solution on the whole time interval $[0, 1]$.

We use the method of decoupling fields to prove that, under some suitable conditions on $\mu$, $\sigma$ and $g$, the FBSDE \eqref{fbsde:base} has a unique solution on $[0,1]$ for every initial value. By using the particular solution with initial value $(x^{(1)},x^{(2)},x^{(3)}) = 0$, we then construct a weak solution of the SDE \eqref{sde:base} and a stopping time $\tau$ embedding $\nu$. Indeed, the second component $X^{(2)}$ of the forward part in \eqref{fbsde:base} can be interpreted as a random time change. One can show that the time change is invertible, say with inverse clock $\gamma(t)$. Moreover, there exists a filtration $(\cG_t)$ and a $(\cG_t)$-Brownian motion $B$ such that, first, $X^{(2)}_1$ is a $(\cG_t)$-stopping time and, second, under the inverse clock the solution component $Y$ together with $B$ solve the SDE \eqref{sde:base} in a weak sense.  
By the very construction the time changed process $Y_{\gamma(\cdot)}$ at $X^{(2)}_1$ is equal to $g(W_1)$, and hence $X^{(2)}_1$ is a stopping time embedding $\nu$ into a weak solution of \eqref{sde:base}. 

In a further step we characterize the embedding stopping time $X^{(2)}_1$ in terms of a four dimensional Lipschitz SDE driven by the constructed Brownian motion $B$. The SDE establishes a mapping from the paths of $B$ to $X^{(2)}_1$, and hence allows to find stopping times embedding $\nu$ into {\itshape strong} solutions of the SDE \eqref{sde:base}. 

A major idea of our approach for solving the SEP is to change the time of a stochastic process that has the wanted distribution at the deterministic time $1$.  This idea goes back to Bass \cite{Bass83} who solves the SEP for Brownian motion. Indeed, our approach generalizes Bass's solution method. If $\mu$ is zero and $\sigma$ constant equal to one, then the component $X^{(3)}$ of \eqref{fbsde:base} vanishes and the solution part $Y$ of the backward equation coincides with the martingale of conditional expectations of $g(W_1)$, which is the process used by Bass. Moreover, the time change $X^{(2)}$ coincides with the quadratic variation of $Y$, the time change used in \cite{Bass83}. 

The time change idea has been employed in several further articles. In \cite{AHI08} the solution of a quadratic BSDE is time changed in order to solve the SEP for the Brownian motion with drift. The FBSDE \eqref{fbsde:base} simplifies to the BSDE of \cite{AHI08} if $A$ is a Brownian motion with drift. 
\cite{AHS15} uses a time change argument to construct stopping times embedding a given distribution into a stochastic process solving a homogeneous SDE. In \cite{FIP15} a fully coupled FBSDE is solved and then time changed to obtain a stopping time embedding a distribution into a Gaussian process satisfying an SDE with deterministic coefficients. \cite{FIP15} also relies on the method for decoupling fields for proving existence of a solution of the FBSDE. 

There are more recent articles that are inspired by or related to Bass‘ time-change approach for solving the SEP for the Brownian motion. E.g.\ the article \cite{beiglbock2017measure} proves optimality of the Bass solution, among all solutions of the SEP for Brownian motion, for some minimization problems formulated in terms of associated measure-valued martingales. \cite{doering2017skorokhod} solve the SEP for a class of Levy processes via an analytic approach and by extending Bass’ time-change arguments. The process of conditional expectations of $g(X^{(1)}_1)$, used by Bass, is shown in \cite{Veraguas2019} to minimize a martingale transport problem.  

To the best of our knowledge there do not exist any articles that consider the SEP for general inhomogeneous diffusions of the type \eqref{sde:base}. There are various contributions to the SEP for homogeneous diffusions. The article \cite{PedersenPeskir2001} classifies the distributions that can be embedded into homogeneous diffusions. The survey \cite{obloj} collects results on the SEP, including results for homogeneous diffusions. We remark that in the homogeneous case where the coefficients of the SDE \eqref{sde:base} do not depend on time, the FBSDE \eqref{fbsde:base} can be decoupled. We explain this in Section \ref{sec:numerics} below.

The manuscript is organized as follows: In Section \ref{sec:fbsde} we present our main results. In Section \ref{sec:introDecFields} we  explain the decoupling fields technique. In Sections \ref{sec:diffDecoupling} and \ref{sec:boundingDecoupling} we compute the dynamics of the decoupling field gradient process and derive some estimates allowing to conclude the existence of an FBSDE \eqref{fbsde:base} on the whole interval. In Sections \ref{sec:weakSolution} and \ref{sec:strongSolution} we present the weak and strong solution for the SEP. Illustrative numerical results can be found in Section \ref{sec:numerics}.

%
%
%

%
%
%
\section{Main results}
\label{sec:fbsde}

Our goal is to solve the Skorokhod embedding problem (\SEP) for a stochastic process $A$ solving the SDE \eqref{sde:base}.  
More precisely, for a given probability measure $\nu$ on $\IR$ we aim at finding an integrable stopping time $\tau$ and a real $a$ such that the solution $A$ of \eqref{sde:base}, in a weak or strong sense, with intial condition $A_0 = a$ fulfills $A_\tau \sim \nu$.
Let $F_\nu$ be the cumulative distribution function of $\nu$. We set
\begin{equation*}
g:=g_\nu:=F^{-1}_\nu \circ \Phi,
\end{equation*}
where $\Phi$ is the cumulative distribution function of the standard normal distribution and $F^{-1}_\nu$ the right-continuous generalized inverse of $F_\nu$. In the following, for a differentiable function $f:\IR^n \to \IR$ we denote by $\diff{x_i}{f}$ its partial derivate with respect to the $i$th coordinate.

\begin{assumption} 
\label{asump:conditionsU1bounded}
Let $g$, $\mu$ and $\sigma$ be differentiable, $\sigma \geq \epsilon > 0$ and $g'$, $\frac{\mu}{\sigma^2}$, $\frac{\diff{t}{\mu}}{\sigma^2}$, $\frac{\diff{a}{\mu}}{\sigma^2}$, $\frac{\diff{t}{\sigma}}{\sigma}$ as well as $\frac{\diff{a}{\sigma}}{\sigma}$ be bounded. Furthermore, let
\begin{equation}\label{inequ1 assu}
\inf_{(\theta,x)\in \IR_+ \times \IR} \frac{ \sigma \cdot \diff{a}{\mu} - 2 \diff{a}{\sigma} \cdot \mu }{\sigma^3}(\theta, x) > - \frac{1}{ 2 \Vert g' \Vert_\infty^2}
\end{equation}
and one of the following conditions be satisfied:
\begin{enumerate}[i)]
\item $\diff{a}{\sigma} \equiv 0$
\item $\diff{a}{\sigma} \geq 0$, $2\diff{t}{\sigma} \cdot \mu  - \sigma \cdot \diff{t}{\mu} \geq 0$ \ or
\item $\diff{a}{\sigma} \leq 0$, $2\diff{t}{\sigma} \cdot \mu - \sigma \cdot \diff{t}{\mu} \leq 0$.
\end{enumerate}
\end{assumption}

Our main results are the following theorems. 

\begin{theorem}\label{thm:mainResultWeak}
Let Assumption \ref{asump:conditionsU1bounded} be satisfied. Then there exists a complete filtered probability space $(\Omega, \cF, (\cG_t)_{t \ge 0}, \IP)$, a $(\cG_t)$-Brownian motion $(B_t)$, a bounded $(\cG_t)$-stopping time $\tau$ and a real number $a$ such that for the strong solution $A$ of the SDE \eqref{sde:base} with driving Brownian motion $B$ and initial condition $A_0 = a$ we have $A_\tau \sim \nu$. 
Furthermore, $\tau$ can be chosen such that
\begin{equation}
\label{est:tau}
\tau \leq \epsilon^{-2} \left( \frac{1}{\Vert g' \Vert_\infty ^2} + 2 \min \left\{0, \inf_{ (\theta,x) \in \IR_+ \times \IR } \left( \frac{ \sigma \cdot \diff{a}{\mu} - 2 \diff{a}{\sigma} \cdot \mu }{\sigma^3} \right) (\theta,x) \right\} \right)^{-1} \quad a.s.
\end{equation}
\end{theorem}

\begin{remarks}
In the following we refer to the tupel $\big( (\cG_t), (B_t), \tau, a \big)$ as a weak solution of the \SEP.
\end{remarks}

\begin{theorem}\label{thm:mainResultStrong}
Let Assumption \ref{asump:conditionsU1bounded} be satisfied and assume furthermore that $\sigma$, $\frac{1}{g'}$ the first, second and third derivatives of $g$, $\mu$ and $\sigma$ are bounded.
Let $B$ be a Brownian motion on a probability space $(\Omega, \cF, P)$ and denote by $(\cF_t)$ the augmented Brownian filtration. Then there exists $a \in \IR$ and a bounded $(\cF_t)$-stopping time $\tau$ satisfying \eqref{est:tau} such that for the strong solution $A$ of the SDE \eqref{sde:base} with driving Brownian motion $B$ and initial condition $A_0 = a$ we have $A_\tau \sim \nu$. 
\end{theorem}

\begin{remarks}
We refer to the pair $(\tau, a)$ as a strong solution of the \SEP.
\end{remarks}

\begin{remarks}
\label{remark:muBounded}
Note that the combination of Assumption \ref{asump:conditionsU1bounded} and $\sigma$ being bounded already implies that $\mu$ is bounded as well.
\end{remarks}

\begin{remarks}
\label{remark:InterpretingAssumptions}
We now comment on Assumption \ref{asump:conditionsU1bounded}. In particular, we relate the assumption to some conditions appearing in the literature that have been shown to be sufficient for a bounded solution of the SEP to exist. 
\begin{enumerate}[a)]
\item The assumption that $g'$ is bounded entails that there exists a compact set outside of which the tails of $\nu$ are dominated by the tails of  a normal distribution. If, as in Theorem \ref{thm:mainResultStrong}, we additionally have that $g'$ is bounded from below by a positive constant, then the tails of $\nu$ also dominate the tails of a normal distribution. For a precise statement, see Lemma \ref{lemma:appendixGMeans} in the appendix.

Furthermore, observe that the left hand side of Condition \eqref{inequ1 assu} is equal to $\partial_a \left( \frac{\mu}{\sigma^2} \right)$ and in the cases $ii)$ and $iii)$ the term $2\diff{t}{\sigma} \cdot \mu  - \sigma \cdot \diff{t}{\mu}$ equals $-\sigma^3 \partial_t \left( \frac{\mu}{\sigma^2} \right)$; hence Assumption \ref{asump:conditionsU1bounded} imposes conditions on the growth of $\frac{\mu}{\sigma^2}$.

\item
Theorem  3.1 in \cite{AS11} states that the boundedness of $g'$ is sufficient for the SEP for the BM, possibly with a constant drift, to possess a bounded solution. Notice that for $\sigma \equiv 1$ and constant $\mu$ Inequality \eqref{est:tau} simplifies to 
\begin{align*}
\tau \le \| g'\|_\infty^2,
\end{align*}
and hence coincides with the estimate on the embedding stopping time provided in Theorem 3.1 in \cite{AS11}. Moreover, observe that if $\sigma$ and $\mu$ are constant, then all the other properties of Assumption \ref{asump:conditionsU1bounded} are satisfied trivially. 

\item The ratio on the left-hand side of \eqref{inequ1 assu} is equal to $\partial_a \left( \frac{\mu}{\sigma^2} \right)$. Thus, \eqref{inequ1 assu} is somewhat weaker than requiring that $\frac{\mu}{\sigma^2}$ is non-decreasing in $x$. 
For some mean-reversion processes, e.g.\ the Ornstein-Uhlenbeck process, $\partial_a \left( \frac{\mu}{\sigma^2} \right)$ is unbounded from below. A mean reversion effect can imply that at any time the tails of the diffusion $A$ are lighter than the tails of $\nu$; in this case $\nu$ can not be embedded into $A$ in bounded time.  

A condition related to \eqref{inequ1 assu} appears in Theorem 6 of the article \cite{AHS15} studying the SEP in the special case where $\mu$ and $\sigma$ do only depend on $x$. The theorem states that if $-\frac{2\mu}{\sigma} + \sigma'$ is non-increasing and $\frac{g'}{\sigma(g)}$ is bounded, then there exists a bounded solution of the SEP. Note that if, in addition, $\sigma$ is constant, the assumption of Theorem 6, \cite{AHS15}, coincides with our Assumption \ref{asump:conditionsU1bounded}. 

\item In \cite{FIP15} the authors consider the special case when $\mu,\sigma$ do not depend on $a$, but on time only. To obtain weak solutions for the SEP using the FBSDE approach the authors of that work assume that $\sigma$ is bounded away from zero as well as that $g'$ and $\delta'$ are bounded, where $\delta'(r)=\frac{\mu(H^{-1}(r))}{\sigma^2(H^{-1}(r))}$ and where $H^{-1}$ is the inverse of the mapping $t\mapsto\int_0^t\sigma^2(s) ds$. This boundedness of $\delta'(r)$ is equivalent to our assumption that $\frac{\mu}{\sigma^2}$ is bounded. 
\end{enumerate}
\end{remarks}

In order to derive Theorem \ref{thm:mainResultWeak} and \ref{thm:mainResultStrong} we consider the FBSDE \eqref{fbsde:base}. To this end let $W$ be a Brownian motion on a probability space $(\Omega, \cF, \IP)$ and denote by $(\cF_t)_{t \ge 0}$ the associated augmented Brownian filtration. In Section \ref{sec:diffDecoupling} and \ref{sec:boundingDecoupling} we show that under Assumption \ref{asump:conditionsU1bounded} there exists a unique solution of the FBSDE \eqref{fbsde:base} with initial condition $(\X{1}{0}, \X{2}{0}, \X{3}{0}) = ( 0,0,0 )$. 
We then use this solution and a time transformation to prove Theorem \ref{thm:mainResultWeak} (see Section \ref{sec:weakSolution} and in particular Theorem \ref{thm:weakSolution}). 
More precisely, we construct a filtration $(\cG_t)$, a $(\cG_t)$-Brownian motion $(B_t)$, a bounded $(\cG_t)$-stopping time $\tau$ and find a real number $a$ such that for the strong solution $A$ of the SDE \eqref{sde:base} with driving Brownian motion $B$ and initial condition $A_0 = a$ we have $A_\tau \sim \nu$. 

In order to find a strong solution of the \SEP, we transform the FBSDE \eqref{fbsde:base} via a time change into an SDE driven by the new Brownian motion $B$. The new SDE allows to characterize the stopping time $\tau$ as a path functional of $B$, and hence to prove Theorem \ref{thm:mainResultStrong} \ (see Section \ref{sec:strongSolution} and in particular Theorem \ref{thm:stoppingB:strongSolution} and Proposition \ref{prop:conditionsStronSolution}).

In Section \ref{sec:numerics} we show that solving the system
\begin{align}
W_s =&  \int_0^s \frac{\sigma ( \X{2}{r}, Y_r + \X{3}{r} ) }{ Z_r } \d B_{\X{2}{r}} \nonumber\\
\X{2}{s} =& \int _0 ^s \frac{Z_r^2}{\sigma^2 ( \X{2}{r}, Y_r + \X{3}{r} ) } \d r \nonumber\\
\X{3}{s} =& \int_0 ^s \mu (\X{2}{r}, Y_r + \X{3}{r}) \frac{Z_r^2}{\sigma^2 ( \X{2}{r}, Y_r + \X{3}{r} ) } \d r \nonumber\\
Y_s =& g( W_1 ) - \X{3}{1} - \int _s ^1 Z_r \d W_r  \label{sys:numscheme}
\end{align}
for all $s \in [0,1]$ and setting $\tau := \X{2}{1}$ also yields a strong solution. Furthermore, we propose a scheme, based on the system  \eqref{sys:numscheme}, to numerically simulate a solution of the SEP (see Section \ref{sec:numerics}). 

In the next section we recall some facts concerning decoupling fields and explain the method we use for proving the existence of a unique solution for the FBSDE \eqref{fbsde:base}.

\section{The method of decoupling fields}
\label{sec:introDecFields}

In this section we briefly summarize the key results of the abstract theory of Markovian decoupling fields, we rely on later in the paper. The presented theory is derived from the SLC theory (standing for Standard Lipschitz Conditions) of Chapter 2 of \cite{Fromm15} and is proven in \cite{FIP15}.

We consider families $(M,\Sigma,f)$ of measurable functions, more precisely
$$ M: [0,T]\times\IR^n\times\IR^m\times\IR^{m\times d}\longrightarrow \IR^n, $$
$$ \Sigma: [0,T]\times\IR^n\times\IR^m\times\IR^{m\times d}\longrightarrow \IR^{n\times d}, $$
$$ f: [0,T]\times\IR^n\times\IR^m\times\IR^{m\times d}\longrightarrow \IR^m, $$
where $n,m,d\in\IN$ and $T>0$. Let further $(\Omega,\cF,\IP)$ be a probability space with a $d$-dimensional Brownian motion $(W_t)_{t\in[0,T]}$ and denote by $(\cF_t)_{t \in [0,T]}$ the augmented Brownian filtration.

For $x \in \IR^n$ and measurable $\xi : \IR^n \to \IR^m$ we consider the FBSDE
\begin{align*}
X_t &= x + \int_0^t M (s, X_s, Y_s, Z_s ) \d s + \int_0^t \Sigma (s, X_s, Y_s, Z_s) \d W_s \\
Y_t &= \xi ( X_T ) + \int_t^T f (s, X_s, Y_s, Z_s ) \d s - \int_t^T Z_s \d W_s.
\end{align*}
The aim is to study existence and uniqueness properties of the above FBSDE. The basic idea is to find a ''good'' function $u$ such that $Y_t = u(t,X_t)$, thereby establishing a pathwise relation between the processes $X$ and $Y$.

Note that contrary to Chapter 2 of \cite{Fromm15} we allow deterministic mappings $M,\Sigma,f$ and $\xi:\IR^n\rightarrow\IR^m$ only. In this, so-called Markovian, case we can somewhat relax the Lipschitz continuity assumptions of Chapter 2 of \cite{Fromm15} and still obtain local existence together with uniqueness. What makes the Markovian case so special is the property
$$"Z_s=u_x(s,X_s)\cdot\Sigma(s,X_s,Y_s,Z_s)"$$
which comes from the fact that $u$ will also be deterministic. This property allows us to bound $Z$ by a constant if we assume that $\Sigma$ and $u_x$ are bounded.
This boundedness of $Z$ in the Markovian case motivates the following definition, which allows to develop a theory for non-Lipschitz problems.

For a stochastic process $A:\Omega \times I \to \IR^N$, where $I$ is an interval in $[0,T]$ and $N \in \IN$, we introduce the norm
\begin{equation*}
\Vert A \Vert_{\infty, I} := \esssup \limits _{(s,\omega) \in I \times \Omega} \vert A_s(\omega) \vert
\end{equation*}
with regard to the product measure $\lambda \times \IP$ and for a function $f:I \times \IR^N \to \IR^M$ with $N,M \in \IN$ we define 
\begin{equation*}
\Vert f \Vert_{\infty,I} := \sup_{s \in I} \sup_{x \in \IR^N} |f(s,\cdot) \vert.
\end{equation*}
We simply write $\|A\|_{\infty,t_1}$ and $\|f\|_{\infty,t_1}$ if $I = [t_1, T]$ and $\|A\|_{\infty}$ and $\|f\|_{\infty}$ if $I=[0,T]$.

\begin{definition}
Let $\xi:\IR^n\rightarrow\IR^m$ be measurable and let $t\in[0,T]$.
We call a function $u:[t,T]\times\IR^n\rightarrow\IR^m$ with $u(T,\cdot)=\xi$ a \emph{Markovian decoupling field} for $\fbsde (\xi,(M,\Sigma,f))$ on $[t,T]$ if for all $t_1,t_2\in[t,T]$ with $t_1\leq t_2$ and any $\cF_{t_1}$ - measurable $X_{t_1}:\Omega\rightarrow\IR^n$ there exist progressive processes $X,Y,Z$ on $[t_1,t_2]$ such that
\begin{itemize}
\item $X_s=X_{t_1}+\int_{t_1}^sM(r,X_r,Y_r,Z_r)\ds r+\int_{t_1}^s\Sigma(r,X_r,Y_r,Z_r)\ds W_r$ a.s.,
\item $Y_s=Y_{t_2}-\int_{s}^{t_2}f(r,X_r,Y_r,Z_r)\ds r-\int_{s}^{t_2}Z_r\ds W_r$ a.s.,
\item $Y_s=u(s,X_s)$ a.s.
\end{itemize}
for all $s\in[t_1,t_2]$ and such that $\|Z\|_{\infty ,[t_1,t_2]}<\infty$ holds. In particular, we want all integrals to be well-defined and $X,Y,Z$ to have values in $\IR^n$, $\IR^m$ and $\IR^{m\times d}$ respectively. \\
Furthermore, we call a function $u:(t,T]\times\IR^n\rightarrow\IR^m$ a Markovian decoupling field for \linebreak $\fbsde(\xi,(M,\Sigma,f))$ on $(t,T]$ if $u$ restricted to $[t',T]$ is a Markovian decoupling field for all $t'\in(t,T]$.

We refer to the stated property that $Y_s = u(s,X_s)$ a.s.\ as the decoupling condition.
\end{definition}

In the following we work with weak derivatives. This allows us to obtain variational differentiability (i.e.\ w.r.t.\ the initial value $x\in\IR^n$) of the processes $X,Y,Z$ for Lipschitz (or locally Lipschitz) continuous $M,\Sigma,f,\xi$.  We start by fixing notation and giving some definitions:

If $x\in\IR^{m\times d}$ or $x\in\IR^{n\times d}$, the expression $|x|$ denotes the Frobenius norm of the linear operator $x$, i.e.\ the square root of the sum of the squares of its matrix coefficients. \\
We denote by $S^{n-1}:=\{x\in\IR^n\,|\,|x|=1\}$\index{$S^{n-1}$} the $(n-1)$ - dimensional sphere. If $x\in\IR^{n\times n}$ or $x\in\IR^{m\times n}$ or $x\in\IR^{m \times d \times n}$ or $x\in\IR^{n \times d \times n}$, we define $|x|_v:=|x\cdot v|$ for all $v\in S^{n-1}$, where $\cdot$ is the application of the linear operator $x$ to the vector $v$ such that $x\cdot v$ is in $\IR^{n}$ or $\IR^{m}$ or $\IR^{m \times d}$ or $\IR^{n \times d}$ respectively. We refer to $\sup_{v\in S^{n-1}}|x|_v$ as the operator norm of $x$.

For a measurable map $\xi: \IR^n\rightarrow \IR^m$ we define
$$L_{\xi}\index{$L_{\xi}$}:=\inf\left\{L\geq 0\,|\,|\xi(x)-\xi(x')|\leq L|x-x'|\textrm{ for all }x,x'\in\IR^n\right\}, $$
where $\inf \emptyset:=\infty$. We also set $L_{\xi}:=\infty$ if $\xi$ is not measurable. $L_{\xi}<\infty$ implies that $\xi$ is Lipschitz continuous. For a map $u: [t,T]\times\IR^n\rightarrow \IR^m$ we define $L_{u,x}:=\sup_{s\in[t,T]}L_{u(s,\cdot)}$. 

Now, consider a mapping $X:\cM\times\Lambda\rightarrow\IR$, where $(\cM,\cA,\rho)$ is some measure space with finite measure $\rho$ and $\Lambda\subseteq\IR^N$ is open, $N\in\IN$. We say that $X$ is \emph{weakly differentiable}\index{weakly differentiable} w.r.t. the parameter $\lambda\in\Lambda$, if for almost all $\omega\in\cM$ the mapping $X(\omega,\cdot):\Lambda\rightarrow\IR$ is weakly differentiable. This means that there exists a mapping $\diff{\lambda}{X}:\cM\times\Lambda\rightarrow\IR^{1\times N}$ such that
\begin{equation}\label{weakderidefi}
 \int_{\Lambda}\varphi(\lambda)\diff{\lambda}{X}(\omega,\lambda)\ds\lambda=
-\int_{\Lambda}X(\omega,\lambda)\diff{\lambda}{\varphi}(\lambda)\ds\lambda,
\end{equation}
for any real valued test function $\varphi\in C^{\infty}_c(\Lambda)$, for almost all $\omega\in\cM$. In particular, $X(\omega,\cdot)$ and the \emph{weak derivative}\index{weak derivative} $\diff{\lambda}{X}(\omega,\cdot)$ have to be locally integrable for a.a. $\omega$. This of course includes measurability w.r.t. $\lambda$ for almost every \emph{fixed} $\omega$. 

We remark that weak differentiability for vector valued mappings is defined component-wise. We refer to Section 2.1.2 of \cite{Fromm15} for more on weak derivatives. 

Note that if $L_{u,x}<\infty$ is satisfied and, therefore, $u$ is Lipschitz continuous in $x$ then $u$ is weakly differentiable in $x$ (see e.g.\ Lemma A.3.1.\ of \cite{Fromm15}) and even classically differentiable almost everywhere. If not otherwise specified we refer to $\diff{x}{u}:[t,T]\times\IR^n\rightarrow \IR^{m\times n}$ as the particular version of the weak derivative which is identical to the classical derivative in all points for which a classical derivative exists and is zero in all other points. See for instance the statement and proof of Lemma A.3.1.\ of \cite{Fromm15} for details.

We denote by $L_{\Sigma,z}$ the Lipschitz constant of $\Sigma$ w.r.t. the dependence on the last component $z$ (and w.r.t. the Frobenius norms on $\IR^{m\times d}$ and $\IR^{n\times d}$), by which we mean the minimum of all Lipschitz constants or $\infty$ in case $\Sigma$ is not Lipschitz continuous in $z$.
In case $L_{\Sigma,z}<\infty$ we denote by $L_{\Sigma,z}^{-1}=\frac{1}{L_{\Sigma,z}}$ the value $\frac{1}{L_{\Sigma,z}}$ if $L_{\Sigma,z}>0$ and $\infty$ otherwise.

We write $\IE_{t,\infty}[X]$ for $\esssup\,\IE[X|\cF_t]$ in the following definition:

\begin{definition}
Let $u:[t,T]\times\IR^n\rightarrow\IR^m$ be a Markovian decoupling field to $\fbsde(\xi,(M,\Sigma,f))$. We call $u$ \emph{weakly regular}, if
$L_{u,x}<L_{\Sigma,z}^{-1}$ and $\sup_{s\in[t,T]}|u(s,0)|<\infty$.

Furthermore, we call a weakly regular $u$ \emph{strongly regular} if for all fixed $t_1,t_2\in[t,T]$, $t_1\leq t_2,$ the processes $X,Y,Z$ arising in the defining property of a Markovian decoupling field
are a.e. unique for each \emph{constant} initial value $X_{t_1}=x\in\IR^n$ and satisfy
\begin{equation}\label{STRongregul1}
\sup_{s\in [t_1,t_2]}\IE_{t_1,\infty}[|X_s|^2]+\sup_{s\in [t_1,t_2]}\IE_{t_1,\infty}[|Y_s|^2]
+\IE_{t_1,\infty}\left[\int_{t_1}^{t_2}|Z_s|^2\ds s\right]<\infty\quad\forall x\in\IR^n.
\end{equation}
 In addition $X, Y, Z$ must be measurable as functions of $(x,s,\omega)$ and
even weakly differentiable w.r.t. $x\in\IR^n$ such that for every $s\in[t_1,t_2]$ the mappings $X_s$ and $Y_s$ are measurable functions of $(x,\omega)$ and even weakly differentiable w.r.t. $x$ such that
\begin{multline}\label{STRongregul2}
\esssup_{x\in\IR^n}\sup_{v\in S^{n-1}}\sup_{s\in [t_1,t_2]}\IE_{t_1,\infty}\left[\left|\frac{\dn}{\dn x}X_s\right|^2_v\right]<\infty, \\
\esssup_{x\in\IR^n}\sup_{v\in S^{n-1}}\sup_{s\in [t_1,t_2]}\IE_{t_1,\infty}\left[\left|\frac{\dn}{\dn x}Y_s\right|^2_v\right]<\infty, \\
\esssup_{x\in\IR^n}\sup_{v\in S^{n-1}}\IE_{t_1,\infty}\left[\int_{t_1}^{t_2}\left|\frac{\dn}{\dn x}Z_s\right|^2_v\ds s\right]<\infty,
\end{multline}
where $S^{n-1}$ is the $(n-1)$ - dimensional sphere. \\
We say that a Markovian decoupling field $u$ on $[t,T]$ is \emph{strongly regular} on a subinterval \linebreak $[t_1,t_2]\subseteq[t,T]$ if $u$ restricted to $[t_1,t_2]$ is a strongly regular Markovian decoupling field for $\fbsde(u(t_2,\cdot),(M,\Sigma,f))$. \\
Furthermore, we say that a Markovian decoupling field $u:(t,T]\times\IR^n\rightarrow\IR^m$
\begin{itemize}
\item is weakly regular if $u$ restricted to $[t',T]$ is weakly regular for all $t'\in(t,T]$,
\item is strongly regular if $u$ restricted to $[t',T]$ is strongly regular for all $t'\in(t,T]$.
\end{itemize}
\end{definition}
For the following class of problems an existence and uniqueness theory is developed:
\begin{definition}
We say that $\xi,M,\Sigma,f$ satisfy \emph{modified local Lipschitz conditions (MLLC)} if
\begin{itemize}
\item $M,\Sigma,f$ are
\begin{itemize}
\item Lipschitz continuous in $x,y,z$ on sets of the form $[0,T]\times\IR^n\times\IR^{m} \times B$, where $B\subset \IR^{m\times d}$ is an arbitrary bounded set
\item and such that $\|M(\cdot,0,0,0)\|_\infty,\|f(\cdot,0,0,0)\|_{\infty},\|\Sigma(\cdot,\cdot,\cdot,0)\|_{\infty},L_{\Sigma,z}<\infty$,
\end{itemize}
\item $\xi: \IR^n\rightarrow \IR^m$ satisfies $L_{\xi}<L_{\Sigma,z}^{-1}$.
\end{itemize}
\end{definition}
The following natural concept introduces a type of Markovian decoupling field for non-Lipschitz problems (non-Lipschitz in $z$), to which nevertheless standard Lipschitz results can be applied.

\begin{definition}
Let $u$ be a Markovian decoupling field for $\fbsde(\xi,(M,\Sigma,f))$. We call $u$ \emph{controlled in $z$} if there exists a constant $C>0$ such that for all $t_1,t_2\in[t,T]$, $t_1\leq t_2$, and all initial values $X_{t_1}$, the corresponding processes $X,Y,Z$ from the definition of a Markovian decoupling field satisfy $|Z_s(\omega)|\leq C$,
for almost all $(s,\omega)\in[t,T]\times\Omega$. If for a fixed triple $(t_1,t_2,X_{t_1})$ there are different choices for $X,Y,Z$, then all of them are supposed to satisfy the above control.

We say that a Markovian decoupling field $u$ on $[t,T]$ is \emph{controlled in $z$} on a subinterval $[t_1,t_2]\subseteq[t,T]$ if $u$ restricted to $[t_1,t_2]$ is a Markovian decoupling field for $\fbsde(u(t_2,\cdot),(M,\Sigma,f))$ that is controlled in $z$.

Furthermore, we call a Markovian decoupling field on an interval $(s,T]$ \emph{controlled in $z$} if it is controlled in $z$ on every compact subinterval $[t,T]\subseteq (s,T]$ (with $C$ possibly depending on $t$).
\end{definition}

\begin{definition}
Let $I^{M}_{\mathrm{max}}\subseteq[0,T]$ for $\fbsde(\xi,(M,\Sigma,f))$ be the union of all intervals $[t,T]\subseteq[0,T]$ such that there exists a weakly regular Markovian decoupling field $u$ on $[t,T]$.
\end{definition}

\begin{theorem}[Existence and uniqueness on a maximal interval, Theorem 3.21 in \cite{FIP15}.]\label{GLObalexistM}
Let $M,\Sigma,f,\xi$ satisfy MLLC. Then there exists a unique weakly regular Markovian decoupling field $u$ on $I^M_{\mathrm{max}}$. This $u$ is also controlled in $z$, strongly regular and continuous. \\
Furthermore, either $I^{M}_{\mathrm{max}}=[0,T]$ or $I^{M}_{\mathrm{max}}=(t^{M}_{\mathrm{min}},T]$, where $0\leq t^{M}_{\mathrm{min}}<T$.
\end{theorem}

Existence of weakly regular decoupling fields implies existence and uniqueness of classical solutions:

\begin{lemma}[Theorem 3.18 in \cite{FIP15}.]\label{UNIqXYZM}
Let $M,\Sigma,f,\xi$ satisfy MLLC and assume that there exists a weakly regular Markovian decoupling field $u$ on some interval $[t,T]$.\\ Then for any initial condition $X_t=x\in\IR^n$ there is a unique solution $(X,Y,Z)$ of the FBSDE on $[t,T]$ such that
$$\sup_{s\in[t,T]}\IE[|X_s|^2]+\sup_{s\in[t,T]}\IE[|Y_s|^2]+\|Z\|_{\infty ,t}<\infty.$$
\end{lemma}

The following result basically states that for a singularity $t^{M}_{\mathrm{min}}$ to occur $\diff{x}{u}$ has to "explode" at $t^M_{\mathrm{min}}$. It is the key to showing well-posedness for particular problems via contradiction.

\begin{lemma}[Lemma 3.22 in \cite{FIP15}.]\label{EXPlosionM}
Let $M,\Sigma,f,\xi$ satisfy MLLC. If $I^{M}_{\mathrm{max}}=(t^{M}_{\mathrm{min}},T]$, then
$$\lim_{t\downarrow t^{M}_{\mathrm{min}}}L_{u(t,\cdot)}=L_{\Sigma,z}^{-1},$$
where $u$ is the unique weakly regular Markovian decoupling field from Theorem \ref{GLObalexistM}.
\end{lemma}

In the following sections we will use the aforementioned theoretical results to study the solvability and regularity of system \eqref{fbsde:base}. This FBSDE naturally implies parameter functions $M,\Sigma,f$ and $\xi$ such that $n=3$, $d=m=1$ and $T=1$. Note that in our case $f$ vanishes, while $\Sigma$ is, in some sense, degenerate. We have $L_{\Sigma,z}=0$ and $L_{\Sigma,z}^{-1}=\infty$. Our aim is to rigourously conduct the following steps and arguments: Considering the maximal interval $I^{M}_{\mathrm{max}}$ associated with our problem, we employ Theorem \ref{GLObalexistM} to obtain a decoupling field $u$ on an arbitrary non-empty interval $[t,T] \subseteq I^{M}_{\mathrm{max}}$ such that $u$ is Lipschitz continuous in $x$ with a Lipschitz constant possibly depending on $t$. By studying the object $\partial_x u \left( s, X_s \right)$ we derive a bound for the Lipschitz constant of $u$ which is independent of $t$. The final step is to use Lemma \ref{EXPlosionM} to conclude that the case $I^{M}_{\mathrm{max}}=(t^{M}_{\mathrm{min}},T]$ cannot be fulfilled and hence, by Theorem \ref{GLObalexistM}, $I^{M}_{\mathrm{max}}=[0,T]$ must hold, which means that our FBSDE has a solution.

\section{Gradient dynamics of the decoupling field}
\label{sec:diffDecoupling}

In this section we investigate the dynamics of the spatial gradient of the decoupling field for the FBSDE \eqref{fbsde:base}. Based on the findings of this section we will derive, in the subsequent section, a uniform bound for the Lipschitz constant of the decoupling field.

Let $g$, $\mu$ and $\sigma$ be differentiable, $\sigma \geq \epsilon > 0$ and $g'$, $\frac{\mu}{\sigma^2}$, $\frac{\diff{t}{\mu}}{\sigma^2}$, $\frac{\diff{a}{\mu}}{\sigma^2}$, $\frac{\diff{t}{\sigma}}{\sigma}$ as well as $\frac{\diff{a}{\sigma}}{\sigma}$ be bounded.  

It is straightforward to verify that the associated FBSDE satisfies (MLLC) such that the theory of the previous section is applicable. By Theorem \ref{GLObalexistM} the maximal interval $I_{\mathrm{max}}^M$ contains an interval $[t,1]$ with $t < 1$. Let $x \in \IR^3$ and denote by $X = (\X{1}{}, \X{2}{}, \X{3}{})^\top, Z, Y$ the solution of the FBSDE \eqref{fbsde:base} on $[t,1]$ with initial condition $(\X{1}{t}, \X{2}{t}, \X{3}{t}) = x$. Moreover, denote by $u$ the decoupling field associated to the FBSDE \eqref{fbsde:base}. From Theorem \ref{GLObalexistM} we also know that the partial derivatives $\diff{x_1}{u}$, $\diff{x_2}{u}$, $\diff{x_3}{u}$ and the process $Z$ are bounded on $[t,1]$.

For shorter notation we define for all $s \in [t,1]$ 
\begin{align*}
\sigma_s := \sigma (\X{2}{s}, Y_s + \X{3}{s}), \qquad & \mu_s := \mu ( \X{2}{s}, Y_s + \X{3}{s} ),\\
\sigma_{t,s} := \diff{t}{\sigma} ( \X{2}{s}, Y_s + \X{3}{s} ), \qquad & \sigma_{a,s} := \diff{a}{\sigma} ( \X{2}{s}, Y_s + \X{3}{s} ), \\
\mu_{t,s} := \diff{t}{\mu} ( \X{2}{s}, Y_s + \X{3}{s} ), \qquad & \mu_{a,s} := \diff{a}{\mu} ( \X{2}{s}, Y_s + \X{3}{s} )
\end{align*}
and 
\begin{align*}
\u{1}{s} &:= \diff{x_1}{u} (s, \X{1}{s}, \X{2}{s}, \X{3}{s} ), \\ 
\u{2}{s} &:= \diff{x_2}{u} (s, \X{1}{s}, \X{2}{s}, \X{3}{s} ), \\
\u{3}{s} &:= \diff{x_3}{u} (s, \X{1}{s}, \X{2}{s}, \X{3}{s} ).
\end{align*}
In the following we refer to $\u{1}{}, \u{2}{}, \u{3}{}$ as the gradient processes associated to the inital value $x$ at time $t$.
The next result describes the dynamics of the gradient processes. For its derivation we first argue that the processes are It\^o processes and then match the coefficients appropriately. In contrast to the approach of \cite{FIP15}, we do not explicitly compute the dynamics of the inverse of the Jacobi matrix of $X$.
\begin{lemma}
\label{lemma:dynamicsFirstDiff}
Let $g$, $\mu$ and $\sigma$ be differentiable, $\sigma \geq \epsilon > 0$ and $g'$, $\frac{\mu}{\sigma^2}$, $\frac{\diff{t}{\mu}}{\sigma^2}$, $\frac{\diff{a}{\mu}}{\sigma^2}$, $\frac{\diff{t}{\sigma}}{\sigma}$ as well as $\frac{\diff{a}{\sigma}}{\sigma}$ be bounded.
Then the gradient processes $\u{1}{}$, $\u{2}{}$ and $\u{3}{}$ have the dynamics
\begin{align}
\u{1}{s} & = g' \left( \X{1}{1} \right) + \int_s ^1 \u{1}{r} \frac{Z_r^2}{\sigma_r^2 } \left( \u{3}{r} \left( \mu_{a,r} - 2 \mu_r \frac{\sigma_{a,r}}{\sigma_r} \right) - 2 \u{2}{r} \frac{\sigma_{a,r}}{\sigma_r }\right) \d r - \int_s ^1 \Z{1}{r} \d \W{r} \nonumber \\
\u{2}{s} & =  \int_s^1 \u{3}{r} \frac{Z_r^2}{\sigma_r^2 } \left( \u{2}{r} \mu_{a,r} + \mu_{t,r} \right) -2 \frac{Z_r^2}{\sigma_r^2} \left( \frac{\sigma_{t,r}}{\sigma_r} + \u{2}{r} \frac{\sigma_{a,r}}{\sigma_r} \right) \left( \u{2}{r} + \u{3}{r} \mu_r \right) \d r - \int_s^1 \Z{2}{r} \d \W{r} \nonumber \\
\u{3}{s} & = -1 + \int_s^1 \left( \u{3}{r} + 1 \right) \frac{Z_r^2}{\sigma^2_r} \left( \u{3}{r} \mu_{a,s} -2 \frac{\sigma_{a,r}}{\sigma_r} \left( \u{2}{r} + \u{3}{r} \mu_r \right) \right) \d r - \int_s^1 \Z{3}{r} \d \W{r}, \label{bsdes:ux}
\end{align}
for all $s \in [t,1]$, where $\Z{1}{}$, $\Z{2}{}$, $\Z{3}{}$ are locally square integrable processes. Moreover, the process 
\begin{equation*}
\W{s} := W_s - \int_t^s 2 \frac{Z_r}{\sigma_r^2} \left( \u{2}{r} + \u{3}{r} \mu_r \right) \d r
\end{equation*}
is a Brownian motion under an equivalent probability measure, and the Jacobi matrix
\begin{equation*}
\partial_x X_s := \left( \begin{array}{ccc}
\partial_{x_1} \X{1}{s} & \partial_{x_2} \X{1}{s} & \partial_{x_3} \X{1}{s} \\
\partial_{x_1} \X{2}{s} & \partial_{x_2} \X{2}{s} & \partial_{x_3} \X{2}{s} \\
\partial_{x_1} \X{3}{s} & \partial_{x_2} \X{3}{s} & \partial_{x_3} \X{3}{s}
\end{array} \right)
\end{equation*}
is invertible for every $s \in [t,1]$ almost surely.
\end{lemma}

\begin{proof}
For $x' = (x'_1,x'_2,x'_3)^\top \in \IR^3$, $y,z \in \IR$ we define 
\begin{equation*}
M \left( x' , y, z \right) := \left( \begin{array}{c} 0 \\ \frac{z^2}{\sigma^2 ( x'_2, y + x'_3)} \\ \mu \left( x'_2 , y + x'_3 \right) \frac{z^2}{\sigma^2(x'_2 , y + x'_3)} \end{array} \right),
\qquad
\Sigma := \left( \begin{array}{c} 1 \\ 0 \\ 0 \end{array} \right)
\end{equation*}
and 
\begin{equation*}
\xi \left( x' \right) := g(x'_1) - x'_3.
\end{equation*}
Then
\begin{equation*}
X_s = x + \int_t ^s M \left( X_r, Y_r, Z_r \right) \d r + \int_t^s \Sigma \d W_r
\end{equation*}
and
\begin{equation*}
Y_s = \xi \left( X_1 \right) - \int_s^1 Z_r \d W_r.
\end{equation*}
Now, define a stopping time $\tau$ via
\begin{equation*}
\tau := \inf \{ s \in [t,1] \vert \det \left( \partial_x X_s \right) \leq 0 \} \wedge 1.
\end{equation*}
Notice that $\tau > t$ since $\det(\partial_x X_t) = 1$. For all $s \in [t, \tau)$ we have that $\partial_x X_s$ is invertible with $( \partial_x X_s)^{-1}$ being an It\^o process. By setting
\begin{equation*}
U_s := \partial_x u \left( s, X_s \right)
= \left( \begin{array}{ccc} 
\diff{x_1}{u}, & \diff{x_2}{u}, & \diff{x_3}{u}
\end{array} \right) \left( s, X_s \right)
\end{equation*}
which is the gradient process we get
\begin{equation*}
\partial_x Y_s = U_s \cdot \partial_x X_s
\end{equation*}
for all $s \in [t, \tau)$ by the chain rule in Lemma A.3.1 in \cite{Fromm15}. Hence, $U_s = \partial_x Y_s \cdot ( \partial_x X_s )^{-1}$ is an It\^o process and thus there exist $(b_s)$ and $(\tilde{Z}_s)$ such that
\begin{equation*}
U_s = U_1 + \int_s^\tau b_r \d r - \int_s^\tau \tilde{Z}_r \d W_r
\end{equation*}
for all $s \in [t, \tau )$.

For the following we also introduce for an It\^o process $I_s = I_0 + \int_0^s i_r \d r + \int_0^s j_r \d W_r$ the two operators $\Dt$ and $\Dw$ defined via $(\Dt I)_s := i_s$ and $(\Dw I)_s := j_s$. Using this notation we have
\begin{align*}
\partial_x Z_s
&= \Dw \partial_x Y_s \\
&= \Dw \left( U_s \cdot \partial_x X_s \right) 
\\
&= U_s \cdot \Dw \partial_x X_s + \Dw U_s \cdot \partial_x X_s. 
\end{align*}
Since $\Dw \partial_x X_s = 0$, we further obtain $\partial_x Z_s= \tilde{Z}_s \cdot \partial_x X_s$
and thus we get
\begin{equation*}
\tilde{Z}_s = \partial_x Z_s \cdot \left( \partial_x X_s \right)^{-1}
\end{equation*}
for all $s \in [t, \tau )$.
Also,
\begin{align*}
\partial_x \left[ M \left( X_s, Y_s, Z_s \right) \right] \hspace*{-2cm} & \\
&= \diff{x}{M} \left( X_s, Y_s, Z_s \right) \partial_x X_s + \diff{y}{M} \left( X_s, Y_s, Z_s \right) \partial_x Y_s + \diff{z}{M} \left( X_s, Y_s, Z_s \right) \partial_x Z_s \\
&= \diff{x}{M} \left( X_s, Y_s, Z_s \right) \partial_x X_s + \diff{y}{M} \left( X_s, Y_s, Z_s \right) U_s \partial_x X_s + \diff{z}{M} \left( X_s, Y_s, Z_s \right) \tilde{Z}_s \partial_x X_s
\end{align*}
and
\begin{equation*}
0 = \Dt  \partial_x Y_s  = \Dt ( U_s \partial_x X_s ) = - b_s \cdot \partial_x X_s + U_s \cdot \partial_x \left[ M \left( X_s , Y_s , Z_s \right) \right]
\end{equation*}
yielding
\begin{equation*}
b_s
= U_s \left[ \diff{x}{M} \left( X_s, Y_s, Z_s \right) + \diff{y}{M} \left( X_s, Y_s, Z_s \right) U_s + \diff{z}{M} \left( X_s, Y_s, Z_s \right) \tilde{Z}_s \right]
\end{equation*}
for all $s \in [t, \tau )$ with
\begin{equation*}
\diff{x}{M} \left( x , y, z \right)
= \left( \begin{array}{ccc}
0 & 0 & 0 \\
0 & -2 z^2 \frac{\diff{t}{\sigma} (x_2 , y + x_3)}{\sigma^3 (x_2 , y + x_3)} & \frac{\diff{t}{\mu} \left( x_2 , y + x_3 \right) \cdot z^2}{\sigma^2(x_2 , y + x_3)} -2 z^2 \frac{\diff{t}{\sigma} (x_2 , y + x_3)}{\sigma (x_2 , y + x_3)} \frac{\mu \left( x_2 , y + x_3 \right)}{\sigma^2 (x_2 , y + x_3)} \\
0 & -2 z^2 \frac{\diff{a}{\sigma} (x_2 , y + x_3)}{\sigma^3 (x_2 , y + x_3)} & \frac{\diff{a}{\mu} \left( x_2 , y + x_3 \right) \cdot z^2}{\sigma^2(x_2 , y + x_3)} -2 z^2 \frac{\diff{a}{\sigma} (x_2 , y + x_3)}{\sigma (x_2 , y + x_3)} \frac{\mu \left( x_2 , y + x_3 \right)}{\sigma^2 (x_2 , y + x_3)} \end{array} \right) ^T,
\end{equation*}
\begin{equation*}
\diff{y}{M} \left( x , y, z \right) 
= \left( \begin{array}{c} 0 \\ 
-2 z^2 \frac{\diff{a}{\sigma} (x_2 , y + x_3)}{\sigma^3 (x_2 , y + x_3)} \\ 
\frac{\diff{a}{\mu} \left( x_2 , y + x_3 \right) \cdot z^2}{\sigma^2(x_2 , y + x_3)} -2 z^2 \frac{\diff{a}{\sigma} (x_2 , y + x_3)}{\sigma (x_2 , y + x_3)} \frac{\mu \left( x_2 , y + x_3 \right)}{\sigma^2 (x_2 , y + x_3)} \end{array} \right),
\end{equation*}
\begin{equation*}
\diff{z}{M} \left( x , y, z \right) 
= \left( \begin{array}{c} 0 \\ \frac{2 z}{\sigma^2(x_2 , y + x_3)} \\ 2 z \frac{ \mu \left( x_2 , y + x_3 \right)}{\sigma^2(x_2 , y + x_3)} \end{array} \right)
\end{equation*}
being the derivatives of $M$.

Next we turn our attention to the question whether $\partial_x X$ is invertible.
We use that on the interval $[t,1]$ the processes $U$ and $Z$ as well as the functions $\frac{1}{\sigma}$, $\frac{\mu}{\sigma^2}$, $\frac{\diff{t}{\mu}}{\sigma^2}$, $\frac{\diff{a}{\mu}}{\sigma^2}$, $\frac{\diff{t}{\sigma}}{\sigma}$ and $\frac{\diff{a}{\sigma}}{\sigma}$ are bounded, giving that $\diff{x}{M} \left( X_r, Y_r, Z_r \right)$, $\diff{y}{M} \left( X_r, Y_r, Z_r \right) U_r$ and $\diff{z}{M} \left( X_r, Y_r, Z_r \right)$ are bounded, too. Thus, there exist some bounded processes $\alpha$ and $\beta$ depending on $U$, $X$, $Y$ and $Z$, such that for every stopping time $\tilde{\tau} < \tau$, $i=1,2,3$ and $s \in [t,1]$ the process $\u{i}{\cdot \wedge \tilde{\tau}}$ has dynamics
\begin{equation*}
\u{i}{s \wedge \tilde{\tau}} = \u{i}{t} + \int_t^s \left( \p{\alpha}{i}{r} + \p{\beta}{i}{r} \cdot \Z{i}{r} \right)  \1_{ \{ r < \tilde{\tau} \} } \d r + \int_t^s \Z{i}{r} \1_{ \{ r < \tilde{\tau} \} } \d W_r.
\end{equation*}
Standard results on linear BSDEs (see e.g.\ Theorem A.1.11 in \cite{Fromm15}) yield, for every stopping time $\tilde{\tau} < \tau$ and $i=1,2,3$, that $\Z{i}{}$ has a bounded BMO($\IP$)-norm which is independent of $\tilde{\tau}$. Hence, 
\begin{equation}
\label{eq:Zbmo}
\E \left[ \int_t^\tau \vert \tilde{Z}_r \vert^2 \d r \right] < \infty.
\end{equation}
Now observe that
\begin{align*}
\partial_x X_s 
&= \Id + \int_t^s \partial_x \left[ M \left( X_r, Y_r, Z_r \right) \right] \d r \\
&= \Id + \int_t^s \left[ \diff{x}{M} \left( X_r, Y_r, Z_r \right) + \diff{y}{M} \left( X_r, Y_r, Z_r \right) U_r + \diff{z}{M} \left( X_r, Y_r, Z_r \right) \tilde{Z}_r \right] \partial_x X_r \d r
\end{align*}
implying that
\begin{equation*}
\partial_x X_s = \exp \left( \int_t^s \left[ \diff{x}{M} \left( X_r, Y_r, Z_r \right) + \diff{y}{M} \left( X_r, Y_r, Z_r \right) U_r + \diff{z}{M} \left( X_r, Y_r, Z_r \right) \tilde{Z}_r \right] \d r \right).
\end{equation*}
Together with Inequality \eqref{eq:Zbmo} this implies that $\partial_x X_s$ is invertible for all $s \in [t,\tau]$, which again yields that $\tau = 1$ and $\partial_x X$ is invertible on the whole interval $[t,1]$.

What remains to do is to calculate the explicit dynamics of $U$. Observe that
\begin{align*}
&b_s \\
&= U_s \left[ \diff{x}{M} \left( X_s, Y_s, Z_s \right) + \diff{y}{M} \left( X_s, Y_s, Z_s \right) U_s + \diff{z}{M} \left( X_s, Y_s, Z_s \right) \tilde{Z}_s \right] \\
&= \left( \u{1}{s}, \u{2}{s}, \u{3}{s} \right) \left[ 
\left( \begin{array}{ccc} 0&0&0 \\ 
0& -2 \frac{Z_s^2}{\sigma_s^2} \frac{\sigma_{t,s}}{\sigma_s } & -2 \frac{Z_s^2}{\sigma_s^2} \frac{\sigma_{a,s}}{\sigma_s } \\ 
0& \mu_{t,s} \frac{Z_s^2}{\sigma_s^2 } - 2 \mu_s \frac{Z_s^2}{\sigma_s^2} \frac{\sigma_{t,s}}{\sigma_s} & \mu_{a,s} \frac{Z_s^2}{\sigma_s^2 } - 2 \mu_s \frac{Z_s^2}{\sigma_s^2} \frac{\sigma_{a,s}}{\sigma_s}
\end{array} \right)
\right. \\ & \left. \hspace*{2,5cm} 
+ \left( \begin{array}{c} 0 \\ -2 \frac{Z_s^2}{\sigma_s^2} \frac{\sigma_{a,s}}{\sigma_s} \\ \mu_{a,s} \frac{Z_s^2}{ \sigma_s^2}  - 2 \mu_s \frac{Z_s^2}{\sigma_s^2} \frac{\sigma_{a,s}}{\sigma_s} \end{array} \right) \left( \u{1}{s}, \u{2}{s}, \u{3}{s} \right)
+ \left( \begin{array}{c} 0 \\ \frac{2 Z_s}{\sigma_s^2} \\ \frac{2 Z_s}{\sigma_s^2} \mu_s \end{array} \right) \left( \Z{1}{s}, \Z{2}{s}, \Z{3}{s} \right) \right] \\
&= \hspace*{-0.5mm} \left( \hspace*{-1.5mm} \begin{array}{c}
- 2\u{1}{s} \u{2}{s} \frac{Z_s^2}{\sigma_s^2} \frac{\sigma_{a,s}}{\sigma_s } + \u{1}{s} \u{3}{s} \frac{Z_s^2}{\sigma_s^2 } \left( \mu_{a,s} - 2 \mu_s \frac{\sigma_{a,s}}{\sigma_s} \right) \\
-2\u{2}{s}\frac{Z_s^2}{\sigma_s^2} \frac{\sigma_{t,s}}{\sigma_s} + \u{3}{s} \frac{Z_s^2}{\sigma_s^2 } \left( \mu_{t,s} - 2 \mu_s \frac{\sigma_{t,s}}{\sigma_s} \right) - 2 \left( \u{2}{s} \right)^2 \frac{Z_s^2}{\sigma_s^2} \frac{\sigma_{a,s}}{\sigma_s } + \u{2}{s} \u{3}{s} \frac{Z_s^2}{\sigma_s^2 } \left( \mu_{a,s} - 2 \mu_s \frac{\sigma_{a,s}}{\sigma_s} \right) \\
 - 2\u{2}{s} \frac{Z_s^2}{\sigma_s^2} \frac{\sigma_{a,s}}{\sigma_s } + \u{3}{s} \frac{Z_s^2}{\sigma_s^2 } \left( \mu_{a,s} - 2 \mu_s \frac{\sigma_{a,s}}{\sigma_s} \right) - 2\u{2}{s} \u{3}{s} \frac{Z_s^2}{\sigma_s^2} \frac{\sigma_{a,s}}{\sigma_s } + \left( \u{3}{s} \right) ^2 \frac{Z_s^2}{\sigma_s^2 } \left( \mu_{a,s} - 2 \mu_s \frac{\sigma_{a,s}}{\sigma_s} \right)
\end{array} \hspace*{-1.5mm} \right)^T \\
& \hspace*{8.5cm}
+ \left( \begin{array}{c}
\frac{2 Z_s}{\sigma_s^2} \left( \u{2}{s} + \u{3}{s} \mu_s \right) \Z{1}{s} \\
\frac{2 Z_s}{\sigma_s^2} \left( \u{2}{s} + \u{3}{s} \mu_s \right) \Z{2}{s} \\
\frac{2 Z_s}{\sigma_s^2} \left( \u{2}{s} + \u{3}{s} \mu_s \right) \Z{3}{s}
\end{array} \right)^T
\end{align*}
Using that $Y_1=\xi(X_1)$ and hence $U_1 = \nabla \xi (X_1)$ we obtain for the gradient processes the dynamics
\begin{align*}
\u{1}{s} & = g' \left( \X{1}{1} \right) + \int_s ^1 \u{1}{r} \frac{Z_r^2}{\sigma_r^2 } \left( \u{3}{r} \left( \mu_{a,r} - 2 \mu_r \frac{\sigma_{a,r}}{\sigma_r} \right) - 2 \u{2}{r} \frac{\sigma_{a,r}}{\sigma_r }\right) \d r - \int_s ^1 \Z{1}{r} \d \W{r} \\
\u{2}{s} & = \int_s^1 \u{3}{r} \frac{Z_r^2}{\sigma_r^2 } \left( \u{2}{r} \mu_{a,r} + \mu_{t,r} \right) -2 \frac{Z_r^2}{\sigma_r^2} \left( \frac{\sigma_{t,r}}{\sigma_r} + \u{2}{r} \frac{\sigma_{a,r}}{\sigma_r} \right) \left( \u{2}{r} + \u{3}{r} \mu_r \right) \d r - \int_s^1 \Z{2}{r} \d \W{r} \\
\u{3}{s} & = -1 + \int_s^1 \left( \u{3}{r} + 1 \right) \frac{Z_r^2}{\sigma^2_r} \left( \u{3}{r} \mu_{a,s} -2 \frac{\sigma_{a,r}}{\sigma_r} \left( \u{2}{r} + \u{3}{r} \mu_r \right) \right) \d r - \int_s^1 \Z{3}{r} \d \W{r},
\end{align*}
where $\W{s} := W_s - \int_t^s \frac{2 Z_r}{\sigma_r^2 } \left( \u{2}{r} + \u{3}{r} \mu_r \right) \d r$ for all $s \in [t,1]$. Since furthermore \\
$\frac{2 Z_s}{\sigma_r^2} \left( \u{2}{s} + \u{3}{s} \mu_s \right)$ is bounded for all $s \in [t,1]$, where $t \in I_{\mathrm{max}}^M$, we get by Girsanov's theorem that $\W{}$ is a Brownian motion for an equivalent probability measure.
\end{proof}

\section{Bounding the gradient of the decoupling field}
\label{sec:boundingDecoupling}

In this chapter we use the notations and definitions of Chapter \ref{sec:diffDecoupling}.

In the following we derive bounds for the gradient processes that do not depend on the starting time $t \in I^M_{\mathrm{max}}$ and initial value $x \in \IR^3$. In particular, we obtain global estimates for the space derivatives $\diff{x_i}{u}$, $i \in \{ 1,2,3 \}$, of the decoupling field $u$. By appealing to Lemma \ref{EXPlosionM}, we then conclude that FBSDE \eqref{fbsde:base} has a solution on the whole interval $[0,1]$.

\begin{lemma}
\label{lemma:ZequalsU1}
Assume that $g$, $\mu$ and $\sigma$ are differentiable, $\sigma \geq \epsilon > 0$ and $g'$, $\frac{\mu}{\sigma^2}$, $\frac{\diff{t}{\mu}}{\sigma^2}$, $\frac{\diff{a}{\mu}}{\sigma^2}$, $\frac{\diff{t}{\sigma}}{\sigma}$, $\frac{\diff{a}{\sigma}}{\sigma}$ are bounded. Let $u$ be the unique decoupling field to FBSDE \eqref{fbsde:base} on $I^M_{\mathrm{max}}$.

Furthermore, let $t \in I^M_{\mathrm{max}}$, $x \in \IR^3$ and $(\X{1}{}, \X{2}{}, \X{3}{}, Y, Z)$ be the solution of FBSDE \eqref{fbsde:base} with initial condition $x$ at time $t$, and let $\u{1}{}, \u{2}{}, \u{3}{}$ be the associated gradient processes. Then for $s \in [t,1]$
\begin{equation*}
\vert Z_s \vert \leq \sup_{r \in (s,1]} \sup_{x \in \IR^3} \vert \diff{x_1}{u} (r,x) \vert \quad \text{ a.s.}
\end{equation*}
and in particular $\Vert Z \Vert_{\infty,t} \leq \Vert \diff{x_1}{u} \Vert_{\infty,t}$.

Furthermore, if the weak derivative $\diff{x_1}{u}$ has a version whose restriction to the set $[t,1) \times \IR^3$ is continuous in the first two components $t$ and $x_1$, and $\diff{x_1}{u}$ is bounded, then 
\begin{equation*}
Z_s(\omega) = \diff{x_1}{u} \left( s, \X{1}{s}(\omega), \X{2}{s}(\omega), \X{3}{s}(\omega) \right) = \u{1}{s} (\omega)
\end{equation*}
for almost all $(s, \omega) \in  [t,1] \times \Omega$.
\end{lemma}

\begin{proof}
Observe that with It\^o's formula we get for $h > 0$ and $s, s+h \in [t,1]$
\begin{align*}
\frac{1}{h} \E \left[ \left. Y_{s+h} ( W_{s+h} - W_s ) \right\vert \mathcal{F}_s \right]
&= \frac{1}{h} \E \left[ \left. \int_s^{s+h} Y_r \d W_r + \int_s^{s+h} (W_{r} - W_s) Z_r \d W_r + \int_s^{s+h} Z_r \d r \right\vert \mathcal{F}_s \right] \\
&= \frac{1}{h} \E \left[ \left. \int_s^{s+h} Z_r \d r \right\vert \mathcal{F}_s \right] \\
&\rightarrow  Z_s \ \ a.s. \quad \text{ for } \quad h \rightarrow 0.
\end{align*}
On the other hand we get, using the decoupling condition $Y_r = u \left( r, \X{1}{r}, \X{2}{r}, \X{3}{r} \right)$, that
\begin{align}
Y_{s+h} &( W_{s+h} - W_s ) \nonumber \\
=& u \left( s+h, \X{1}{s+h}, \X{2}{s+h}, \X{3}{s+h} \right) (W_{s+h} - W_s ) \nonumber \\
=& u \left( s+h, \X{1}{s+h}, \X{2}{s}, \X{3}{s} \right) (W_{s+h} - W_s ) \label{est:YdWForZ} \\
& + \left( u \left( s+h, \X{1}{s+h}, \X{2}{s+h}, \X{3}{s} \right) - u \left( s+h, \X{1}{s+h}, \X{2}{s}, \X{3}{s} \right) \right) (W_{s+h} - W_s ) \nonumber \\
& + \left( u \left( s+h, \X{1}{s+h}, \X{2}{s+h}, \X{3}{s+h} \right) - u \left( s+h, \X{1}{s+h}, \X{2}{s+h}, \X{3}{s} \right) \right) (W_{s+h} - W_s ). \nonumber
\end{align}
At first let us take a look at the third summand on the right hand side of \eqref{est:YdWForZ}. Since $u$ is Lipschitz continuous in its fourth argument on $[t,1]$ with some constant $L_{u,x_3}^t$ that might depend on $t$ and since furthermore $\X{3}{s+h} = \X{3}{s} + \int_s ^{s+h} \mu_r \frac{ Z_r ^2 }{ \sigma^2 (\X{2}{r}, Y_r + \X{3}{r})} \d r$ we can estimate the absolute value of the third summand against
\begin{align*}
\frac{1}{h} & \left\vert \E \left[ \left. \left( u \left( s+h, \X{1}{s+h}, \X{2}{s+h}, \X{3}{s+h} \right) - u \left( s+h, \X{1}{s+h}, \X{2}{s+h}, \X{3}{s} \right) \right) (W_{s+h} - W_s ) \right\vert \mathcal{F}_s \right] \right\vert \\
& \leq \frac{1}{h} \E \left[ \left. \left\vert  u \left( s+h, \X{1}{s+h}, \X{2}{s+h}, \X{3}{s+h} \right) - u \left( s+h, \X{1}{s+h}, \X{2}{s+h}, \X{3}{s} \right) \right\vert \left\vert W_{s+h} - W_s  \right\vert \right\vert \mathcal{F}_s \right] \\
& \leq \frac{1}{h} \E \left[ \left. L_{u,x_3}^t \left\vert \int_s ^{s+h} \mu_r \frac{ Z_r^2 }{ \sigma^2 ( \X{2}{r}, Y_r + \X{3}{r} ) } \d r \right\vert \left\vert W_{s+h} - W_s  \right\vert \right\vert \mathcal{F}_s \right] \\
& \leq \frac{1}{h} L_{u,x_3}^t h \Big\Vert \frac{\mu}{\sigma^2} \Big\Vert_\infty \Vert Z \Vert_{\infty,t} ^2 \E \left[ \left. \left\vert W_{s+h} - W_s  \right\vert \right\vert \mathcal{F}_s \right],
\end{align*}
which clearly goes to $0$ as $h \rightarrow 0$ because $\Vert \frac{\mu}{\sigma^2} \Vert_\infty$ and $\Vert Z \Vert_{\infty,t}$ are finite on $[t,1]$.

With analogous arguments we also get that
\begin{align*}
\frac{1}{h} & \left\vert \E \left[ \left. \left( u \left( s+h, \X{1}{s+h}, \X{2}{s+h}, \X{3}{s} \right) - u \left( s+h, \X{1}{s+h}, \X{2}{s}, \X{3}{s} \right) \right) (W_{s+h} - W_s ) \right\vert \mathcal{F}_s \right] \right\vert \\
& \leq \frac{1}{h} \E \left[ \left. \left\vert  u \left( s+h, \X{1}{s+h}, \X{2}{s+h}, \X{3}{s} \right) - u \left( s+h, \X{1}{s+h}, \X{2}{s}, \X{3}{s} \right) \right\vert \left\vert W_{s+h} - W_s  \right\vert \right\vert \mathcal{F}_s \right] \\
& \leq \frac{1}{h} \E \left[ \left. L_{u,x_2}^t \left\vert \int_s ^{s+h} \frac{ Z_r^2 }{ \sigma^2 ( \X{2}{r}, Y_r + \X{3}{r} ) } \d r \right\vert \left\vert W_{s+h} - W_s  \right\vert \right\vert \mathcal{F}_s \right] \\
& \leq \frac{1}{h} L_{u,x_2}^t h \Vert Z \Vert_{\infty,t} ^2 \epsilon^{-2} \E \left[ \left. \left\vert W_{s+h} - W_s  \right\vert \right\vert \mathcal{F}_s \right] \\
& \rightarrow 0 \ \ a.s. \quad \text{ for } \quad h \rightarrow 0,
\end{align*}
where $L_{u,x_2}^t$ is the Lipschitz constant of $u$ in the third argument on the time interval $[t,1]$.

Now consider the remaining first term on the right hand side of Equation \eqref{est:YdWForZ}. For this remember
\begin{itemize}
\item $\X{1}{s}$, $\X{2}{s}$, $\X{3}{s}$ are $\mathcal{F}_s$ measurable,
\item $\X{1}{s+h} = \X{1}{s} + (W_{s+h} - W_s )$,
\item $W_{s+h} - W_s$ is independent of $\mathcal{F}_s$,
\item $u$ is deterministic, i.e.\ is 
a function of $(s,x^{(1)},x^{(2)},x^{(3)}) \in [t,1] \times \IR \times \IR \times \IR$ only.
\end{itemize}
Using integration by parts these properties imply
\begin{align*}
&\E \left[ \left. u \left( s+h, \X{1}{s+h}, \X{2}{s}, \X{3}{s} \right) ( W_{s+h} - W_s) \right| \mathcal{F}_s \right] \\
& \hspace*{5cm} = \int_\IR u \left( s+h, \X{1}{s} + z \sqrt{h}, \X{2}{s}, \X{3}{s} \right) z \sqrt{h} \frac{1}{\sqrt{2 \pi}} e^{-\frac{1}{2} z^2} \d z \\
& \hspace*{5cm} = \int_\IR \diff{x_1}{u} \left( s+h, \X{1}{s} + z \sqrt{h} , \X{2}{s}, \X{3}{s} \right) h \frac{1}{\sqrt{2 \pi}} e^{-\frac{1}{2} z^2} \d z.
\end{align*}
Hence
\begin{align*}
& \left\vert \frac{1}{h} \E \left[ \left. u \left( s+h, \X{1}{s+h}, \X{2}{s}, \X{3}{s} \right) ( W_{s+h} - W_s) \right| \mathcal{F}_s \right] \right\vert \\
& \hspace*{5cm} = \left\vert \int_\IR \diff{x_1}{u} \left( s+h, \X{1}{s} + z \sqrt{h} , \X{2}{s}, \X{3}{s} \right) \frac{1}{\sqrt{2 \pi}} e^{-\frac{1}{2} z^2} \d z \right\vert \\
& \hspace*{5cm} \leq \int_\IR \sup_{x \in \IR^3} \vert \diff{x_1}{u} (s+h,x) \vert \frac{1}{\sqrt{2 \pi}} e^{-\frac{1}{2} z^2} \d z \\
& \hspace*{5cm} = \sup_{x \in \IR^3} \vert \diff{x_1}{u} (s+h,x) \vert.
\end{align*}
Putting everything together we get
\begin{align*}
\vert Z_s \vert
&= \lim_{h \searrow 0} \left\vert \frac{1}{h} \E \left[ \left. \int_s ^{s+h} Z_r \d r \right\vert \mathcal{F}_s \right] \right\vert \\
&= \lim_{h \searrow 0} \left\vert \frac{1}{h} \E \left[ \left. Y_{s+h} ( W_{s+h} - W_s ) \right\vert \mathcal{F}_s \right] \right\vert \\
&= \lim_{h \searrow 0} \left\vert \frac{1}{h} \E \left[ \left. u \left( s+h, \X{1}{s+h}, \X{2}{s}, \X{3}{s} \right) ( W_{s+h} - W_s) \right| \mathcal{F}_s \right] \right. \\
& \left. \hspace*{0.6cm} + \frac{1}{h} \E \left[ \hspace*{-0.5mm} \left. \left( u \left( s \hspace*{-0.2mm} + \hspace*{-0.2mm} h, \X{1}{s+h}, \X{2}{s+h}, \X{3}{s} \right) - u \left( s+h, \X{1}{s+h}, \X{2}{s}, \X{3}{s} \right) \hspace*{-0.3mm} \right) (W_{s+h} - W_s ) \right\vert \mathcal{F}_s \right] \right. \\
& \left. \hspace*{0.6cm} + \frac{1}{h} \E \left[ \left. \hspace*{-0.5mm} \left( u \left( s \hspace*{-0.2mm} + \hspace*{-0.2mm} h, \X{1}{s+h}, \X{2}{s+h}, \X{3}{s+h} \right) - u \left( s+h, \X{1}{s+h}, \X{2}{s+h}, \X{3}{s} \right) \hspace*{-0.3mm} \right) (W_{s+h} - W_s ) \right\vert \mathcal{F}_s \right] \right\vert \\
& \leq \limsup_{h \searrow 0} \sup_{x \in \IR^3} \vert \diff{x_1}{u} (s+h,x) \vert + \vert 0 \vert + \vert 0 \vert \\
& \leq \sup_{r \in (s,1]} \sup_{x \in \IR^3} \vert \diff{x_1}{u} (r,x) \vert.
\end{align*}

If we have that $\diff{x_1}{u}$ is continuous in the first two arguments, we can derive, by using dominated convergence since $\u{1}{}$ is bounded on $[t,1]$, the more precise result
\begin{align*}
Z_s =& \lim_{h \searrow 0} \frac{1}{h} \E \left[ \left. u \left( s+h, \X{1}{s+h}, \X{2}{s}, \X{3}{s} \right) ( W_{s+h} - W_s) \right| \mathcal{F}_s \right] \\
=& \int_\IR \lim_{h \searrow 0} \diff{x_1}{u} \left( s+h, \X{1}{s} + z \sqrt{h} , \X{2}{s}, \X{3}{s} \right) \frac{1}{\sqrt{2 \pi}} e^{-\frac{1}{2} z^2} \d z \\
=& \diff{x_1}{u} \left( s, \X{1}{s}, \X{2}{s}, \X{3}{s} \right)
\end{align*}
almost surely.
\end{proof}

To obtain estimates for the gradient processes we use the following result.

\begin{lemma}[See \cite{MPF91}, p. 362]
\label{lemma:genaralGronwall}
Let the function $f$ be continuous and non-negative on $J=[\alpha,\beta]$, $a,b \geq 0$, and $n$ be a positive integer $(n \geq 2)$. If
\begin{equation*}
f(t) \leq a + b \int_\alpha^t f^n(s) \d s, \quad t \in J,
\end{equation*}
then
\begin{equation*}
f(t) \leq a \left[ 1 - (n-1) \int_\alpha^t a^{n-1} b \d s \right]^{\frac{1}{1-n}}, \quad \alpha \leq t \leq \beta_n,
\end{equation*}
where $\beta_n = \sup \left\{ t \in J : (n-1) \int_\alpha^t a^{n-1} b \d s < 1 \right\}$.
\end{lemma}

\begin{lemma}
\label{lemma:u3=-1:u1Bounded}
Assume that $g$, $\mu$ and $\sigma$ are differentiable, $\sigma \geq \epsilon > 0$ and $g'$, $\frac{\mu}{\sigma^2}$, $\frac{\diff{t}{\mu}}{\sigma^2}$, $\frac{\diff{a}{\mu}}{\sigma^2}$, $\frac{\diff{t}{\sigma}}{\sigma}$, $\frac{\diff{a}{\sigma}}{\sigma}$ are bounded.
Let $u$ be the unique decoupling field of the FBSDE \eqref{fbsde:base}. Then for any $t \in I^M_{\mathrm{max}}$ and initial condition $(\X{1}{t}, \X{2}{t}, \X{3}{t}) = x \in \IR^3$ the associated gradient process $\u{3}{}$ satisfies for all $s \in [t,1]$
\begin{equation*}
\u{3}{s} = -1.
\end{equation*}
If we additionally assume that $\sigma_{a,s} \cdot \u{2}{s} \geq 0$ a.s.\ for all $s \in [t,1]$ and
\begin{equation*}
\inf_{(\theta,x)\in \IR_+ \times \IR} \frac{ \sigma \cdot \diff{a}{\mu} - 2 \diff{a}{\sigma} \cdot \mu }{\sigma^3}(\theta, x) > - \frac{1}{ 2 \Vert g' \Vert_\infty^2},
\end{equation*}
then it also holds that
\begin{equation*}
0 \leq \u{1}{s} \leq \left( \frac{1}{\Vert g' \Vert_\infty ^2} + 2 \min \left\{0, \inf_{ (\theta,x) \in \IR_+ \times \IR } \left( \frac{ \sigma \cdot \diff{a}{\mu} - 2 \diff{a}{\sigma} \cdot \mu }{\sigma^3} \right) (\theta,x) \right\} \right)^{-\frac{1}{2}} < \infty
\end{equation*}
for all $s \in [t,1]$.
\end{lemma}

\begin{proof}
By interpreting \eqref{bsdes:ux} as a system of BSDEs we get for $\u{3}{}$ the trivial solution $\u{3}{s} = -1$ for all $s \in [t,1]$ as the unique bounded solution of this BSDE. 

Also  note that $g'\geq 0$ since $g = F_\nu^{-1} \circ \Phi$ and $F_\nu$ as well as $\Phi$ are non-decreasing. Thus $\check{u}_s = 0$ is the trivial and unique solution to
\begin{equation*}
\check{u}_s = 0 + \int_s^1 - \check{u}_r \frac{Z_r^2}{\sigma_r^2 } \left( - \mu_{a,r} + 2 \mu_r \frac{\sigma_{a,r}}{\sigma_r} - 2 \u{2}{r} \frac{\sigma_{a,r}}{\sigma_r }\right) \d r - \int_s^1 \Z{1}{r} \d \W{r},
\end{equation*}
which implies by comparison that $0= \check{u}_s \leq \u{1}{s}$ for all $s \in [t,1]$.

For the upper bound of $\u{1}{}$ remember that $\u{1}{s} = \diff{x_1}{u}(s, \X{1}{s}, \X{2}{s}, \X{3}{s} )$ for all $s \in [t,1]$ and in particular for any fixed $t \in I^M_{\mathrm{max}}$ and all starting conditions $x = (\x{1}, \x{2}, \x{3})\in \IR^3$ we have
\begin{equation*}
\diff{x_1}{u} (t, x )
= \u{1}{t}
= g'\left( \X{1}{1} \right) -  \int_t^1 \u{1}{r} \frac{Z_r^2}{\sigma_r^2 } \left( \mu_{a,r} - 2 \mu_r \frac{\sigma_{a,r}}{\sigma_r} + 2 \u{2}{r} \frac{\sigma_{a,r}}{\sigma_r } \right) \d r - \int_t^1 \Z{1}{r} \d \W{r}.
\end{equation*}
Using this and that $Z$ is bounded on every interval $[t,1] \subset I_{\mathrm{max}}^M$, we get
\begin{align*}
\u{1}{t}
&= \E \left[ \left. \u{1}{t} \right\vert \mathcal{F}_t \right] \\
&= \E \left[ \left. g'\left( \X{1}{1} \right) -  \int_t^1 \u{1}{r} \frac{Z_r^2}{\sigma_r^2 } \left( \mu_{a,r} - 2 \mu_r \frac{\sigma_{a,r}}{\sigma_r} + 2 \u{2}{r} \frac{\sigma_{a,r}}{\sigma_r } \right) \d r \right\vert \mathcal{F}_t \right] \\
&\leq \E \left[ \left. g'\left( \X{1}{1} \right) - \int_t^1 \u{1}{r} \frac{Z_r^2}{\sigma_r^2} \left( \mu_{a,r} - 2\frac{\sigma_{a,r}}{\sigma_r} \mu_r \right) \d r \right\vert \mathcal{F}_t \right]
\end{align*}
for all $t \in I^M_{\mathrm{max}}$ and $(\x{1}, \x{2}, \x{3})\in \IR^3$, where we use that $\sigma_{a,r} \cdot \u{2}{r} \geq 0$. Next we use the inequality
\begin{equation*}
-\frac{ \sigma_s \mu_{a,s} - 2 \sigma_{a,s} \mu_r }{\sigma_s^3} \leq \max \left\{0, -\inf_{ (\theta,x) \in \IR_+ \times \IR } \left( \frac{ \sigma \cdot \diff{a}{\mu} - 2 \diff{a}{\sigma} \cdot \mu }{\sigma^3} \right)(\theta,x) \right\} =: \beta
\end{equation*}
and the estimate from Lemma \ref{lemma:ZequalsU1} for $Z$ to obtain
\begin{equation*}
\u{1}{t} \leq \Vert g' \Vert_\infty + \beta \int_t^1 \sup_{x \in \IR^3} \diff{x_1}{u} (r,x) \sup_{\theta \in [r,1]} \sup_{x \in \IR^3} \left(\diff{x_1}{u}\right)^2 (\theta, x) \d r.
\end{equation*}
Thus we can derive the inequality
\begin{align*}
\sup_{\rho \in [t,1]} \sup_{x \in \IR^3} \diff{x_1}{u}(\rho, x )
& \leq \Vert g' \Vert_\infty + \beta \sup_{\rho \in [t,1]} \left\{ \int_\rho^1 \sup_{x \in \IR^3} \diff{x_1}{u} (r,x) \sup_{\theta \in [r,1]} \sup_{x \in \IR^3} \left(\diff{x_1}{u}\right)^2 (\theta, x) \d r \right\} \\
& \leq \Vert g' \Vert_\infty + \beta \int_t^1 \sup_{\theta \in [r,1]} \sup_{x \in \IR^3} \left(\diff{x_1}{u}\right)^3 (\theta, x) \d r.
\end{align*}
Note that $\inf_{(\theta,x)\in \IR_+ \times \IR} \frac{ \sigma \cdot \diff{a}{\mu} - 2 \diff{a}{\sigma} \cdot \mu }{\sigma^3}(\theta, x) > - \frac{1}{ 2 \Vert g' \Vert_\infty^2}$ implies $\beta < \frac{1}{2 \Vert g' \Vert_\infty^2}$. Hence, we obtain by setting $f(t)= \sup_{\rho \in [t,1]} \sup_{x \in \IR^3} \diff{x_1}{u}(\rho, x )$ and applying Lemma \ref{lemma:genaralGronwall} that
\begin{equation*}
\sup_{\rho \in [t,1]} \sup_{x \in \IR^3} \diff{x_1}{u}(\rho, x )
\leq \left( \frac{1}{\Vert g' \Vert_\infty ^2} - 2 \beta (1-t) \right)^{-\frac{1}{2}}
\end{equation*}
and thus,
\begin{equation*}
\Vert \u{1}{} \Vert_{\infty,t} \leq \Vert \diff{x_1}{u} \Vert_{\infty,t} \leq \left( \frac{1}{\Vert g' \Vert_\infty ^2} - 2 \beta \right)^{-\frac{1}{2}}< \infty.
\end{equation*}
\end{proof}

\begin{theorem}
\label{thm:U1isBounded}
Let $g$, $\mu$ and $\sigma$ fulfill Assumption \ref{asump:conditionsU1bounded}. Then, for FBSDE \eqref{fbsde:base}, we have $I^M_{\mathrm{max}} = [0,1]$ and there exists a unique, strongly regular Markovian decoupling field $u$ on the whole interval $[0,1]$. This $u$ is a continuous function on $[0,1]\times\IR^3$.

Furthermore let $(\X{1}{}, \X{2}{}, \X{3}{}, Y, Z)$ be the solution of FBSDE \eqref{fbsde:base} with an arbitrary initial condition $x \in \IR^3$ and $\u{1}{}, \u{2}{}, \u{3}{}$ be the associated gradient processes on $[0,1]$.
Then we have $\u{3}{} \equiv -1$ and the finite estimates
\begin{equation}
\label{est:u1}
0 \leq \u{1}{} \leq \left( \frac{1}{\Vert g' \Vert_\infty ^2} + 2 \min \left\{0, \inf_{ (\theta,x) \in \IR_+ \times \IR } \left( \frac{ \sigma \cdot \diff{a}{\mu} - 2 \diff{a}{\sigma} \cdot \mu }{\sigma^3} \right) (\theta,x) \right\} \right)^{-\frac{1}{2}},
\end{equation}
\begin{align}
\label{est:u2}
\left\Vert \u{2}{} \right\Vert_\infty 
& \leq \exp \left[ \Vert Z \Vert_\infty^2 \left( \left\Vert \frac{\diff{a}{\mu}}{\sigma^2} \right\Vert_\infty + 2 \left( \left\Vert \frac{\diff{a}{\sigma}}{\sigma} \right\Vert_\infty \left\Vert \frac{\mu}{\sigma^2} \right\Vert_\infty + \frac{1}{\epsilon^2} \left\Vert \frac{\diff{t}{\sigma}}{\sigma} \right\Vert_\infty \right) \right) \right] \nonumber \\
& \hspace*{6cm} \cdot \Vert Z \Vert_\infty^2 \left( 2 \left\Vert \frac{\diff{t}{\sigma}}{\sigma} \right\Vert_\infty \left\Vert \frac{\mu}{\sigma^2} \right\Vert_\infty + \left\Vert \frac{\diff{t}{\mu}}{\sigma^2} \right\Vert_\infty \right)
\end{align}
and
\begin{equation}
\label{est:Z}
\left\Vert Z \right\Vert_\infty 
\leq \left\Vert \u{1}{} \right\Vert_\infty 
\leq \left( \frac{1}{\Vert g' \Vert_\infty ^2} + 2 \min \left\{0, \inf_{ (\theta,x) \in \IR_+ \times \IR } \left( \frac{ \sigma \cdot \diff{a}{\mu} - 2 \diff{a}{\sigma} \cdot \mu }{\sigma^3} \right) (\theta,x) \right\} \right)^{-\frac{1}{2}}.
\end{equation}
\end{theorem}

\begin{proof}
Using Lemma \ref{EXPlosionM} we only need to show that the weak derivative of $u$ with regard to the initial value $x \in \IR^3$ is bounded by some constant which is independent of the time interval $[t,1] \subset I_{\mathrm{max}}^M$ on which it is defined. Then it follows that $I^M_{\mathrm{max}} = [0,1]$ and hence $t$ can be chosen to equal $0$ and the estimates \eqref{est:u1}, \eqref{est:u2} and \eqref{est:Z} hold true for corresponding processes on the whole interval $[0,1]$.

For now fix $t \in I^M_{\mathrm{max}}$ and $x \in \IR^3$ and let $\u{1}{}, \u{2}{}, \u{3}{}$ be the associated gradient processes. 
Lemma \ref{lemma:u3=-1:u1Bounded} yields $\u{3}{} \equiv -1$. In order to derive Estimate \eqref{est:u1} we show that $\sigma_{a,s} \cdot \u{2}{s} \geq 0$ a.s.\ for all $s \in [t,1]$ which then allows us to apply Lemma \ref{lemma:u3=-1:u1Bounded} yielding the estimate. Consider the three cases $i)$, $ii)$ and $iii)$ of Assumption \ref{asump:conditionsU1bounded}: With $\diff{a}{\sigma} \equiv 0$ of case $i)$ this is obviously true. For the remaining two cases observe that
\begin{align*}
\u{2}{s} = \int_s^1 \frac{Z_r^2}{\sigma_r^2} \left[ \left( \u{2}{r} \right)^2 \left( -2 \frac{ \sigma_{a,r}}{\sigma_r} \right) + \u{2}{r} \left( - \mu_{a,r} + 2 \frac{\sigma_{a,r}}{\sigma_r} \mu_r - 2 \frac{\sigma_{t,r}}{\sigma_r}  \right) + \left( 2 \frac{\sigma_{t,r}}{\sigma_r} \mu_r - \mu_{t,r} \right) \right] & \d r \\
- \int_s^1 \Z{2}{r} & \d \W{r}.
\end{align*}
Because $\u{2}{r}$ is bounded on every interval $[t,1] \subset I^M_{\mathrm{max}}$, we can view $\u{2}{}$ as fulfilling a Lipschitz BSDE. This allows us to use the comparison theorem by changing $2 \frac{\sigma_{t,r}}{\sigma_r} \mu_r - \mu_{t,r}$ to zero and hence compare with the trivial solution which is constantly $0$. Thus in the case $ii)$ we have $\u{2}{} \geq 0$ and in case $iii)$ $\u{2}{} \leq 0$. Therefore, we have $\diff{a}{\sigma} \cdot \u{2}{} \geq 0$ for the cases $ii)$ and $iii)$ as well. Hence we can apply Lemma \ref{lemma:u3=-1:u1Bounded} to obtain, for $s \in [t,1]$, 
\begin{equation*}
0 \leq \u{1}{s} \leq \left( \frac{1}{\Vert g' \Vert_\infty ^2} + 2 \min \left\{0, \inf_{ (\theta,x) \in \IR_+ \times \IR } \left( \frac{ \sigma \cdot \diff{a}{\mu} - 2 \diff{a}{\sigma} \cdot \mu }{\sigma^3} \right) (\theta,x) \right\} \right)^{-\frac{1}{2}}.
\end{equation*}
In addition with Lemma \ref{lemma:ZequalsU1} this yields
\begin{equation*}
\Vert Z \Vert_{\infty,t} \leq \Vert \u{1}{} \Vert_{\infty,t} \leq \left( \frac{1}{\Vert g' \Vert_\infty ^2} + 2 \min \left\{0, \inf_{ (\theta,x) \in \IR_+ \times \IR } \left( \frac{ \sigma \cdot \diff{a}{\mu} - 2 \diff{a}{\sigma} \cdot \mu }{\sigma^3} \right) (\theta,x) \right\} \right)^{-\frac{1}{2}} < \infty.
\end{equation*}
Since, as stated before, in case $ii)$ we have $\u{2}{} \geq 0$ and $\diff{a}{\sigma} \geq 0$ and in case $iii)$ $\u{2}{} \leq 0$ and $\diff{a}{\sigma} \leq 0$, we again can apply the comparison theorem to see that in case $ii)$ we have $0 \leq \u{2}{} \leq \bar{u}$ and in case $iii)$ $\bar{u} \leq \u{2}{} \leq 0$, where $\bar{u}$ is the solution of the linear BSDE
\begin{equation*}
\bar{u}_s = \int_s^1 \bar{u}_r \frac{Z_r^2}{\sigma_r^2} \left( - \mu_{a,r} + 2 \frac{\sigma_{a,r}}{\sigma_r} \mu_r - 2 \frac{\sigma_{t,r}}{\sigma_r}  \right) + \frac{Z_r^2}{\sigma_r^2} \left( 2 \frac{\sigma_{t,r}}{\sigma_r} \mu_r - \mu_{t,r} \right) \d r - \int_s^1 \bar{Z}_r \d \W{r}.
\end{equation*}
In case $i)$ we have that $\u{2}{} = \bar{u}$ giving that $\u{2}{}$ is bounded by $\bar{u}$ as well.

By estimating
\begin{align*}
\left\vert \bar{u}_s \right\vert 
=& \left\vert \E \left[ \left. \int_s^1 \exp \left( \int_s ^r \frac{Z_r^2}{\sigma_r^2} \left( - \mu_{a,r} + 2 \frac{\sigma_{a,r}}{\sigma_r} \mu_r - 2 \frac{\sigma_{t,r}}{\sigma_r}  \right) \d \rho \right) \frac{Z_r^2}{\sigma_r^2} \left( 2 \frac{\sigma_{t,r}}{\sigma_r} \mu_r - \mu_{t,r} \right) \d r \right\vert \mathcal{F}_s \right] \right\vert \\
\leq & \exp \left[ \Vert Z \Vert_\infty^2 \left( \left\Vert \frac{\diff{a}{\mu}}{\sigma^2} \right\Vert_\infty + 2 \left( \left\Vert \frac{\diff{a}{\sigma}}{\sigma} \right\Vert_\infty \left\Vert \frac{\mu}{\sigma^2} \right\Vert_\infty + \frac{1}{\epsilon^2} \left\Vert \frac{\diff{t}{\sigma}}{\sigma} \right\Vert_\infty \right) \right) \right] \\
& \hspace*{7cm} \cdot \Vert Z \Vert_\infty^2 \left( 2 \left\Vert \frac{\diff{t}{\sigma}}{\sigma} \right\Vert_\infty \left\Vert \frac{\mu}{\sigma^2} \right\Vert_\infty + \left\Vert \frac{\diff{t}{\mu}}{\sigma^2} \right\Vert_\infty \right)
\end{align*}
we have found a finite bound for $\u{2}{}$ that is independent of $t$.

Thus $\u{1}{}$, $\u{2}{}$ and $\u{3}{}$ are bounded independently of $t$. Hence there exists a solution on the whole interval $[0,1] = I_{\mathrm{max}}^M$. Therefore, we also have that all bounds are valid on this interval.
\end{proof}

\section{Weak solution}
\label{sec:weakSolution} 

In this section we show that a weak solution of the SEP can be obtained from the solution of the FBSDE \eqref{fbsde:base}. Recall that if Assumption \ref{asump:conditionsU1bounded} is fulfilled, then by Theorem \ref{thm:U1isBounded} FBSDE \eqref{fbsde:base} has a solution on the whole interval $[0,1]$ and the gradient processes are bounded.

In the following we sometimes use the fact that for two It\^o processes $A$ and $B$ and a time change $\gamma$, in the sense of Definition 1.2 in Chapter V, \cite{RY13}, it holds that
\begin{equation*}
\int_0^{\gamma(t)} A_r \d B_r = \int_0^t A_{\gamma(r)} \d B_{\gamma(r)}
\end{equation*}
(see e.g.\ Proposition 1.4, Chapter V, \cite{RY13}).  

The next theorem is a version of Theorem \ref{thm:mainResultWeak} with an explicit weak solution of the \SEP. 
\begin{theorem}
\label{thm:weakSolution}
Let $g$, $\mu$ and $\sigma$ fulfill Assumption \ref{asump:conditionsU1bounded}. Furthermore let $(\X{1}{}, \X{2}{}, \X{3}{}, Y, Z)$ be the solution of the FBSDE \eqref{fbsde:base} with initial value $(\X{1}{0}, \X{2}{0}, \X{3}{0}) = (0,0,0)$. Define
the random time
\begin{equation*}
\tilde{\tau} := \X{2}{1},
\end{equation*}
the time change
\begin{equation*}
\gamma (t) := \begin{cases}
\inf \left\{ s \geq 0 \vert \X{2}{s} > t \right\} & \text{ if } 0 \leq t < \tilde{\tau}, \\
1 & \text{ if } t \geq \tilde{\tau},
\end{cases}
\end{equation*}
the filtration $\cG_t := \cF_{\gamma(t)}$ and the process $A_t := Y_{\gamma(t)} + \X{3}{\gamma(t)}$ on $[0, \tilde{\tau}]$.

Then $\tilde{\tau}$ is a $(\cG_t)$-stopping time satisfying 
\begin{equation*}
\tilde{\tau} \leq \epsilon^{-2} \left( \frac{1}{\Vert g' \Vert_\infty ^2} + 2 \min \left\{0, \inf_{ (\theta,x) \in \IR_+ \times \IR } \left( \frac{ \sigma \cdot \diff{a}{\mu} - 2 \diff{a}{\sigma} \cdot \mu }{\sigma^3} \right) (\theta,x) \right\} \right)^{-1} \text{ a.s.}
\end{equation*}
Furthermore, on $[0, \tilde{\tau}]$, the process $B_t := \int_0^{t} \frac{1}{\sigma (\X{2}{\gamma(r)}, Y_{\gamma(r)} + \X{3}{\gamma(r)} )} \d Y_{\gamma (r)}$ is a $(\cG_t)$-Brownian motion, $A$ fulfills the SDE
\begin{equation*}
A_t = Y_0 + \int_0^t \mu \left( r, A_r \right) \d r + \int_0^t \sigma(r, A_r) \d B_r
\end{equation*}
and we have 
\begin{equation*}
A_{\tilde{\tau}} \sim \nu.
\end{equation*}
\end{theorem}

\begin{proof}
By standard results it follows that $\tilde{\tau}$ is a $(\cG_t)$-stopping time (see e.g.\ Proposition 1.1, Chapter V, \cite{RY13}).
With \begin{equation}\label{defi gamma inv}
\gamma^{-1} (s) := \X{2}{s}
\end{equation} 
for all $s \in [0,1]$ we have for all $t \in [0, \tilde{\tau}]$ that $\X{2}{\gamma(t)} = \gamma^{-1}(\gamma(t))=t $. Therefore, and because $\d Y_r = Z_r \d W_r$, we obtain
\begin{equation*}
\qvar{B}{t}= \int_0^{\gamma (t)} \frac{Z_r^2}{\sigma^2( \X{2}{r}, Y_r + \X{3}{r}  )} \d r = \gamma^{-1}(\gamma(t)) = t.
\end{equation*}
By Levy's characterisation of Brownian motion we get that $(B_t)$ is a $(\cG_t)$-Brownian motion on $[0, \tilde{\tau}]$.

Note that for all $\omega \in \Omega$ the function $\gamma$ is $\lambda$-a.e.\ differentiable on $[0, \tilde{\tau}]$ with
\begin{equation}
\label{eq:derivativeGamma}
\gamma' (t)
= ((\gamma^{-1})^{-1})'(t)
= \frac{1}{(\gamma^{-1})'( \gamma (t) )}
= \frac{\sigma^2 ( \X{2}{\gamma(t)}, Y_{\gamma(t)} + \X{3}{\gamma(t)} ) }{Z^2_{\gamma (t)}}
\end{equation}
and hence
\begin{align*}
A_{t}
&= \X{3}{\gamma(t)} + Y_{\gamma(t)} - Y_0 + Y_0 \\
&= Y_0 + \int_0 ^{\gamma (t) } \mu \left( \X{2}{r}, Y_r + \X{3}{r} \right) \frac{Z_r^2}{\sigma^2 (\X{2}{r}, Y_r + \X{3}{r} )} \d r + \int_0^{\gamma(t)} \frac{\sigma (\X{2}{r}, Y_r + \X{3}{r} )}{\sigma (\X{2}{r}, Y_r + \X{3}{r} )} \d Y_r \\
&= Y_0 + \int_0 ^t \mu \left( \X{2}{\gamma(r)}, Y_{\gamma (r) } + \X{3}{\gamma (r) } \right) \d r + \int_0^{t} \sigma \left( \X{2}{\gamma(r)}, Y_{\gamma(r)} + \X{3}{\gamma(r)} \right) \d B_r  \\
&= Y_0 + \int_0 ^t \mu \left( r, A_r \right) \d r + \int_0^t \sigma(r, A_r) \d B_r
\end{align*}
for all $t \in [0, \tilde{\tau}]$.
Also
\begin{equation*}
A_{\tilde{\tau}} = Y_{\gamma (\tilde{\tau})} + \X{3}{\gamma(\tilde{\tau})} = Y_1 + \X{3}{1} = g(W_1) \sim \nu.
\end{equation*}
The bound for $\tilde{\tau}$ follows with the bound for $\Vert Z \Vert_\infty$ stated in Theorem \ref{thm:U1isBounded} and by $\sigma \geq \epsilon$.
\end{proof}
The next lemma characterizes the stopping time $\tilde \tau = \gamma^{-1}(1)$ of Theorem \ref{thm:weakSolution} in terms of the solution of an FBSDE driven by the Brownian motion $B$. We use the lemma later to show existence of strong solutions of the \SEP.

\begin{lemma}
\label{lemma:localLipschitzSystem}
Assume $g$, $\mu$ and $\sigma$ to fulfill Assumption \ref{asump:conditionsU1bounded}. Let the decoupling field $u$ of the FBSDE \eqref{fbsde:base} have a continuous weak derivative $\diff{x_1}{u} > 0$. Also let $(\X{1}{}, \X{2}{}, \X{3}{}, Y, Z)$, $\gamma$ and $B$ be defined as in Theorem \ref{thm:weakSolution}. Moreover, let $\hat{B}$ be any Brownian motion coinciding with $B$ on $[0, \X{2}{1}]$. Then $\gamma$, $W$, $\X{3}{}$ and $Y$ solve the system
\begin{align}
\gamma (t) &= \int_0^t \frac{\sigma^2 \left(r , Y_{\gamma(r)} + \X{3}{\gamma(r)} \right)}{\left(\diff{x_1}{u}\right)^2 (\gamma(r), W_{\gamma(r)}, r, \X{3}{\gamma(r)})} \d r \nonumber \\
W_{\gamma(t)} &= \int_0^t \frac{\sigma \left(r , Y_{\gamma(r)} + \X{3}{\gamma(r)} \right)}{\diff{x_1}{u}(\gamma(r), W_{\gamma(r)}, r, \X{3}{\gamma(r)})} \d \hat{B}_r \label{sys:localLipschitzProblem} \\
\X{3}{\gamma(t)} &= \int_0^t \mu \left(r , Y_{\gamma(r)} + \X{3}{\gamma(r)} \right) \d r \nonumber \\
Y_{\gamma(t)} &= Y_0 + \int_0^t \sigma\left(r , Y_{\gamma(r)} + \X{3}{\gamma(r)} \right) \d \hat{B}_r \nonumber
\end{align}
for all $t \geq 0$ such that $\gamma(t) \leq 1$. Additionally, for $\gamma^{-1}$ defined as in \eqref{defi gamma inv} we have 
\begin{equation}\label{bound gamma inv}
\gamma^{-1}(1) \leq \frac{\Vert \diff{x_1}{u} \Vert_\infty^2}{\epsilon^2} \leq \epsilon^{-2} \left( \frac{1}{\Vert g' \Vert_\infty ^2} + 2 \min \left\{0, \inf_{ (\theta,x) \in \IR_+ \times \IR } \left( \frac{ \sigma \cdot \diff{a}{\mu} - 2 \diff{a}{\sigma} \cdot \mu }{\sigma^3} \right) (\theta,x) \right\} \right)^{-1}.
\end{equation}
\end{lemma}

\begin{proof}
Note that Theorem \ref{thm:U1isBounded} implies the bound \eqref{bound gamma inv}. Since $\diff{x_1}{u}$ is continuous we get with Lemma \ref{lemma:ZequalsU1} that $Z_s=\diff{x_1}{u}(s, \X{1}{s}, \X{2}{s}, \X{3}{s})>0$ for all $s \in [0,1]$ and hence both $\gamma$ and $\gamma^{-1}$ are strict monotone increasing and continuous. Moreover, Lemma \ref{lemma:ZequalsU1}, Equation \eqref{eq:derivativeGamma} and the fact that $\X{2}{\gamma(t)} = t$ yield
\begin{equation*}
\gamma' (t)
= \frac{\sigma^2 \left(\X{2}{\gamma(t)}, Y_{\gamma(t)} + \X{3}{\gamma(t)} \right)}{Z^2_{\gamma (t)}}
= \frac{\sigma^2 \left(t, Y_{\gamma(t)} + \X{3}{\gamma(t)} \right)}{\left(\diff{x_1}{u}\right)^2 \left( \gamma (t), \X{1}{\gamma (t)}, \X{2}{\gamma (t)}, \X{3}{\gamma (t)} \right)}
\end{equation*}
for all $0 \leq t \leq \gamma^{-1} (1)$.

Furthermore, $\X{1}{s} = W_s$ yields that
\begin{align*}
W_{\gamma(t)} 
= \int_0^{t} 1 \d W_{\gamma(r)} 
&= \int_0^{t} \frac{\sigma\left( \X{2}{\gamma(r)}, Y_{\gamma(r)} + \X{3}{\gamma(r)} \right)}{ Z_{\gamma(r)}} \d \hat{B}_r \\
&= \int_0^t \frac{\sigma\left( r, Y_{\gamma(r)} + \X{3}{\gamma(r)} \right)}{\diff{x_1}{u} \left( \gamma (r), \X{1}{\gamma (r)}, \X{2}{\gamma (r)}, \X{3}{\gamma (r)} \right)} \d \hat{B}_r \\
&= \int_0^t \frac{\sigma\left( r, Y_{\gamma(r)} + \X{3}{\gamma(r)} \right)}{\diff{x_1}{u} \left( \gamma (r), W_{\gamma(r)}, r, \X{3}{\gamma (r)} \right)} \d \hat{B}_r.
\end{align*}
Also
\begin{align*}
Y_{\gamma(t)}
= Y_0 + \int_0^{\gamma(t)} Z_r \d W_r 
&= Y_0 + \int_0^t Z_{\gamma(r)} \d W_{\gamma(r)} \\
&= Y_0 + \int_0^t \sigma \left( r, Y_{\gamma(r)} + \X{3}{\gamma(r)} \right) \d \hat{B}_r,
\end{align*}
\begin{align*}
\X{3}{\gamma(t)}
&=\int_0^{\gamma(t)} \mu \left( \X{2}{r}, Y_r + \X{3}{r} \right) \frac{Z_r^2}{\sigma^2(\X{2}{r}, Y_r + \X{3}{r})} \d r \\
&= \int_0^t \mu \left( \X{2}{\gamma(r)}, Y_{\gamma(r)} + \X{3}{\gamma(r)} \right) \frac{Z_{\gamma(r)}^2}{\sigma^2\left( \X{2}{\gamma(r)}, Y_{\gamma(r)} + \X{3}{\gamma(r)} \right)} \frac{\sigma^2\left( \X{2}{\gamma(r)}, Y_{\gamma(r)} + \X{3}{\gamma(r)} \right)}{Z_{\gamma(r)}^2} \d r \\
&= \int_0^t \mu \left( r, Y_{\gamma(r)} + \X{3}{\gamma(r)} \right) \d r
\end{align*}
and
\begin{align*}
\gamma(t)
= \int_0^t \gamma'(r) \d r 
&= \int_0^t \frac{\sigma^2\left( r, Y_{\gamma(r)} + \X{3}{\gamma(r)} \right)}{\left(\diff{x_1}{u}\right)^2 \left( \gamma (r), \X{1}{\gamma (r)}, \X{2}{\gamma (r)}, \X{3}{\gamma (r)} \right)} \d r \\
&= \int_0^t \frac{\sigma^2\left( r, Y_{\gamma(r)} + \X{3}{\gamma(r)} \right)}{\left(\diff{x_1}{u}\right)^2 \left( \gamma (r), W_{\gamma (r)}, r, \X{3}{\gamma (r)} \right)} \d r
\end{align*}
for all $t \in [0, \gamma^{-1}(1)]$.
\end{proof}

\section{Strong solution}
\label{sec:strongSolution}

We use the definitions and constructions of the former chapters. In particular let $u$ be the unique strongly regular decoupling field of the FBSDE \eqref{fbsde:base} which exists on the whole interval $[0,1]$ if Assumption \ref{asump:conditionsU1bounded} is fulfilled.

\begin{theorem}
\label{thm:stoppingB:strongSolution}
Let $g$, $\mu$ and $\sigma$ fulfill Assumption \ref{asump:conditionsU1bounded} and $\mu$, $\sigma$ and their derivatives be bounded. Denote by $u$ the decoupling field of FBSDE \eqref{fbsde:base} and assume the partial derivative $\diff{x_1}{u}$ with respect to the first space variable to be Lipschitz continuous in every argument and $\diff{x_1}{u} \geq \delta > 0$. Let $B$ be an arbitrary Brownian motion and denote by $( \mathcal{F}^B ) = ( \mathcal{F}^B_s)_{s \in [0, \infty)}$ the augmented filtration generated by $B$. Then there exists a bounded stopping time $\tau$ with respect to the filtration $\mathcal{F}^B$ such that for the process $A$ given by
\begin{equation*}
A_t = Y_0 + \int_0^t \mu \left( r, A_r \right) \d r + \int_0^t \sigma(r, A_r) \d B_r,
\end{equation*}
for all $t \in [0, \tau ]$, we have that $A_{\tau} \sim \nu$ and the stopping time $\tau$ satisfies
\begin{equation*}
\tau \leq \epsilon^{-2} \left( \frac{1}{\Vert g' \Vert_\infty ^2} + 2 \min \left\{0, \inf_{ (\theta,x) \in \IR_+ \times \IR } \left( \frac{ \sigma \cdot \diff{a}{\mu} - 2 \diff{a}{\sigma} \cdot \mu }{\sigma^3} \right) (\theta,x) \right\} \right)^{-1} \ \text{a.s.}
\end{equation*}
By solving the  Lipschitz SDE
\begin{align}
\gamma(r) &= \int_0 ^r \frac{\sigma^2 ( s, \Theta_s + \Delta_s ) }{\left( \diff{x_1}{u} ( \gamma(s), \Gamma_s, s, \Delta_s) \right)^2} \d s \nonumber\\
\Gamma_r &= \int_0^r \frac{\sigma ( s, \Theta_s + \Delta_s ) }{ \diff{x_1}{u} ( \gamma(s), \Gamma_s, s, \Delta_s) } \d B_s \label{sys:systemTau} \\
\Delta_r &= \int_0^r \mu ( s, \Theta_s + \Delta_s ) \d s \nonumber \\
\Theta_r &= Y_0 + \int_0^r \sigma( s, \Theta_s + \Delta_s ) \d B_s \nonumber
\end{align}
for all $r \geq 0$ such that $\gamma(r) \leq 1$ and where $Y_0$ is the starting value of the process $Y$ in the FBSDE \eqref{fbsde:base} and setting $\tau := \inf \{ r \geq 0 \vert \gamma(r) = 1 \}$ we can obtain such a stopping time. 
\end{theorem}

\begin{proof}
Since any solution of FBSDE \eqref{fbsde:base} has a unique distribution independent of the driving Brownian motion, we know that the constant $Y_0$ is always the same and does not depend on the driving Brownian motion.

Let us take a look at the system \eqref{sys:systemTau}. Note that for all $a,b \in [0,1] \times \IR^3$ 
\begin{equation*}
\left\vert \frac{1}{\diff{x_1}{u} (a)} - \frac{1}{\diff{x_1}{u}(b)} \right\vert 
= \left\vert \frac{ \diff{x_1}{u}(b) - \diff{x_1}{u} (a) }{ \diff{x_1}{u}(a) \cdot \diff{x_1}{u}(b) } \right\vert
\leq \frac{ L_{u,x_1} }{ \delta ^2 } \vert b - a \vert,
\end{equation*}
yielding that $(\diff{x_1}{u})^{-1}$ is Lipschitz continuous.
Since hence both $(\diff{x_1}{u})^{-1}$ and $\sigma$ are Lipschitz continuous and bounded we get that $\sigma \cdot (\diff{x_1}{u})^{-1}$ and $\sigma^2 \cdot (\diff{x_1}{u})^{-2}$ are Lipschitz and bounded as well. Thus, we have that all coefficients of the system \eqref{sys:systemTau} are Lipschitz continuous. 
Therefore there exists a unique solution $(\gamma, \Gamma, \Delta, \Theta)$ of \eqref{sys:systemTau} which is progressively measurable w.r.t.\ $(\cF^B_t)$. Hence $\tau := \inf \{ r \geq 0 \vert \gamma(r) = 1 \}$ is a stopping time w.r.t.\ $(\mathcal{F}^{B}_t)$ because $\gamma$ is continuous.

Furthermore, the systems \eqref{sys:localLipschitzProblem} and \eqref{sys:systemTau} just differ by notation and the driving Brownian motion. By the principle of causality (see \cite{karatzas1991brownian}) the distributions of $(\gamma, W_{\gamma}, \X{3}{\gamma}, Y_{\gamma})$ from Lemma \ref{lemma:localLipschitzSystem} and $(\gamma, \Gamma, \Delta, \Theta)$ are the same. Hence, we immediately have the bound for $\tau$ as stated in Lemma \ref{lemma:localLipschitzSystem} and also for $A_t:= \Delta_t + \Theta_t$ that
\begin{equation*}
A_\tau = \Delta_\tau + \Theta_\tau = \Delta_{\gamma^{-1}(1)} + \Theta_{\gamma^{-1}(1)} \sim \X{3}{\gamma(\gamma^{-1}(1))} + Y_{\gamma(\gamma^{-1}(1))} = \X{3}{1} + Y_1 = g(W_1) \sim \nu
\end{equation*}
and
\begin{align*}
A_t
= \Delta_t + \Theta_t 
&= Y_0 + \int_0^t \mu (s, \Delta_s + \Theta_s) \d s + \int_0^t \sigma(s, \Delta_s + \Theta_s) \d B_s \\
&= Y_0 + \int_0^t \mu (s, A_s) \d s + \int_0^t \sigma(s, A_s) \d B_s.
\end{align*}
\end{proof}

What remains to do is to find sufficient conditions for the assumptions of Theorem \ref{thm:stoppingB:strongSolution} to hold true. 
For this we use that the decoupling field $u$ of FBSDE \eqref{fbsde:base} is three times weakly differentiable. To show this we extend FBSDE \eqref{fbsde:base} by the dynamics of the gradient processes and view this system as a extended FBSDE, for which we can show the weak differentiability of its decoupling field.

Let $a := \max \big( \Vert \diff{x_1}{u} \Vert_\infty, \Vert \diff{x_2}{u} \Vert_\infty, \Vert \diff{x_3}{u} \Vert_\infty \big)$ and define the truncation operator $T:\IR \to \IR$ by $T(z):= \min(\max(z,-a), a)$. Note that the map $T$ is uniformly Lipschitz.
Assume that $g$, $\mu$, $\sigma$ and their first derivatives are Lipschitz continuous and consider the FBSDE
\begin{align}\label{align:fbsde:secondOrder}
\X{1}{s} &= \x{1} + \int_t^s 1 \d W_r, \nonumber \\
\X{2}{s} &= \x{2} + \int _t^s \frac{ ( \Zo{0}{r} )^2}{ \sigma_r^2 } \d r, \nonumber \\
\X{3}{s} &= \x{3} + \int_t^s \mu_r \frac{( \Zo{0}{r} )^2}{ \sigma_r^2 } \d r, \nonumber \\
\Y{0}{s} &= g(\X{1}{1}) - \X{3}{1} - \int_s^1 \Zo{0}{r} \d W_r, \nonumber \\
\Y{1}{s} &= g'(\X{1}{1}) + \int_s^1 T \left( \Y{1}{r} \right) \frac{( \Zo{0}{r} )^2}{\sigma_r^2 } \left( T \left( \Y{3}{r} \right) \left( \mu_{a,r} - 2 \mu_r \frac{\sigma_{a,r}}{\sigma_r} \right) - 2 T \left(  \Y{2}{r} \right) \frac{\sigma_{a,r}}{\sigma_r }\right) \d r \nonumber \\
& \hspace*{5cm} + \int_s^1 2 \frac{ \Zo{0}{r} }{ \sigma_r^2} T \left( \left(\Y{2}{r} \right) + T \left( \Y{3}{r} \right) \mu_r \right) \Zo{1}{r} \d r - \int_s^1 \Zo{1}{r} \d W_r \nonumber \\
\Y{2}{s} &= 0 + \int_s^1 -2 \frac{( \Zo{0}{r} )^2}{\sigma_r^2} \left( \frac{\sigma_{t,r}}{\sigma_r} + T \left( \Y{2}{r} \right) \frac{\sigma_{a,r}}{\sigma_r} \right) \left( T \left( \Y{2}{r} \right) + T \left( \Y{3}{r} \right) \mu_r \right) \d r \nonumber \\
& \hspace*{5cm} + \int_s^1 T \left( \Y{3}{r} \right) \frac{( \Zo{0}{r} )^2}{\sigma_r^2 } \left( T \left( \Y{2}{r} \right) \mu_{a,r} + \mu_{t,r} \right) \d r \nonumber \\
& \hspace*{5cm} + \int_s^1 2 \frac{ \Zo{0}{r} }{ \sigma_r^2} \left( T \left( \Y{2}{r} \right) + T \left( \Y{3}{r} \right) \mu_r \right) \Zo{2}{r} \d r - \int_s^1 \Zo{2}{r} \d W_r \nonumber \\
\Y{3}{r} &= -1 + \int_s^1 \left( T \left( \Y{3}{r} \right) + 1 \right) \frac{( \Zo{0}{r} )^2}{\sigma^2_r} \left( T \left( \Y{3}{r} \right) \mu_{a,s} -2 \frac{\sigma_{a,r}}{\sigma_r} \left( T \left( \Y{2}{r} \right) + T \left( \Y{3}{r} \right) \mu_r \right) \right) \d r \nonumber \\
& \hspace*{5cm} + \int_s^1 2 \frac{ \Zo{0}{r} }{ \sigma_r^2} \left( T \left( \Y{2}{r} \right) + T \left( \Y{3}{r} \right) \mu_r \right) \Zo{3}{r} \d r - \int_s^1 \Zo{3}{r} \d W_r \nonumber \\
\end{align}
with the decoupling condition 
\begin{align*}
\Y{0}{s} &= \u{0}{}(s, \X{1}{s}, \X{2}{s}, \X{3}{s} ), \\
\Y{1}{s} &= \u{1}{}(s, \X{1}{s}, \X{2}{s}, \X{3}{s} ),  \\
\Y{2}{s} &= \u{2}{}(s, \X{1}{s}, \X{2}{s}, \X{3}{s} ), \\
\Y{3}{s} &= \u{3}{}(s, \X{1}{s}, \X{2}{s}, \X{3}{s} ), 
\end{align*}
where
\begin{equation*}
\mu_r := \mu \left( \X{2}{r}, \Y{0}{r} + \X{3}{r} \right), \qquad  \qquad \sigma_r := \sigma \left( \X{2}{r}, \Y{0}{r} + \X{3}{r} \right) 
\end{equation*}
and 
\begin{align*}
\mu_{t,r} := \diff{t}{\mu} \left( \X{2}{r}, \Y{0}{r} + \X{3}{r} \right), \qquad & \qquad \mu_{a,r} := \diff{a}{\mu} \left( \X{2}{r}, \Y{0}{r} + \X{3}{r} \right), \\
\sigma_{t,r} := \diff{t}{\sigma} \left( \X{2}{r}, \Y{0}{r} + \X{3}{r} \right), \qquad & \qquad \sigma_{a,r} := \diff{a}{\sigma} \left( \X{2}{r}, \Y{0}{r} + \X{3}{r} \right).
\end{align*}

\begin{lemma}
\label{lemma:fbsdeSecondOrderBounded}
Let $g$, $\mu$ and $\sigma$ fulfill Assumption \ref{asump:conditionsU1bounded}. In addition, suppose that $g$, $\mu$ and $\sigma$ are twice differentiable and that the second derivatives are bounded. Then, for the FBSDE \eqref{align:fbsde:secondOrder}, we have $I^M_{\mathrm{max}} = [0,1]$ and there exists a unique, strongly regular Markovian decoupling field $(\u{0}{}, \u{1}{}, \u{2}{}, \u{3}{})$ on the whole interval $[0,1]$.
Furthermore,
\begin{equation*}
\u{0}{} = u, \quad \u{1}{} = \diff{x_1}{u}, \quad \u{2}{} = \diff{x_2}{u} \quad \text{and} \quad \u{3}{} = \diff{x_3}{u},
\end{equation*} 
a.e., where $u$ is the unique decoupling field to FBSDE \eqref{fbsde:base}. In particular, $u$ is twice weakly differentiable w.r.t.\ the initial value $x$ with uniformly bounded derivatives.
\end{lemma}

\begin{proof}
It is straightforward to verify that FBSDE \eqref{align:fbsde:secondOrder} satisfies (MLLC), and hence Theorem \ref{GLObalexistM} is applicable. Let $\u{i}{}$, $i = 0, 1, 2, 3$ be the corresponding unique weakly regular Markovian decoupling field on $I_{\mathrm{max}}^M$. $\u{i}{}$, $i = 0, 1, 2, 3$, are continuous functions on $I_{\mathrm{max}}^M \times \IR^3$.
In order to show that $I_{\mathrm{max}}^M = [0,1]$ we again need to prove that every partial derivative of $\u{i}{}$ for $i=0,1,2,3$ is bounded independently with regard to the interval $[t,1] \subset I_{\mathrm{max}}^M$ where we consider it.

Let $t \in I_{\mathrm{max}}^M$. For an arbitrary initial condition $\bar x \in \IR^3$ consider the corresponding processes
\begin{equation*}
\X{1}{}, \X{2}{}, \X{3}{}, \Y{0}{}, \Y{1}{}, \Y{2}{}, \Y{3}{}, \Zo{0}{}, \Zo{1}{}, \Zo{2}{}, \Zo{3}{}
\end{equation*}
on $[t,1]$. Note that $\X{1}{}, \X{2}{}, \X{3}{}, \Y{0}{}, \Zo{0}{}$ solve FBSDE \eqref{fbsde:base}, which implies that they coincide with the processes $\X{1}{}, \X{2}{}, \X{3}{}, Y, Z$ from \eqref{fbsde:base} since strong regularity of Markovian decoupling fields guarantees uniqueness. 
Now $\Y{0}{} = Y$ implies $u(t',x') = \u{0}{}(t',x')$ for all $t' \in [t,1]$, $x' \in \IR^3$.

Note that a truncation with $T$ does not effect any gradient process of FBSDE \eqref{fbsde:base}. Thus, $(\Y{1}{s})$, $(\Y{2}{s})$, $(\Y{3}{s})$ fulfill the same dynamics resp.\ BSDEs as the gradient processes $(\u{1}{s})$, $(\u{2}{s})$, $(\u{3}{s})$ in \eqref{bsdes:ux}. Therefore, we can apply the same arguments and conclude that they also satisfy the estimates \eqref{est:u1}, \eqref{est:u2} and \eqref{est:Z} (see Theorem \ref{thm:U1isBounded}).
In particular $\Y{3}{s} = -1 = \u{3}{s}$ for all $s \in [t,1]$ and therefore also $\Zo{3}{s} = 0 = \Z{3}{s}$. Hence,
\begin{align*}
\Y{2}{s} - \u{2}{s} 
= \ & \int_s^1 \left( \left( \Y{2}{r} \right)^2 - \left( \u{2}{r} \right)^2 \right) \frac{(\Zo{0}{r})^2}{\sigma_r^2} \left( - 2 \frac{\sigma_{a,r}}{\sigma_r} \right) \d r \\
& \quad + \int_s^1 \left( \Y{2}{r} - \u{2}{r} \right) \frac{(\Zo{0}{r})^2}{\sigma_r^2} \left( - \mu_{a,r} + 2 \frac{\sigma_{a,r}}{\sigma_r} \mu_r - 2 \frac{\sigma_{t,r}}{\sigma_r} \right) \d r \\
& \quad + \int_s^1 \frac{2 \Zo{0}{r}}{\sigma_r^2} \left( \left( \Y{2}{r} - \mu_r \right) \Zo{2}{r} - \left( \u{2}{r} - \mu_r \right) \Z{2}{r} \right) \d r - \int_s^1 \left( \Zo{2}{r} - \Z{2}{r} \right) \d W_r \\
= \ & \int_s^1 \left( \left( \Y{2}{r} \right)^2 - \left( \u{2}{r} \right)^2 \right) \frac{(\Zo{0}{r})^2}{\sigma_r^2} \left( - 2 \frac{\sigma_{a,r}}{\sigma_r} \right) \d r - \int_s^1 \left( \Zo{2}{r} - \Z{2}{r} \right) \d W_r  \\
& \quad + \int_s^1 \left( \Y{2}{r} - \u{2}{r} \right) \frac{(\Zo{0}{r})^2}{\sigma_r^2} \left( - \mu_{a,r} + 2 \frac{\sigma_{a,r}}{\sigma_r} \mu_r - 2 \frac{\sigma_{t,r}}{\sigma_r} \right) \d r \\
& \quad + \int_s^1 \frac{2 \Zo{0}{r}}{\sigma_r^2} \left( \left( \Y{2}{r} - \u{2}{r} \right) \Zo{2}{r} + \left( \u{2}{r} - \mu_r \right) \left( \Zo{2}{r} - \Z{2}{r} \right) \right) \d r 
\end{align*}
Since $\frac{2 \Zo{0}{r}}{\sigma_r^2} \left( \u{2}{r} - \mu_r \right)$ is bounded we have that $\widetilde{W}_s := W_s - W_t - \int_t^s \frac{2 \Zo{0}{r}}{\sigma_r^2} \left( \u{2}{r} - \mu_r \right) \d r$, $s \in [t,1]$, is a Brownian motion under some probability measure equivalent to $\IP$. Under the new measure the process pair $(\Y{2}{s} - \u{2}{s}, \Zo{2}{s} - \Z{2}{s})$ is a solution of the following linear BSDE with bounded coefficients
\begin{align*}
\hat Y_s
= \ & \int_s^1 \hat Y_r \left( \Y{2}{r} + \u{2}{r} \right) \frac{(\Zo{0}{r})^2}{\sigma_r^2} \left( - 2 \frac{\sigma_{a,r}}{\sigma_r} \right) \d r \\
& \quad + \int_s^1 \hat Y_r \frac{(\Zo{0}{r})^2}{\sigma_r^2} \left( - \mu_{a,r} + 2 \frac{\sigma_{a,r}}{\sigma_r} \mu_r - 2 \frac{\sigma_{t,r}}{\sigma_r} \right) \d r \\
& \quad  + \int_s^1 \hat Y_r \frac{2 \Zo{0}{r}}{\sigma_r^2} \Zo{2}{r} \d r - \int_s^1 \hat Z_r \d \widetilde{W}_r.
\end{align*}
Note that $(0,0)$ is the unique solution of the previous BSDE. Consequently, 
$\Y{2}{}$ and $\u{2}{}$ are indistinguishable and $\Zo{2}{} = \Z{2}{}$, $\lambda \otimes P$-almost everywhere on $[t,1] \times \Omega$. 

Similarly we can show that $\Y{1}{}$ and $\u{1}{}$ as well as $\Zo{1}{}$ and $\Z{1}{}$ coincide. Thus we have
\begin{align*}
& \diff{x_1}{u^{(0)}} (s, \X{1}{s}, \X{2}{s}, \X{3}{s} ) = \diff{x_1}{u} (s, \X{1}{s}, \X{2}{s}, \X{3}{s} ) = \u{1}{s} = \Y{1}{s}, \\
& \diff{x_2}{u^{(0)}} (s, \X{1}{s}, \X{2}{s}, \X{3}{s} ) = \diff{x_2}{u} (s, \X{1}{s}, \X{2}{s}, \X{3}{s} ) = \u{2}{s} = \Y{2}{s}, \\
& \diff{x_3}{u^{(0)}} (s, \X{1}{s}, \X{2}{s}, \X{3}{s} ) = \diff{x_3}{u} (s, \X{1}{s}, \X{2}{s}, \X{3}{s} ) = \u{3}{s} = \Y{3}{s}
\end{align*}
a.e.\ on $[t,1]$.

It remains to show that $I_{\mathrm{max}}^M = [0,1]$. Define for $x=(x_1, x_2, x_3)^T \in \IR^3$, $y = (y_0, y_1, y_2, y_3)^T \in \IR^4$, $z = (z_0, z_1, z_2, z_3)^T \in \IR^4$
\begin{equation*}
\bar{X}_s := \left( \begin{array}{c} \X{1}{s} \\ \X{2}{s} \\ \X{3}{s} \end{array} \right), \qquad
\bar{Y}_s := \left( \begin{array}{c} \Y{0}{s} \\ \Y{1}{s} \\ \Y{2}{s} \\ \Y{3}{s} \end{array} \right), \qquad
\bar{Z}_s := \left( \begin{array}{c} \Zo{0}{s} \\ \Zo{1}{s} \\ \Zo{2}{s} \\ \Zo{3}{s} \end{array} \right)
\end{equation*}
\begin{equation*}
\bar{M} \left( x , y, z \right) 
:= \left( \begin{array}{c} 0 \\ \frac{z_0^2}{\sigma^2 ( x_2, y_0 + x_3)} \\ \mu \left( x_2 , y_0 + x_3 \right) \frac{z_0^2}{\sigma^2(x_2 , y_0 + x_3)} \end{array} \right),
\qquad 
\bar{\Sigma} := \left( \begin{array}{c} 1 \\ 0 \\ 0 \end{array} \right), \qquad
\bar{\xi} \left( x \right) := \left( \hspace{-1mm} \begin{array}{c} g \left( x_1 \right) - x_3 \\ g' \left( x_1 \right) \\ 0 \\ -1 \end{array} \hspace{-1mm} \right)
\end{equation*}
and
\begin{align*}
\bar{F} \left( x , y, z \right) \hspace*{-0.5cm} & \\
:=& \left( \begin{array}{c} 
0 \\
y_1 \frac{(z_0)^2}{\sigma^2 (x_2, y_0 + x_3)} \left( \diff{a}{\mu} \left( x_2, y_0 + x_3 \right) - 2 \mu \left( x_2, y_0 + x_3 \right) \frac{\diff{a}{\sigma} \left( x_2, y_0 + x_3 \right)}{ \sigma \left( x_2, y_0 + x_3 \right) } + 2 y_2 \frac{\diff{a}{\sigma} \left( x_2, y_0 + x_3 \right)}{ \sigma \left( x_2, y_0 + x_3 \right) } \right) \\
2 \frac{(z_0)^2}{\sigma^2 \left( x_2, y_0 + x_3 \right) } \left( \frac{ \diff{t}{\sigma} \left( x_2, y_0 + x_3 \right) }{ \sigma \left( x_2, y_0 + x_3 \right) } + y_2 \frac{ \diff{a}{\sigma} \left( x_2, y_0 + x_3 \right) }{ \sigma \left( x_2, y_0 + x_3 \right) } \right) \left( y_2 - \mu \left( x_2, y_0 + x_3 \right) \right) \\
0
\end{array} \right) \\
& + \left( \begin{array}{c} 0 \\ 0 \\ \frac{(z_0)^2}{\sigma^2 \left( x_2, y_0 + x_3 \right)} \left( y_2 \diff{a}{\mu} \left( x_2, y_0 + x_3 \right) + \diff{t}{\mu} \left( x_2, y_0 + x_3 \right) \right) \\ 0 \end{array} \right) \\
& + \left( \begin{array}{c} 0 \\  - \frac{2 z_0}{\sigma^2 \left( x_2, y_0 + x_3 \right)} \left( y_2 - \mu \left( x_2, y_0 + x_3 \right) \right) z_1 \\  - \frac{2 z_0}{\sigma^2 \left( x_2, y_0 + x_3 \right)} \left( y_2 - \mu \left( x_2, y_0 + x_3 \right) \right) z_2 \\ 0 \end{array} \right).
\end{align*}
Then
\begin{equation*}
\bar{X}_s = \bar{x} + \int_t ^s \bar{M} \left( \bar{X}_r, \bar{Y}_r, \bar{Z}_r \right) \d r + \int_t^s \bar{\Sigma} \d W_r
\end{equation*}
and
\begin{equation*}
\bar{Y}_s = \bar{\xi} \left( \bar{X}_1 \right) - \int_s^1 \bar{F} \left( \bar{X}_r, \bar{Y}_r, \bar{Z}_r \right) \d r - \int_s^1 \bar{Z}_r \d W_r.
\end{equation*}
By setting
\begin{equation*}
\bar{U}_s := \partial_x \left( \begin{array}{c} \u{0}{} \\ \u{1}{} \\ \u{2}{} \\ \u{3}{} \end{array} \right) \left( s, \bar{X}_s \right)
= \left( \begin{array}{ccc} 
\u{1}{} & \u{2}{} & \u{3}{} \\ 
\diff{x_1}{\u{1}{}} & \diff{x_2}{\u{1}{}} & \diff{x_3}{\u{1}{}} \\ 
\diff{x_1}{\u{2}{}} & \diff{x_2}{\u{2}{}} & \diff{x_3}{\u{2}{}} \\ 
\diff{x_1}{\u{3}{}} & \diff{x_2}{\u{3}{}} & \diff{x_3}{\u{3}{}} \\
\end{array} \right) \left( s, \bar{X}_s \right)
\end{equation*}
we get
\begin{equation*}
\partial_x \bar{Y}_s = \bar{U}_s \cdot \partial_x \bar{X}_s.
\end{equation*}
Since $( \partial_x \bar{X}_s)^{-1}$ is a multidimensional It\^o process on $[t,1]$ (see Lemma \ref{lemma:dynamicsFirstDiff} and its proof) we get that $\bar{U}_s = \partial_x \bar{Y}_s \cdot ( \partial_x \bar{X}_s )^{-1}$ is also an It\^o process and hence there exist $(b_s)$ and $(\hat{Z}_s)$ such that
\begin{equation*}
\bar{U}_s = \bar{U}_1 - \int_s^1 b_r \d r - \int_s^1 \hat{Z}_r \d W_r.
\end{equation*}
For the following we also introduce for an It\^o process $I_s = I_0 - \int_0^s i_r \d r - \int_0^s j_r \d W_r$ the two operators $\Dt$ and $\Dw$ via $(\Dt I)_s := i_s$ and $(\Dw I)_s := j_s$. Using this notation we have
\begin{align*}
\partial_x \bar{Z}_s
&= \Dw \partial_x \bar{Y}_s \\
&= \Dw \left( \bar{U}_s \cdot \partial_x \bar{X}_s \right) \\
&= \bar{U}_s \cdot \Dw \partial_x \bar{X}_s + \Dw \bar{U}_s \cdot \partial_x \bar{X}_s \\
&= \hat{Z}_s \cdot \partial_x \bar{X}_s,
\end{align*}
\begin{align*}
\partial_x \left[\bar{F} \left( \bar{X}_s, \bar{Y}_s, \bar{Z}_s \right) \right] 
&= \Dt \partial_x \bar{Y}_s \\
&= \Dt \left( \bar{U}_s \cdot \partial_x \bar{X}_s \right) \\
&= \bar{U}_s \cdot \Dt \partial_x \bar{X}_s + \Dt \bar{U}_s \cdot \partial_x \bar{X}_s + \Dw \bar{U}_s \cdot \Dw \partial_x \bar{X}_s \\
&= \bar{U}_s \cdot \partial_x \left[ \bar{M} \left( \bar{X}_s, \bar{Y}_s, \bar{Z}_s \right) \right] + b_s \cdot \partial_x \bar{X}_s,
\end{align*}
where we can further specify
\begin{align*}
&\partial_x \left[ \bar{M} \left( \bar{X}_s, \bar{Y}_s, \bar{Z}_s \right) \right] \\
& \hspace*{1.5cm} = \diff{x}{\bar{M}} \left( \bar{X}_s, \bar{Y}_s, \bar{Z}_s \right) \partial_x \bar{X}_s + \diff{y}{\bar{M}} \left( \bar{X}_s, \bar{Y}_s, \bar{Z}_s \right) \partial_x \bar{Y}_s + \diff{z}{\bar{M}} \left( \bar{X}_s, \bar{Y}_s, \bar{Z}_s \right) \partial_x \bar{Z}_s \\
& \hspace*{1.5cm} = \diff{x}{\bar{M}} \left( \bar{X}_s, \bar{Y}_s, \bar{Z}_s \right) \partial_x \bar{X}_s + \diff{y}{\bar{M}} \left( \bar{X}_s, \bar{Y}_s, \bar{Z}_s \right) \bar{U}_s \partial_x \bar{X}_s + \diff{z}{\bar{M}} \left( \bar{X}_s, \bar{Y}_s, \bar{Z}_s \right) \hat{Z}_s \partial_x \bar{X}_s
\end{align*}
and likewise
\begin{align*}
& \partial_x \left[ \bar{F} \left( \bar{X}_s, \bar{Y}_s, \bar{Z}_s \right) \right] \\
& \hspace*{1.5cm} = \diff{x}{\bar{F}} \left( \bar{X}_s, \bar{Y}_s, \bar{Z}_s \right) \partial_x \bar{X}_s + \diff{y}{\bar{F}} \left( \bar{X}_s, \bar{Y}_s, \bar{Z}_s \right) \bar{U}_s \partial_x \bar{X}_s + \diff{z}{\bar{F}} \left( \bar{X}_s, \bar{Y}_s, \bar{Z}_s \right) \hat{Z}_s \partial_x \bar{X}_s.
\end{align*}
Thus we get
\begin{equation*}
\hat{Z}_s = \partial_x \bar{Z}_s \cdot \left( \partial_x \bar{X}_s \right)^{-1}
\end{equation*}
and
\begin{align}\label{eq:driftBarU}
b_s
&= \diff{x}{\bar{F}} \left( \bar{X}_s, \bar{Y}_s, \bar{Z}_s \right) + \diff{y}{\bar{F}} \left( \bar{X}_s, \bar{Y}_s, \bar{Z}_s \right) \bar{U}_s + \diff{z}{\bar{F}} \left( \bar{X}_s, \bar{Y}_s, \bar{Z}_s \right) \hat{Z}_s \\
& \qquad \qquad + \bar{U}_s \left[ \diff{x}{\bar{M}} \left( \bar{X}_s, \bar{Y}_s, \bar{Z}_s \right) + \diff{y}{\bar{M}} \left( \bar{X}_s, \bar{Y}_s, \bar{Z}_s \right) \bar{U}_s + \diff{z}{\bar{M}} \left( \bar{X}_s, \bar{Y}_s, \bar{Z}_s \right) \hat{Z}_s \right], \nonumber
\end{align}
where the derivatives of $\bar{M}$ and $\bar{F}$ are bounded due to the assumptions made.
Therefore, we see that the dynamics of $\bar{U}$ are linear with exception to the quadratic terms $\bar{U}_s \diff{y}{\bar{M}} (\bar{X}_s, \bar{Y}_s, \bar{Z}_s) \bar{U}_s$ and $\diff{z}{\bar{M}} \left( \bar{X}_s, \bar{Y}_s, \bar{Z}_s \right) \hat{Z}_s$. However, we claim that we can reduce the dynamics of $\bar U$ to a linear BSDE. 

It is straightforward to see that
\begin{equation*}
\diff{y}{\bar{M}} (\bar{X}_s, \bar{Y}_s, \bar{Z}_s) = \left( \begin{array}{cccc}
0 & 0 & 0 & 0 \\
-2 \frac{(\Zo{0}{s})^2}{\sigma_s^2} \frac{\sigma_{a,s}}{\sigma_s} & 0 & 0 & 0 \\
\frac{(\Zo{0}{s})^2}{\sigma_s^2} \mu_{a,s} - 2 \frac{(\Zo{0}{s})^2}{\sigma_s^2} \frac{\sigma_{a,s}}{\sigma_s} \mu_s & 0 & 0 & 0
\end{array} \right).
\end{equation*}
Note that $\alpha := -2 \frac{(\Zo{0}{s})^2}{\sigma_s^2} \frac{\sigma_{a,s}}{\sigma_s}$ and $ \beta := \frac{(\Zo{0}{s})^2}{\sigma_s^2} \mu_{a,s} - 2 \frac{(\Zo{0}{s})^2}{\sigma_s^2} \frac{\sigma_{a,s}}{\sigma_s} \mu_s$ are both uniformly bounded, and we have
\begin{equation*}
\diff{y}{\bar{M}} \left( \bar{X}_s, \bar{Y}_s, \bar{Z}_s \right) \bar{U}_s = \left( \begin{array}{ccc}
0 & 0 & 0 \\
\alpha \cdot \u{1}{s} & \alpha \cdot \u{2}{s} & \alpha \cdot \u{3}{s} \\
\beta \cdot \u{1}{s} & \beta \cdot \u{2}{s} & \beta \cdot \u{3}{s}
\end{array} \right),
\end{equation*}
which is bounded independently of $[t,1]$ (cf.\ in Theorem \ref{thm:U1isBounded}). 

Moreover, note that 
\begin{equation*}
\diff{z}{\bar{M}} \left( \bar{X}_s, \bar{Y}_s, \bar{Z}_s \right) 
= \left( \begin{array}{cccc} 0 & 0 & 0 & 0 \\
\frac{2 \Zo{0}{s}}{\sigma_s^2} & 0 & 0 & 0 \\
\mu_s \frac{2 \Zo{0}{s}}{\sigma_s^2} & 0 & 0 & 0 \end{array} \right)
\end{equation*}
only depends on the solution components $(\X{2}{}, \X{3}{}, \Y{0}{},\Zo{0}{})$. Hence, together with the estimates of Theorem \ref{thm:U1isBounded}, we conclude that $\diff{x}{\bar{M}}(\bar{X}_s,\bar{Y}_s,\bar{Z}_s)$ is bounded. 
Since $\bar U$ is bounded on $[t,1]$, the term $\bar U_s \diff{z}{\bar{M}} \left( \bar{X}_s, \bar{Y}_s, \bar{Z}_s \right) \hat{Z}_s$ in Equation \eqref{eq:driftBarU} can be shifted, via a Girsanov measure change, into the Brownian motion $W$. Similary, the term $\diff{z}{\bar F} \left( \bar{X}_s, \bar{Y}_s, \bar{Z}_s \right) \hat{Z}_s$ in Equation \eqref{eq:driftBarU} can be shifted into $W$. To sum up, there exists a Brownian motion $\hat W$ under an equivalent probability measure such that $(\bar U, \hat Z)$ solves the BSDE on $[t,1]$ driven by $\hat W$ with linear driver 
\begin{align*}
f(s,y,z) = \diff{x}{\bar{F}} \left( \bar{X}_s, \bar{Y}_s, \bar{Z}_s \right) + \diff{y}{\bar{F}} \left( \bar{X}_s, \bar{Y}_s, \bar{Z}_s \right) y + y \left[ \diff{x}{\bar{M}} \left( \bar{X}_s, \bar{Y}_s, \bar{Z}_s \right) + \diff{y}{\bar{M}} \left( \bar{X}_s, \bar{Y}_s, \bar{Z}_s \right) \bar{U}_s \right]
\end{align*}
and terminal condition $\nabla \bar \xi(\bar X_1)$. Observe that the terminal condition and all coefficients are bounded by some constant independent of $t$ and $x$. Therefore, also $\bar U$ is bounded independently of $t$ and $x$.
By Lemma \ref{UNIqXYZM} this yields that $I^M_{\mathrm{max}} = [0,1]$.
\end{proof}

\begin{remarks} The second and third derivatives do not have to be bounded. It would suffice if the second and third derivatives of $\mu$ divided by $\sigma^2$ and the second and third derivatives of $\sigma$ divided by $\sigma$ are bounded.
\end{remarks}

\begin{lemma}
\label{lemma:UthirdTimeDiff}
Let $g$, $\mu$ and $\sigma$ fulfill Assumption \ref{asump:conditionsU1bounded} and their second and third derivatives be bounded. Then the decoupling field $u$ of FBSDE \eqref{fbsde:base} is three times weakly differentiable w.r.t.\ to the initial condition $x \in \IR^3$ with uniformly bounded derivatives.
\end{lemma}

\begin{proof}
This proof is completely analogous the proof of Lemma \ref{lemma:fbsdeSecondOrderBounded}. Therefore, we only give a sketch.

Extend the system \eqref{align:fbsde:secondOrder} by the dynamics of $\p{\bar{Y}}{ij}{} := \u{ij}{} := \diff{x_j}{\u{i}{}}$ for all $i,j \in \{1,2,3\}$ as obtained in the proof of Lemma \ref{lemma:fbsdeSecondOrderBounded} and by the corresponding entries in the decoupling field. Then argue analogously to the proof of Lemma \ref{lemma:fbsdeSecondOrderBounded} that for every $i \in \{0,1,2,3 \} $ the $\u{i}{}$ of FBSDE \eqref{align:fbsde:secondOrder} coincides with the $\u{i}{}$ of the extended system. Redefine, if necessary, the vectors $\bar{X}$, $\bar{Y}$,$\bar{Z}$ and the functions $\bar{M}$, $\bar{\Sigma}$, $\bar{\xi}$, $\bar{F}$ such that for the extended system we have
\begin{equation*}
\bar{X}_s = x + \int_t ^s \bar{M} \left( \bar{X}_r, \bar{Y}_r, \bar{Z}_r \right) \d r + \int_t^s \bar{\Sigma} \d W_r
\end{equation*}
and
\begin{equation*}
\bar{Y}_s = \bar{\xi} \left( \bar{X}_1 \right) - \int_s^1 \bar{F} \left( \bar{X}_r, \bar{Y}_r, \bar{Z}_r \right) \d r - \int_s^1 \bar{Z}_r \d W_r.
\end{equation*}
Also define $\bar{U}_s$ as the partial derivatives of the decoupling field $u (s,\bar{X}_s)$ of the extended system for all $s \in [t,1]$. Again there exist $(b_s)$ and $(\hat{Z}_s)$ such that
\begin{equation*}
\bar{U}_s = \bar{U}_1 - \int_s^1 b_r \d r - \int_s^1 \hat{Z}_r \d W_r.
\end{equation*}
By the same calculation as in the proof of Lemma \ref{lemma:fbsdeSecondOrderBounded} we obtain that
\begin{equation*}
\hat{Z}_s = \partial_x \bar{Z}_s \cdot \left( \partial_x \bar{X}_s \right)^{-1}
\end{equation*}
and
\begin{align*}
b_s
&= \diff{x}{\bar{F}} \left( \bar{X}_s, \bar{Y}_s, \bar{Z}_s \right) + \diff{y}{\bar{F}} \left( \bar{X}_s, \bar{Y}_s, \bar{Z}_s \right) \bar{U}_s + \diff{z}{\bar{F}} \left( \bar{X}_s, \bar{Y}_s, \bar{Z}_s \right) \hat{Z}_s \\
& \qquad \qquad + \bar{U}_s \left[ \diff{x}{\bar{M}} \left( \bar{X}_s, \bar{Y}_s, \bar{Z}_s \right) + \diff{y}{\bar{M}} \left( \bar{X}_s, \bar{Y}_s, \bar{Z}_s \right) \bar{U}_s + \diff{z}{\bar{M}} \left( \bar{X}_s, \bar{Y}_s, \bar{Z}_s \right) \hat{Z}_s \right].
\end{align*}
Analogous to the proof above, $\diff{x}{\bar{F}}$, $\diff{y}{\bar{F}}$, $\diff{z}{\bar{F}}$, $\diff{x}{\bar{M}}$, $\diff{y}{\bar{M}}$ and $\diff{z}{\bar{M}}$ are bounded while additionally $\diff{y}{\bar{M}}$ only has entries in the first column which allows us to conclude that $\diff{y}{\bar{M}}(\bar{X}_s,\bar{Y}_s,\bar{Z}_s) \bar{U}_s$ is bounded. Furthermore every coefficient in front of $\hat{Z}$ is bounded on every Interval $[t,1] \subset I_{\mathrm{max}}^M$ and can therefore be transformed away with Girsanov's Theorem. Hence we have linear dynamics for $\bar{U}$ with bounded coefficients which yields that it is bounded independently of the interval $[t,1]$, giving $I_{\mathrm{max}}^M = [0,1]$.
\end{proof}

\begin{lemma}
\label{lemma:u1GreaterZero}
Let $g$, $\mu$ and $\sigma$ fulfill Assumption \ref{asump:conditionsU1bounded}, their first and second derivatives be bounded and $g' \geq \delta > 0$. Then the weak derivative $\diff{x_1}{u}$ of the decoupling field $u$ from the FBSDE \eqref{fbsde:base} fulfills
\begin{equation}\label{est:1overU1}
\left\Vert \frac{1}{\diff{x_1}{u}} \right\Vert_\infty 
\leq \left\Vert \frac{1}{g'} \right\Vert_\infty + \Vert \diff{x_1}{u} \Vert_\infty \left( \left\Vert \frac{\diff{a}{\mu}}{\sigma^2} \right\Vert_\infty + 2 \left\Vert \frac{\mu}{\sigma^2} \right\Vert_\infty \left\Vert \frac{\diff{a}{\sigma}}{\sigma} \right\Vert_ \infty + \frac{2}{\epsilon^2} \Vert \diff{x_2}{u} \Vert_\infty \left\Vert \frac{\diff{a}{\sigma}}{\sigma} \right\Vert_\infty \right)
\end{equation}
and in particular $\diff{x_1}{u}$ is bounded away from $0$.
\end{lemma}

\begin{proof}
By Lemma \ref{lemma:fbsdeSecondOrderBounded} the decoupling field of the FBSDE \eqref{fbsde:base} exists on the whole interval $[0,1]$ and is twice weakly differentiable. In particular $\diff{x_1}{u}$ is continuous (see e.g.\ Theorem 4.2.17 in \cite{Fromm15}), and hence we can apply Lemma \ref{lemma:ZequalsU1} yielding $\Zo{0}{r} = \diff{x_1}{u}\left(r,\X{1}{r}, \X{2}{r}, \X{3}{r} \right)$ for all $r \in [0,1]$. Also using Lemma \ref{lemma:u3=-1:u1Bounded} we know that $\u{1}{}$ is bounded by some constant for every starting time $t \in I_{\mathrm{max}}^M = [0,1]$ and every initial value $x \in \IR^3$.

Now we set $V_r := \frac{1}{\diff{x_1}{u} \left(r,\X{1}{r}, \X{2}{r}, \X{3}{r} \right)}$ for all $r \in (t_0,1]$ where $t_0:= \inf \{ t \geq 0 \vert \diff{x_1}{u} (t,x) = 0 \text{ for at least one } x \in \IR^3 \}$ with the convention that $\inf \emptyset = 0$.
We immediately get that $\frac{1}{V_r} \leq \Vert \diff{x_1}{u} \Vert_\infty < \infty$ and the dynamics
\begin{align*}
V_s
&= \frac{1}{g'(\X{1}{1})} - \int_s^1 \left( V_r ^3 \left( \Zo{1}{r} \right)^2 - \frac{1}{V_r} \frac{\mu_{a,r} - 2 \mu_r \frac{\sigma_{a,r}}{\sigma_r} + 2 \u{2}{r} \frac{\sigma_{a,r}}{\sigma_r} }{\sigma_r^2} \right) \d r - \int_s ^1 - \Zo{1}{r} V_r^2 \d \W{r} \\
&= \frac{1}{g'(\X{1}{1})} - \int_s^1 \frac{1}{V_r} \left( \left( \hat{Z}_r \right)^2 - \frac{\mu_{a,r} - 2 \mu_r \frac{\sigma_{a,r}}{\sigma_r} + 2 \u{2}{r} \frac{\sigma_{a,r}}{\sigma_r} }{\sigma_r^2} \right) \d r - \int_s ^1 \hat{Z}_r \d \W{r},
\end{align*}
where $\u{2}{r} := \diff{x_2}{u} \left( r, \X{1}{r}, \X{2}{r}, \X{3}{r} \right)$, $\hat{Z}_r := - \frac{\Zo{1}{r}}{V_r^2}$ and $\tilde W$ is defined as in the proof of Lemma \ref{lemma:fbsdeSecondOrderBounded}.

Using that $\frac{1}{V_s} \leq \Vert \diff{x_1}{u} \Vert_\infty$ we can apply Corollary 2.2 of \cite{Kobylanski2000} to obtain
\begin{equation*}
\Vert V \Vert_\infty 
\leq \left\Vert \frac{1}{g'} \right\Vert_\infty + \Vert \diff{x_1}{u} \Vert_\infty \left( \left\Vert \frac{\diff{a}{\mu}}{\sigma^2} \right\Vert_\infty + 2 \left\Vert \frac{\mu}{\sigma^2} \right\Vert_\infty \left\Vert \frac{\diff{a}{\sigma}}{\sigma} \right\Vert_ \infty + \frac{2}{\epsilon^2} \Vert \diff{x_2}{u} \Vert_\infty \left\Vert \frac{\diff{a}{\sigma}}{\sigma} \right\Vert_\infty \right) < \infty
\end{equation*}
because $\diff{x_1}{u}$ and $\diff{x_2}{u}$ are bounded by Theorem \ref{thm:U1isBounded}. Since this bound is independent of $s$ we also get that
\begin{equation*}
\diff{x_1}{u} \left(s,\X{1}{s}, \X{2}{s}, \X{3}{s} \right) = \frac{1}{V_s} \geq \frac{1}{\Vert V \Vert_\infty} > 0
\end{equation*}
for all $s$ where $V$ is defined. Because, as stated above, $\diff{x_1}{u}$ is continuous, we get that $t_0 = 0$ and that hence Equation \eqref{est:1overU1} holds true.
\end{proof}

\begin{lemma}
\label{lemma:Z1Bounded}
Let $g$, $\mu$ and $\sigma$ fulfill Assumption \ref{asump:conditionsU1bounded} and their second derivatives be bounded. Then for the problem \eqref{align:fbsde:secondOrder} it holds for all $s \in [0,1]$ almost surely that
\begin{equation*}
\vert \Zo{1}{s} \vert \leq \Vert \diff{x_1}{\u{1}{}} \Vert_\infty < \infty.
\end{equation*}
\end{lemma}

\begin{proof}
Note that this proof runs on similar lines as the proof of Lemma \ref{lemma:ZequalsU1}.

Remember that Lemma \ref{lemma:fbsdeSecondOrderBounded} yields that for problem \eqref{align:fbsde:secondOrder} there exists a unique solution on the whole interval $[0,1]$ for every initial condition in $\IR^3$. 
Observe that with It\^o's formula we get for $h > 0$ and $s, s+h \in [0,1]$
\begin{align*}
\frac{1}{h} \E &\left[ \left. \Y{1}{s+h} ( W_{s+h} - W_s ) \right\vert \mathcal{F}_s \right] \\
=& \frac{1}{h} \E \left[ \left. \int_s^{s+h} \Y{1}{r} \d W_r + \int_s^{s+h} (W_{r} - W_s) \d \Y{1}{r} + \int_s^{s+h} 1 \cdot \Zo{1}{r} \d r \right\vert \mathcal{F}_s \right] \\
=& \frac{1}{h} \E \left[ \int_s^{s+h} \Zo{1}{r} \d r + \int_s^{s+h} \Y{1}{r} \d W_r + \int_s^{s+h} (W_{r} - W_s) \Zo{1}{r} \d W_r \right. \\
& \qquad + \int_s^{s+h} (W_{r} - W_s) \left( - \frac{ ( \Zo{0}{r} )^2}{\sigma_r^2} \Y{1}{r} \left( \Y{3}{r} \left( \mu_{a,r} - 2 \mu_r \frac{\sigma_{a,r}}{\sigma_r} \right) - 2 \Y{2}{r} \frac{\sigma_{a,r}}{\sigma_r} \right) \right) \d r \\
& \qquad + \left. \left. \int_s^{s+h} \left( W_r - W_s \right) \left( - \frac{2 \Zo{0}{r}}{\sigma_r^2 } \left( \Y{2}{r} + \Y{3}{r} \mu_r \right) \Zo{1}{r} \right) \d r \right\vert \mathcal{F}_s \right] \\
\rightarrow & \Zo{1}{s} \ \ a.s. \quad \text{ as } \quad h \rightarrow 0.
\end{align*}
On the other hand we get by using the decoupling condition that
\begin{align}
\Y{1}{s+h} & ( W_{s+h} -  W_s ) \nonumber \\
=& \u{1}{} \left( s+h, \X{1}{s+h}, \X{2}{s+h}, \X{3}{s+h} \right) (W_{s+h} - W_s ) \nonumber \\
=& \u{1}{} \left( s+h, \X{1}{s+h}, \X{2}{s}, \X{3}{s} \right) (W_{s+h} - W_s ) \label{eq:Z1leqU1Expansion} \\
& \ + \left( \u{1}{} \left( s+h, \X{1}{s+h}, \X{2}{s+h}, \X{3}{s} \right) - \u{1}{} \left( s+h, \X{1}{s+h}, \X{2}{s}, \X{3}{s} \right) \right) (W_{s+h} - W_s ) \nonumber \\
& \ + \left( \u{1}{} \left( s+h, \X{1}{s+h}, \X{2}{s+h}, \X{3}{s+h} \right) - \u{1}{} \left( s+h, \X{1}{s+h}, \X{2}{s+h}, \X{3}{s} \right) \right) (W_{s+h} - W_s ). \nonumber
\end{align}
At first let us take a look at the third summand at the right hand side of \eqref{eq:Z1leqU1Expansion}. Since $\u{1}{}$ is Lipschitz continuous in its fourth argument with some constant $L_{\u{1}{},x_3}^t$ and since furthermore
\begin{equation*}
\X{3}{s+h} = \X{3}{s} + \int_s ^{s+h} \mu_r \frac{(\Zo{0}{r}) ^2}{\sigma_r^2} \d r
\end{equation*}
we can estimate
\begin{align*}
\frac{1}{h} & \left\vert \E \left[ \left. \left( \u{1}{} \left( s+h, \X{1}{s+h}, \X{2}{s+h}, \X{3}{s+h} \right) - \u{1}{} \left( s+h, \X{1}{s+h}, \X{2}{s+h}, \X{3}{s} \right) \right) (W_{s+h} - W_s ) \right\vert \mathcal{F}_s \right] \right\vert \\
& \leq \frac{1}{h} \E \left[ \left. L_{\u{1}{},x_3}^t \left\vert \int_s ^{s+h} \mu_r \frac{(\Zo{0}{r}) ^2}{\sigma_r^2 } \d r \right\vert \left\vert W_{s+h} - W_s  \right\vert \right\vert \mathcal{F}_s \right] \\
& \leq \frac{1}{h} L_{\u{1}{},x_3}^t h \Big\Vert \frac{\mu}{\sigma^2} \Big\Vert_\infty \Vert \Zo{0}{} \Vert_\infty ^2 \E \left[ \left. \left\vert W_{s+h} - W_s  \right\vert \right\vert \mathcal{F}_s \right],
\end{align*}
which clearly goes to $0$ as $h \rightarrow 0$ because $\frac{\mu}{\sigma^2}$ and $\Zo{0}{}$ are bounded by Theorem \ref{thm:U1isBounded}.
Analogously we get, with $L_{\u{1}{},x_2}^t$ being the Lipschitz constant of $\u{1}{}$ in the third argument, that
\begin{align*}
\frac{1}{h} & \left\vert \E \left[ \left. \left( \u{1}{} \left( s+h, \X{1}{s+h}, \X{2}{s+h}, \X{3}{s} \right) - \u{1}{} \left( s+h, \X{1}{s+h}, \X{2}{s}, \X{3}{s} \right) \right) (W_{s+h} - W_s ) \right\vert \mathcal{F}_s \right] \right\vert \\
& \leq \frac{1}{h} L_{\u{1}{},x_2}^t h \Vert \Zo{0}{} \Vert_\infty ^2 \epsilon^{-2} \E \left[ \left. \left\vert W_{s+h} - W_s  \right\vert \right\vert \mathcal{F}_s \right] \\
& \rightarrow 0 \ \ a.s. \quad \text{ for } \quad h \rightarrow 0.
\end{align*}

Now consider the remaining first term on the right hand side of Equation \eqref{eq:Z1leqU1Expansion}.
Using integration by parts we obtain
\begin{align*}
&\E \left[ \left. \u{1}{} \left( s+h, \X{1}{s+h}, \X{2}{s}, \X{3}{s} \right) ( W_{s+h} - W_s) \right| \mathcal{F}_s \right] \\
& \hspace*{5cm} = \int_\IR \u{1}{} \left( s+h, \X{1}{s} + z \sqrt{h}, \X{2}{s}, \X{3}{s} \right) z \sqrt{h} \frac{1}{\sqrt{2 \pi}} e^{-\frac{1}{2} z^2} \d z \\
& \hspace*{5cm} = \int_\IR \diff{x_1}{\u{1}{}} \left( s+h, \X{1}{s} + z \sqrt{h} , \X{2}{s}, \X{3}{s} \right) h \frac{1}{\sqrt{2 \pi}} e^{-\frac{1}{2} z^2} \d z.
\end{align*}
Since $\diff{x_1}{\u{1}{}}$ is bounded as proved in Lemma \ref{lemma:fbsdeSecondOrderBounded} we have
\begin{align*}
& \left\vert \frac{1}{h} \E \left[ \left. \u{1}{} \left( s+h, \X{1}{s+h}, \X{2}{s}, \X{3}{s} \right) ( W_{s+h} - W_s) \right| \mathcal{F}_s \right] \right\vert \\
& \hspace*{5cm} = \left\vert \int_\IR \diff{x_1}{\u{1}{}} \left( s+h, \X{1}{s} + z \sqrt{h} , \X{2}{s}, \X{3}{s} \right) \frac{1}{\sqrt{2 \pi}} e^{-\frac{1}{2} z^2} \d z \right\vert \\
& \hspace*{5cm} \leq \Vert \diff{x_1}{\u{1}{}} \Vert_\infty.
\end{align*}
Putting the derived estimates together we get
\begin{equation*}
\left\vert \Zo{1}{s} \right\vert
= \left\vert \lim_{h \searrow 0} \frac{1}{h} \E \left[ \left. \Y{1}{s+h} ( W_{s+h} - W_s ) \right\vert \mathcal{F}_s \right] \right\vert \\
\leq \left\Vert \diff{x_1}{\u{1}{}} \right\Vert_\infty.
\end{equation*}
By Lemma \ref{lemma:fbsdeSecondOrderBounded}, $\Vert \diff{x_1}{\u{1}{}} \Vert_\infty < \infty$, which further implies the result.
\end{proof}

\begin{proposition}
\label{prop:conditionsStronSolution}
Let $g$, $\mu$ and $\sigma$ fulfill Assumption \ref{asump:conditionsU1bounded}, let their first, second and third derivatives as well as $\sigma$ and $\frac{1}{g'}$ be bounded.
Then the requirements of Theorem \ref{thm:stoppingB:strongSolution} are fulfilled.
\end{proposition}

\begin{proof}
Remember that the derivative $\diff{x_1}{u}$ of the decoupling field of FBSDE \eqref{fbsde:base} equals $\u{1}{}$ of the decoupling field of FBSDE \eqref{align:fbsde:secondOrder} by Lemma \ref{lemma:fbsdeSecondOrderBounded} and which, by Lemma \ref{lemma:u1GreaterZero}, is bounded from below by a $\delta > 0$. Hence, it only remains to show that $\diff{x_1}{u}$ which equals $\u{1}{}$ is Lipschitz continuous. Since we already know that the derivatives w.r.t.\ the space variables are bounded (by Lemma \ref{lemma:fbsdeSecondOrderBounded}) we only need to prove that $\u{1}{}$ is Lipschitz continuous in the time variable.

Consider FBSDE \eqref{align:fbsde:secondOrder} for a starting time $t \in [0,1)$ on the interval $[t,1]$ with initial condition $\big( \X{1}{t}, \X{2}{t}, \X{3}{t} \big) = ( \x{1}, \x{2}, \x{3}) = x \in \IR^3$. Let $s \in (t,1]$. Using the triangle inequality several times gives
\begin{align*}
\left\vert \u{1}{} (s,x) - \u{1}{} (t,x) \right\vert
\leq & \left\vert \u{1}{} (s,x) - \E \left[ \u{1}{} \left(s, \X{1}{s}, \x{2}, \x{3} \right) \right] \right\vert \\
& \qquad + \left\vert \E \left[ \u{1}{} \left(s, \X{1}{s}, \x{2}, \x{3} \right) \right] - \E \left[ \u{1}{} \left(s, \X{1}{s}, \X{2}{s}, \x{3} \right) \right] \right\vert \\
& \qquad + \left\vert \E \left[ \u{1}{} \left(s, \X{1}{s}, \X{2}{s}, \x{3} \right) \right] - \E \left[ \u{1}{} \left(s, \X{1}{s}, \X{2}{s}, \X{3}{s} \right) \right] \right\vert \\
&\qquad + \left\vert \E \left[ \u{1}{} \left(s, \X{1}{s}, \X{2}{s}, \X{3}{s} \right) - \u{1}{} \left( t, \X{1}{t}, \X{2}{t}, \X{3}{t} \right) \right]  \right\vert.
\end{align*}
We take a closer look at every summand on the right hand side starting with the first one. By defining
\begin{equation*}
\phi (z) := \u{1}{} (s,\x{1}, \x{2}, \x{3} ) - \u{1}{} ( s, \x{1} + z, \x{2}, \x{3} )
\end{equation*}
we see that the first summand equals $\vert \E [ \phi (W_s - W_t) ] \vert$. Furthermore, $\phi(0) = 0$ and by Lemma \ref{lemma:UthirdTimeDiff}, $\phi$ is two times weakly differentiable with derivatives bounded by some constant $L_{\diff{x}{\u{1}{}}} < \infty$. Hence, the inequality $\left\vert \int_\IR \phi ( a \cdot z ) \frac{1}{\sqrt{2 \pi}} e^{-\frac{1}{2} z^2} \d z \right\vert \leq \frac{1}{2} a^2 \Vert \phi'' \Vert_\infty$ holds true (see e.g.\ Lemma 4.3.11 in \cite{Fromm15}). Therefore,
\begin{equation*}
\left\vert \u{1}{} (s,x) - \E \left[ \u{1}{} \left(s, \X{1}{s}, \x{2}, \x{3} \right) \right] \right\vert
= \left\vert \E \left[ \phi (W_s - W_t) \right] \right\vert
\leq \frac{(s-t)}{2} \cdot L_{\diff{x}{\u{1}{}}}.
\end{equation*}

For the second summand we use the Lipschitz constant of $\u{1}{}$ denoted by $L_{\u{1}{}}$ to get
\begin{align*}
\left\vert \E \left[ \u{1}{} \left(s, \X{1}{s}, \x{2}, \x{3} \right) - \u{1}{} \left(s, \X{1}{s}, \X{2}{s}, \x{3} \right) \right] \right\vert
\leq & L_{\u{1}{}} \E \left\vert \X{2}{s} - \x{2} \right\vert \\
= & L_{\u{1}{}} \E \left\vert \int_t^s \frac{(\Zo{0}{r})^2}{\sigma_r^2 } \d r \right\vert \\
\leq & L_{\u{1}{}} \Vert \u{1}{} \Vert_\infty ^2 \epsilon^{-2} (s-t)
\end{align*}
since $\vert \Zo{0}{} \vert \leq \Vert \u{1}{} \Vert_\infty < \infty$ by Theorem \ref{thm:U1isBounded}.

The third summand can be estimated similarly by
\begin{align*}
&\left\vert \E \left[ \u{1}{} \left(s, \X{1}{s}, \X{2}{s}, \x{3} \right) - \u{1}{} \left(s, \X{1}{s}, \X{2}{s}, \X{3}{s} \right) \right] \right\vert \\
&\qquad \qquad \qquad \qquad \qquad \qquad
\leq  L_{\u{1}{}} \E \left\vert \int_t^s \mu_r \frac{(\Zo{0}{r})^2}{\sigma_r^2} \d r \right\vert \\
&\qquad \qquad \qquad \qquad \qquad \qquad
\leq L_{\u{1}{}} \Big\Vert \frac{\mu}{\sigma^2} \Big\Vert_\infty \Vert \u{1}{} \Vert_\infty^2 (s-t).
\end{align*}

For the last summand we use the decoupling condition and $\Y{3}{\cdot} = -1$ to obtain
\begin{align*}
&\left\vert \E \left[ \u{1}{} \left(s, \X{1}{s}, \X{2}{s}, \X{3}{s} \right) - \u{1}{} \left( t, \X{1}{t}, \X{2}{t}, \X{3}{t} \right) \right]  \right\vert \\
& \qquad \leq \left\vert \E \left[ \Y{1}{s} - \Y{1}{t} \right] \right\vert \\
& \qquad = \left\vert \E \left[ \int_t^s \Y{1}{r} \frac{(\Zo{0}{r} )^2}{\sigma_r^2} \left( \mu_{a,r} - 2 \mu_r \frac{\sigma_{a,r}}{\sigma_r} + 2 \Y{2}{r} \frac{\sigma_{a,r}}{\sigma_r} \right) - \frac{2 \Zo{0}{r}}{\sigma_r^2} \left(\Y{2}{r} - \mu_r \right) \Zo{1}{r} \d r \right] \right\vert \\
& \qquad \leq \left[ \Vert \u{1}{} \Vert_\infty^3 \left( \left\Vert \frac{\diff{a}{\mu}}{\sigma^2} \right\Vert_\infty + 2 \left\Vert \frac{\mu}{\sigma^2} \right\Vert_\infty  \left\Vert \frac{\diff{a}{\sigma}}{\sigma} \right\Vert_\infty + \frac{2}{\epsilon^2} \Vert \u{2}{} \Vert_\infty \left\Vert \frac{\diff{a}{\sigma}}{\sigma} \right\Vert_\infty \right) \right. \\
& \left. \qquad \qquad \qquad \qquad \qquad + 2 \Vert \u{1}{} \Vert_\infty \left( \epsilon^{-2} \Vert \u{2}{} \Vert_\infty + \Big\Vert \frac{\mu}{\sigma^2} \Vert_\infty \right) \Vert \diff{x_1}{\u{1}{}} \Vert_\infty \right] (s-t)
\end{align*}
where we applied Theorem \ref{thm:U1isBounded} and Lemma \ref{lemma:Z1Bounded}. Thus, the last summand is Lipschitz continuous by Theorem \ref{thm:U1isBounded} and Lemma \ref{lemma:fbsdeSecondOrderBounded}, too.

Putting all estimates together we arrive at $\vert \u{1}{} (s,x) - \u{1}{} (r,x) \vert \leq L (s-t)$ for some finite constant $L$ which is independent of $s$ and $t$. Hence $\u{1}{}$ is Lipschitz continuous in the time variable.

\end{proof}

Observe that Proposition \ref{prop:conditionsStronSolution} and Theorem \ref{thm:stoppingB:strongSolution} imply Theorem \ref{thm:mainResultStrong}.

\section{Numerics}
\label{sec:numerics}

We now illustrate numerically an example of an embedding using the methodology developed. This is done by numerically approximating the solution of the FBSDE
\begin{align}
W_s =&  \int_0^s \frac{\sigma ( \X{2}{r}, Y_r + \X{3}{r} ) }{ Z_r } \d B_{\X{2}{r}} \nonumber \\
\X{2}{s} =& \int _0 ^s \frac{Z_r^2}{\sigma^2 ( \X{2}{r}, Y_r + \X{3}{r} ) } \d r \label{sys:fbsdeAndBCompact} \\
\X{3}{s} =& \int_0^s \mu (\X{2}{r}, Y_r + \X{3}{r}) \frac{Z_r^2}{\sigma^2 ( \X{2}{r}, Y_r + \X{3}{r} ) } \d r \nonumber \\
Y_s =& g( W_1 ) - \X{3}{1} - \int _s ^1 Z_r \d W_r. \nonumber
\end{align}
To the best of our knowledge no literature exists able to deal directly with approximations of \eqref{sys:fbsdeAndBCompact} and hence, inspired by known literature, we propose a numerical scheme whose rigorous study is left for future research. FBSDE \eqref{sys:fbsdeAndBCompact} is a fully coupled quadratic growth FBSDE which we deal with as follows:  from \cite{IRZ10} we inject the theoretical a priori hard bounds in the coefficients, reducing FBSDE \eqref{sys:fbsdeAndBCompact} to a uniformly Lipschitz fully-coupled one, then apply a decoupling technique based on Picard iterations \cite{BZ08} to reduce the problem to the iterative simulation of uniformly Lipschitz fully-decoupled FBSDE. The final approximation step is carried out using a classic explicit Euler scheme discretization \cite{BZ08} while the  approximation of the conditional expectations is done via projection over basis functions \cite{GLW05}. The final outcome is the approximation of the embedding stopping time and the verification that the stopped process does embed the target distribution.

From a mathematical point of view, the only step of the described numerical approximation that cannot be fully justified is the convergence of the Picard iteration step. The results of \cite{BZ08} do not apply if the diffusion coefficient $\sigma$ depends on $Z$.
We stress, however, that for some special cases the algorithm outlined below can be shown to converge, e.g.~in the homogeneous case (see Remark \ref{remark:homocase} below).

\subsection{The problem, its conditions and the hard bounds}

At first we show that FBSDE \eqref{sys:fbsdeAndBCompact} has a unique solution from which we can construct a strong solution of the \SEP.

\begin{proposition}
\label{prop:numericsSystem}
Let the assumptions of Theorem \ref{thm:stoppingB:strongSolution} or Proposition \ref{prop:conditionsStronSolution} be satisfied. Denote by $u$ the decoupling field of FBSDE \eqref{fbsde:base}. Let $B$ be an arbitrary Brownian motion and denote by $( \mathcal{F}^B ) = ( \mathcal{F}^B_s)_{s \in [0, \infty)}$ the augmented filtration generated by $B$. 
Then there exist unique square-integrable processes $(W,X^{(2)},X^{(3)},Y)$ solving the FBSDE \eqref{sys:fbsdeAndBCompact}. Moreover, $\tau := \X{2}{1}$ is an $(\cF^B_t)$-stopping time bounded as in \eqref{est:tau}, $W$ is a Brownian motion on $[0,1]$ and the pair $(\tau,Y_0)$ is a strong solution of the SEP.
\end{proposition}

\begin{proof}
Remember that by Theorem \ref{thm:stoppingB:strongSolution} the SDE \eqref{sys:systemTau} has a unique solution $(\gamma, \Gamma, \Delta, \Theta)$. We introduce the time change $\gamma^{-1}(t) = \inf \{ r \geq 0 : \gamma(r) \geq t \} $ for $ t \in [0,1]$. Observe that $ \gamma^{-1}$ has the dynamics
\begin{equation*}
\gamma^{-1}(t) = \int_0^t \frac{\left( \diff{x_1}{u} ( s, \Gamma_{\gamma^{-1}(s)}, \gamma^{-1}(s), \Delta_{\gamma^{-1}(s)}) \right)^2}{\sigma^2 ( \gamma^{-1}(s), \Theta_{\gamma^{-1}(s)} + \Delta_{\gamma^{-1}(s)} ) } \d s.
\end{equation*}
By setting $Z_s := \diff{x_1}{u} ( s, \Gamma_{\gamma^{-1}(s)}, \gamma^{-1}(s), \Delta_{\gamma^{-1}(s)})$ for $s \in [0,1]$, replacing the dynamics of $\gamma$ by the dynamics of $\gamma^{-1}$ and applying the time change $\gamma^{-1}$ to all other processes, we can rewrite the system \eqref{sys:systemTau} as
\begin{align*}
\gamma^{-1}(t) &= \int_0^t \frac{\left( Z_s \right)^2}{\sigma^2 ( \gamma^{-1}(s), \Theta_{\gamma^{-1}(s)} + \Delta_{\gamma^{-1}(s)} ) } \d s \\
\Gamma_{\gamma^{-1}(t)} &= \int_0^t \frac{\sigma ( \gamma^{-1}(s), \Theta_{\gamma^{-1}(s)} + \Delta_{\gamma^{-1}(s)} ) }{ Z_s } \d B_{\gamma^{-1}(s)}  \\
\Delta_{\gamma^{-1}(t)} &= \int_0^t \mu ( \gamma^{-1}(s), \Theta_{\gamma^{-1}(s)} + \Delta_{\gamma^{-1}(s)} ) \frac{\left( Z_s \right)^2}{\sigma^2 ( \gamma^{-1}(s), \Theta_{\gamma^{-1}(s)} + \Delta_{\gamma^{-1}(s)} ) } \d s \\
\Theta_{\gamma^{-1}(t)} &= Y_0 + \int_0^t Z_s \d \Gamma_{\gamma^{-1}(s)}
\end{align*}
for all $t \in [0,1]$. Here it is straightforward to see that with $\gamma^{-1}(t) = \X{2}{t}$, $\Gamma_{\gamma^{-1}(t)} = W_t$, $\Delta_{\gamma^{-1}(t)} = \X{3}{t}$ and $\Theta_{\gamma^{-1}(t)} = Y_t$ we exactly have the system \eqref{sys:fbsdeAndBCompact}. Thus the system \eqref{sys:fbsdeAndBCompact} has a solution $(W, \X{2}{}, \X{3}{}, Y, Z)$ which fulfills that $\tau := \X{2}{1} = \gamma^{-1}(1) = \inf \{ r \geq 0 \vert \gamma(r) = 1 \}$ is a stopping time with regard to $(\cF^B_t)$ bounded as in \eqref{est:tau} and that $A_\tau \sim \nu$.

It remains to show the uniqueness of this solution. Now take an arbitrary square integrable solution $(W, \X{2}{}, \X{3}{}, Y, Z)$ of \eqref{sys:fbsdeAndBCompact}. Define the time change 
\begin{equation*}
\bar{\gamma} (t) := \left\{ \begin{array}{ll}
\inf \{ s \geq 0 : \X{2}{s} \geq t \}, & t \leq \X{2}{1} \\
1, & t > \X{2}{1}
\end{array} \right.
\end{equation*}
and observe that by
\begin{equation*}
\qvar{W}{t} = \int_0^{\X{2}{t}} \frac{ \sigma^2 \left( r, Y_{\bar{\gamma}(r)} + \X{3}{\bar{\gamma}(r)} \right) }{ Z_{\bar{\gamma} (r) }^2 } \d r = \int_0^t \frac{ \sigma^2 \left( \X{2}{r}, Y_{r} + \X{3}{r} \right) }{ Z_{r}^2 } \d \X{2}{r} = \int_0^t 1 \d r = t
\end{equation*}
$W$ is a Brownian motion on $[0,1]$. Thus the processes $(W, \X{2}{}, \X{3}{}, Y, Z)$ solve FBSDE \eqref{fbsde:base} for the initial value $0$. Due to Theorem \ref{GLObalexistM} and Lemma \ref{UNIqXYZM} this solution of FBSDE \eqref{fbsde:base} is unique.
\end{proof}

\begin{remarks}
\label{remark:simulationWeakSimpler}
If one is only interested in a weak solution, then only FBSDE \eqref{fbsde:base} needs to be solved, where $W$ is given, and the Brownian motion $B$ can be calculated afterwards, as described in Theorem \ref{thm:weakSolution}. Aside from simplifying the system that needs to be simulated, this also has the advantage of being valid for more general coefficients $\mu$ and $\sigma$ (compare the assumptions of Theorem \ref{thm:mainResultWeak} and Theorem \ref{thm:mainResultStrong}).
\end{remarks}

By the combination of Lemma \ref{lemma:u1GreaterZero}, Lemma \ref{lemma:ZequalsU1} and Theorem \ref{thm:U1isBounded} we have for $Z$ the $\lambda \times \IP$ a.s.\ bounds $0 < \widecheck{Z} \leq Z \leq \widehat{Z} < \infty$, which are
\begin{align*}
\widehat{Z} =& \left( \frac{1}{\Vert g' \Vert_\infty ^2} + 2 \min \left\{0, \inf_{ (\theta,x) \in \IR_+ \times \IR } \left( \frac{ \sigma \cdot \diff{a}{\mu} - 2 \diff{a}{\sigma} \cdot \mu }{\sigma^3} \right) (\theta,x) \right\} \right)^{-\frac{1}{2}} \quad \text{and} \\
\widecheck{Z} =& \left( \left\Vert \frac{1}{g'} \right\Vert_\infty + \widecheck{Z} \left( \left\Vert \frac{\diff{a}{\mu}}{\sigma^2} \right\Vert_\infty + 2 \left\Vert \frac{\mu}{\sigma^2} \right\Vert_\infty \left\Vert \frac{\diff{a}{\sigma}}{\sigma} \right\Vert_ \infty + \frac{2}{\epsilon^2} \Vert \diff{x_2}{u} \Vert_\infty \left\Vert \frac{\diff{a}{\sigma}}{\sigma} \right\Vert_\infty \right) \right)^{-1}
\end{align*}
with
\begin{align*}
\left\Vert \diff{x_2}{u} \right\Vert_\infty
& \leq \exp \left[ \widehat{Z}^2 \left( \left\Vert \frac{\diff{a}{\mu}}{\sigma^2} \right\Vert_\infty + 2 \left( \left\Vert \frac{\diff{a}{\sigma}}{\sigma} \right\Vert_\infty \left\Vert \frac{\mu}{\sigma^2} \right\Vert_\infty + \frac{1}{\epsilon^2} \left\Vert \frac{\diff{t}{\sigma}}{\sigma} \right\Vert_\infty \right) \right) \right] \\
& \hspace*{8cm} \cdot \widehat{Z}^2 \left( 2 \left\Vert \frac{\diff{t}{\sigma}}{\sigma} \right\Vert_\infty \left\Vert \frac{\mu}{\sigma^2} \right\Vert_\infty + \left\Vert \frac{\diff{t}{\mu}}{\sigma^2} \right\Vert_\infty \right).
\end{align*}
Therefore, we have that
\begin{equation*}
\frac{\widecheck{Z}^2}{\Vert \sigma \Vert_\infty^2} \leq \frac{Z^2_s}{\sigma^2 (\X{2}{s}, Y_s + \X{3}{s} )} \leq \frac{\widehat{Z}^2}{\epsilon^2},\qquad \lambda \times \IP \text{ a.s.}
\end{equation*}
and in particular
\begin{equation}
\label{eq:estimateStoppingTime}
\frac{\widecheck{Z}^2}{\Vert \sigma \Vert_\infty^2} \leq \tau = \X{2}{1} \leq \frac{\widehat{Z}^2}{\epsilon^2} \qquad \text{a.s.}
\end{equation}

\begin{example}[Embedding a Normal distribution into a Brownian motion with drift]
For $\mu \equiv m \in \IR$, $\sigma \equiv 1$ and $\nu = \mathcal{N} (0, \alpha^2 )$ for $\alpha > 0$ we know that $\tau = \alpha^2$ and $A_0 = - m \cdot \alpha^2$ solves the SEP. In this case we have that $g(x) = \alpha x$ and the above bounds for $Z$ become the explicit values $\alpha \leq Z \leq \alpha$ and the system \eqref{sys:fbsdeAndBCompact} simplifies to
\begin{align*}
W_s =&  \int_0^s \frac{1 }{ \alpha } \d B_{\X{2}{r}}, 
\quad
\X{2}{s} = \int _0 ^s \alpha^2 dr, 
\quad
\X{3}{s} = \int _0 ^s m \cdot \alpha^2 dr,
\\
Y_s =& \alpha W_1 - \X{3}{1} - \left( B_{\X{2}{1}} - B_{\X{2}{s}} \right)
\end{align*}
giving that $\tau = \X{2}{1} = \alpha^2$ a.s.\ which equals the above mentioned stopping time. We immediately find the correct value for $A_0$ since 
\begin{equation*}
A_0 = Y_0 = \E \left[ \left. Y_1 \right\vert \mathcal{F}_0 \right] = \E \left[ \left. \alpha W_1 - \int_0^1 m \alpha^2 \d r - B_{\X{2}{1}} + B_{\X{2}{0}} \right\vert \mathcal{F}_0 \right] = - m \alpha^2.
\end{equation*}
\end{example}

\begin{example}
\label{example:sigmoid}
Again let $\nu = \mathcal{N} (0, \alpha^2 )$ for $\alpha > 0$. Furthermore, set
\begin{equation*}
\sigma(t,a) = p^\sigma_1 + \frac{ p^\sigma_2}{1 + e^{-t}} + \frac{ p^\sigma_3}{1 + e^{-a}} \qquad \text{and} \qquad \mu(t,a) = p^\mu_1 + \frac{ p^\mu_2}{1 + e^{-t}} + \frac{ p^\mu_3}{1 + e^{-a}}
\end{equation*}
for the vectors $p^\sigma, p^\mu \in \IR^3$ containing parameters such that
\begin{equation*}
\epsilon := p^\sigma_1 + \min(0, p^\sigma_2) + \min(0,p^\sigma_3) > 0,
\end{equation*}
\begin{equation*}
2p^\sigma_2 p^\sigma_3 p^\mu_1 - p^\sigma_1 p^\mu_2 + \min(0, p^\sigma_2 p^\sigma_3 p^\mu_2 ) + \min(0,2 p^\sigma_2 p^\sigma_3 p^\mu_3 - (p^\sigma_3)^2 p^\mu_2) > 0
\end{equation*}
and
\begin{equation*}
\frac{1}{\alpha^2} + \frac{ p^\sigma_1 p^\mu_3 - 2 p^\sigma_3 p^\mu_1 + \min(0, p^\sigma_2 p^\mu_3 - 2 p^\sigma_3 p^\mu_2) - \max(0, p^\sigma_3 p^\mu_3)}{ 2\epsilon^3} > 0.
\end{equation*}
Then observe that all conditions of Proposition \ref{prop:conditionsStronSolution} and therefore also of Proposition \ref{prop:numericsSystem} are fulfilled,
\begin{equation*}
\widehat{Z} \leq \left( \frac{1}{\alpha^2} + \frac{ p^\sigma_1 p^\mu_3 - 2 p^\sigma_3 p^\mu_1 + \min(0, p^\sigma_2 p^\mu_3 - 2 p^\sigma_3 p^\mu_2) - \max(0, p^\sigma_3 p^\mu_3)}{ 2 \epsilon^3} \right)^{-\frac{1}{2}} < \infty
\end{equation*}
and also $\widecheck{Z}$ can be directly obtained since
\begin{align*}
& \Vert \sigma \Vert_\infty = p^\sigma_1 + \max(0, p^\sigma_2) + \max(0,p^\sigma_3), \\
& \Vert \mu \Vert_\infty = \max \left( p^\mu_1 + \max(0,p^\mu_2) + \max(0, p^\mu_3), -p^\mu_1 - \min(0,p^\mu_2) - \min(0, p^\mu_3) \right), \\
& \Vert \diff{a}{\sigma} \Vert_\infty = \vert p^\sigma_3 \vert, 
\quad
\Vert \diff{t}{\sigma} \Vert_\infty = \vert p^\sigma_2 \vert,
\quad
\Vert \diff{a}{\mu} \Vert_\infty = \vert p^\mu_3 \vert, 
\quad \Vert \diff{t}{\mu} \Vert_\infty = \vert p^\mu_2 \vert.
\end{align*}
\end{example}

\subsection{Iterative procedure}

To numerically approximate \eqref{sys:fbsdeAndBCompact} we first embed the hard bounds for $Z$, as found above, in the system, then create a Picard-type approximative sequence converging to \eqref{sys:fbsdeAndBCompact} and numerically approximate the terms of said sequence. Since we have a coupled system of FBSDEs with a truncated quadratic growth component, we combine \cite{IRZ10} and \cites{BZ08}.

Since $\X{2}{}$ is increasing and
\begin{equation*}
\X{2}{1} \leq \epsilon^{-2} \left( \frac{1}{\Vert g' \Vert_\infty ^2} + 2 \min \left\{0, \inf_{ (\theta,x) \in \IR_+ \times \IR } \left( \frac{ \sigma \cdot \diff{a}{\mu} - 2 \diff{a}{\sigma} \cdot \mu }{\sigma^3} \right) (\theta,x) \right\} \right)^{-1}
\end{equation*}
a.s.\ as stated in Equation \eqref{eq:estimateStoppingTime}, we only need a trajectory of $B$ untill this point. 

Furthermore, choose any starting value for $Z$ between the lower and upper bounds $\widecheck Z,\widehat Z$ respectively. Here we set the starting value $\Zo{0}{} = \Vert g' \Vert_\infty$ since $\widecheck{Z} \leq \Vert \frac{1}{g'} \Vert_\infty^{-1} \leq \Vert g' \Vert_\infty \leq \widehat{Z}$. Moreover, we define a truncation operator to incorporate the hard bounds for $Z$, namely, let $T:\IR\to \IR$ such that given $\widecheck Z,\widehat Z$, we define $T(z):= \min(\max(z,\widecheck Z), \widehat Z)$. The map $T$ is uniformly Lipschitz.

For the other starting conditions we choose $\p{Y}{0}{}=\Xz{0}{}=\Xd{0}{}=0$. Then we do the following iterations for $k \in \IN_0$:
\begin{align*}
\Xz{k+1}{s} =& \int_0^s \frac{ \left( T\big(\Zo{k}{r}\big) \right)^2}{\sigma^2 \left( \Xz{k+1}{r}, \p{Y}{k}{r} + \Xd{k+1}{r} \right) } \d r 
\\
\Xd{k+1}{s} =& \int_0^s \mu \left( \Xz{k+1}{r}, \p{Y}{k}{r} + \Xd{k+1}{r} \right) \frac{ \left( T\big(\Zo{k}{r}\big) \right)^2}{\sigma^2 \left( \Xz{k+1}{r}, \p{Y}{k}{r} + \Xd{k+1}{r} \right) } \d r 
\\
\p{W}{k+1}{s} =& \int_0^s \frac{\sigma \left( \Xz{k+1}{r}, \p{Y}{k}{r} + \Xd{k+1}{r} \right) }{ T\big(\Zo{k}{r}\big) } \d B_{\Xz{k+1}{r}} \\
\p{Y}{k+1}{s} =& g ( \p{W}{k+1}{1} ) - \Xd{k+1}{1} - \int_s^1 \sigma \left( \Xz{k+1}{r}, \p{Y}{k}{r} + \Xd{k+1}{r} \right) \d B_{\Xz{k+1}{r}}.
\end{align*}
Under the conditions imposed on $\mu,\sigma$ (Lipschitz and bounded) and $T$, all the coefficient maps of the truncated FBSDE system are Lipschitz continuous. It is currently not clear how to show that the iterative system converges to the solution of \eqref{sys:fbsdeAndBCompact} where one could possibly use a result similar to \cite{BZ08}*{Theorem 2.1}; this difficulty stems from the fact that the \cite{BZ08} methodology does not allow for either random drift or diffusion coefficients or $\sigma$ depending on $Z$. Note that in the limit ($k \to \infty$) the truncation does not affect the system as $\widecheck Z \leq Z\leq \widehat Z$ .

\subsection{Numerical procedure (time discretization)}

We introduce the time discretization $\pi = \{ 0 = t_0, \ldots , t_n = 1 \}$ for $n \in \IN$ and define $|\pi|:=\max_{i=0,\cdots,n} |t_{i+1}-t_i|$ as the mesh's modulus. The numerical approximation of the iterative system, for each $k\in \IN$ follows \cite{BT04} (or \cite{BZ08}). We apply an explicit Euler type approximation to the integrals and let throughout $t_i \in \pi \setminus \{ t_0 \}$. At first
\begin{align*}
\Xz{k+1}{t_0} =& 0,\qquad \Xd{k+1}{t_0} = 0 
\\
\Xz{k+1}{t_{i+1}} =& \Xz{k+1}{t_i} + (t_{i+1} - t_i ) \left( \frac{ T\big( \Zo{k}{t_i} \big) }{ \sigma ( \Xz{k+1}{t_i}, \p{Y}{k}{t_i} + \Xd{k+1}{t_i} )} \right)^2 \\
\Xd{k+1}{t_{i+1}} =& \Xd{k+1}{t_i} + (t_{i+1} - t_i ) \frac{ \mu \left( \Xz{k+1}{t_i}, \p{Y}{k}{t_i} + \Xd{k+1}{t_i} \right) \left( T\big( \Zo{k}{t_i} \big) \right)^2 }{ \sigma^2 \left( \Xz{k+1}{t_i}, \p{Y}{k}{t_i} + \Xd{k+1}{t_i} \right)},
\end{align*}
then
\begin{align*}
\p{W}{k+1}{t_0} = 0,
\quad 
\p{W}{k+1}{t_{i+1}} =& \p{W}{k+1}{t_i} + \frac{ \sigma ( \Xz{k+1}{t_i}, \p{Y}{k}{t_i} + \Xd{k+1}{t_i} )}{ T\big( \Zo{k}{t_i} \big) } \left( B_{\Xz{k+1}{t_{i+1}}} - B_{\Xz{k+1}{t_i}} \right)
\end{align*}
and
\begin{align*}
\p{Y}{k+1}{t_n} =& g\left( \p{W}{k+1}{1} \right) - \Xd{k+1}{1} \\
\p{Y}{k+1}{t_{i-1}} =& \E \left[ \left. \p{Y}{k+1}{t_{i}} \right\vert \mathcal{F}_{t_{i-1}} \right] \\ 
\Zo{k+1}{t_{i-1}} =& \frac{1}{t_{i} - t_{i-1}} \E \left[ \left. \left( \p{Y}{k+1}{t_{i}} - \E \left[ \left. \p{Y}{k+1}{t_{i}} \right\vert \mathcal{F}_{t_{i-1}} \right] \right) \left( W_{t_{i}} - W_{t_{i-1}} \right) \right\vert \mathcal{F}_{t_{i-1}} \right].
\end{align*}
The time discretization expression for $\Zo{k+1}{t_{i-1}}$ is somewhat non-standard when compared with the \cite{BT04} scheme. The inner term with the conditional expectation of $\p{Y}{k+1}{t_{i}}$ is a variance reduction trick which has been discussed in several places, e.g. \cite{LRS15}*{Section 5.4.2}; independently, the scheme's convergence (for fixed $k$ as $h\searrow 0$) follows via \cite{BT04}*{Theorem 3.1} yielding a convergence rate of order $h^{1/2}$ (the formulation associated to \cite{BZ08}*{Theorem 2.2} would deliver the same convergence). In the calculation of $Z$ we use that $\int _s ^1 \sigma ( \X{2}{r}, Y_r + \X{3}{r} ) \d B_{\X{2}{r}} = \int_s^1 Z_r \d W_r$ for all $s \in [0,1]$ and hence for small $h > 0$
\begin{align*}
Z_t
\approx& \frac{1}{h} \E \left[ \left. \int_t^{t+h} Z_r \d r \right\vert \mathcal{F}_t \right] \\
=& \frac{1}{h} \E \left[ \left. \left( Y_{t+h} - Y_t \right) \left( W_{t+h} - W_t \right) - \int_t^{t+h} \left( Y_r - Y_t + \left( W_r - W_t \right) Z_r \right) \d W_r  \right\vert \mathcal{F}_t \right] \\
=& \frac{1}{h} \E \left[ \left. Y_{t+h} \left( W_{t+h} - W_t \right) \right\vert \mathcal{F}_t \right] \\
=& \frac{1}{h} \E \left[ \left. \left( Y_{t+h} - \E \left[ \left. Y_{t+h} \right\vert \mathcal{F}_t \right] \right) \left(W_{t+h} - W_t \right) \right\vert \mathcal{F}_t \right].
\end{align*}
For the calculation of $W$ we implicitly assume that the value of $B$ is known for every $\Xz{k}{t_i}$ for all $k \geq 0$ and $t_i \in \pi$. This problem is more involved if the trajectory of $B$ is to be calculated at the beginning of the simulation. However, it can be eliminated by calculating the trajectory of $B$ just in time for the points needed by the method of Brownian bridge and storing all thereby obtained points. It is well known that the distribution of a Brownian bridge $B$ at time $t_1$ under the condition of the values of $B$ at the times $t_0 < t_1$ and $t_2 >t_1$ is
\begin{equation*}
B_{t_1}|B_{t_0},B_{t_2} \sim \mathcal{N} \left( B_{t_0} \cdot \frac{t_2-t_1}{t_2-t_0} + B_{t_2} \cdot \frac{t_1-t_0}{t_2-t_0} \ , \ \ \frac{(t_2-t_1)(t_1-t_0)}{t_2-t_0} \right),
\end{equation*}
see e.g.\ \cite{karatzas1991brownian}. Thus the simulation of $B$ at the exact points of time is straightforward as well. Lastly, the conditional expectations are computed via Least-Squares regression functions as shown in \cite{GLW05}; we project over $3$-dimensional polynomials up to degree $2$. 

After finishing the simulation of the FBSDE we can use the simulated trajectory of $B$ to simulate our process $A$ and apply the stopping time $\tau$ to see if $A_\tau$ has the desired distribution.

\begin{remarks}\label{remark:homocase}
For time homogeneous coefficients $\mu$ and $\sigma$ the FBSDE \eqref{fbsde:base} simplifies to the  decoupled FBSDE
\begin{align*}
\X{2}{s} &= \int_0^s \frac{Z_r^2}{\sigma^2 ( \bar{Y}_r )} \d r , \qquad 
\bar{Y}_s 
= g (W_1) - \int_s^1 \mu ( \bar{Y}_r ) \frac{Z_r^2}{\sigma^2 (\bar{Y}_r ) } \d r - \int_s^1 Z_r \d W_r.
\end{align*}
For this decoupled system one can use the same trick as above and inject in the BSDE the hard bounds on $Z$. Once truncated and using the condition on $\mu,\sigma$, the driver of the BSDE, say $f_R(y,z)=T^2(z) \mu(y)/\sigma^2(y)$ using the notation from before, is a standard uniformly Lipschitz driver in $y,z$ for which it is known (\cite{BT04}, \cite{BZ08}, \cite{GLW05}) that the Euler explicit scheme converges to the true solution. For weak solutions (see Remark \ref{remark:simulationWeakSimpler}) of the \SEP \ this explicit scheme is equivalent to the scheme we propose here. Hence, we have a special case where the convergence of our scheme is known.
\end{remarks}

\subsection{Numerical testing for Example \ref{example:sigmoid}}

For the parameters $\alpha=1$, $p^\sigma = (2,0.5,2)$ and $p^\mu = (1.5, -2.5, 0.5 )$ such that $\nu = \mathcal{N}(0, 1)$, 
\begin{equation*}
\sigma(t,a) = 2 + \frac{ 0.5 }{1 + e^{-t}} + \frac{ 2}{1 + e^{-a}} \qquad \text{and} \qquad \mu(t,a) = 1.5 + \frac{ -2.5 }{1 + e^{-t}} + \frac{ 0.5 }{1 + e^{-a}}
\end{equation*}
we get $\epsilon = 2$, $\Vert \sigma \Vert_\infty = 4.5$, $\widehat{Z} \leq \sqrt{\frac{8}{5}}$ and $\widecheck{Z} \geq 0.111$ giving $6 \times10^{-4} \leq \tau \leq 0.4$. A simulation with $10^5$ paths, $20$ time steps and $50$ iterations yielded values for $\tau$ in the interval $[0.061;0.161]$ and the starting value $Y_0 = -0.042$.

\begin{figure}[!hbt]
	\centering
	\subfigure
	{
		\includegraphics[width=8.0cm]{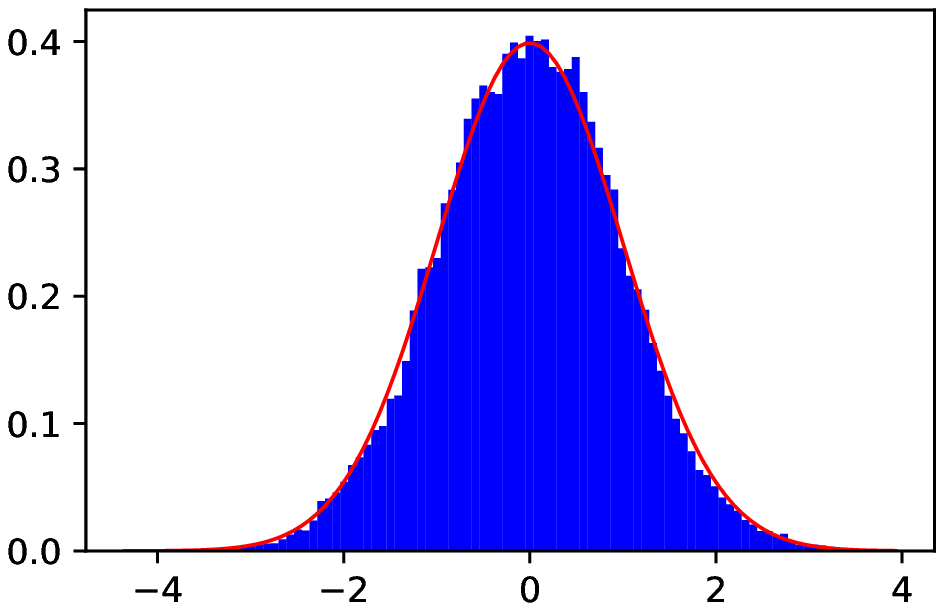}
	}
	\hspace{-1.05cm}
	\subfigure
	{
		\includegraphics[width=8.0cm]{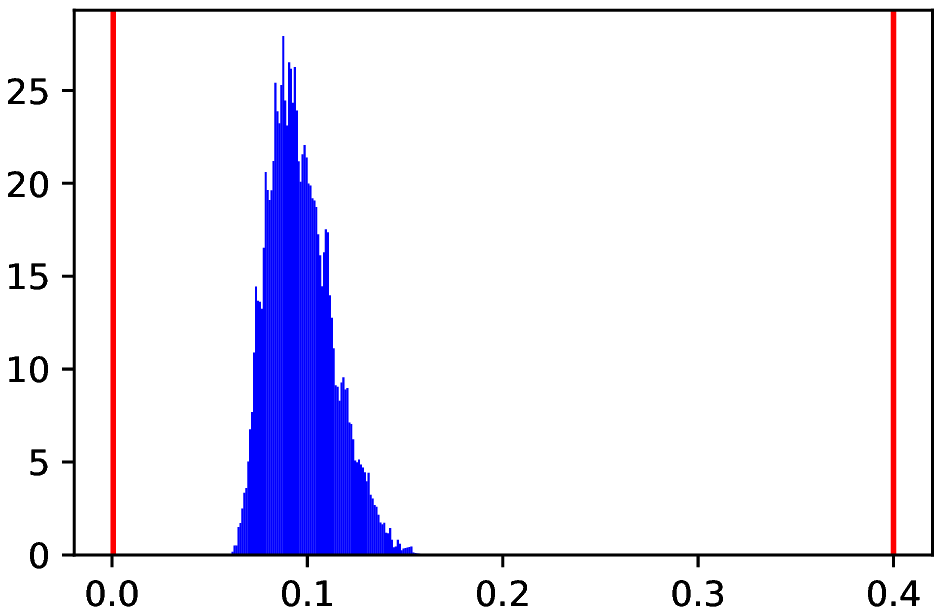}
	}
	\caption{On the left, Histogram of $10^5$ samples of $A_\tau$ against the density of the $\mathcal{N}(0,\alpha)$; on the right, the Histogram of the corresponding samples of $\tau$ and at $x=0.0055$ and $x=0.4$ the a priori hard bounds for the stopping time.}
	\label{Fig:Mtau-Normalpdf-tau}
\end{figure}

We simulated $A_\tau$ with initial condition $A_0=Y_0 = -0.042$. In Figure \ref{Fig:Mtau-Normalpdf-tau} one finds the histogram of the simulated values of the $A_\tau$ (left) and the stopping time $\tau$ (right). The histogram of $A_\tau$ indicates that our algorithm generates the sought normal distribution (with the appropriate characteristics). Also, D'Agostino and Pearson's \cites{D71,DP73} test for normality, applied to the simulated data $A_\tau$, yielded a $p$-value of $0.37$. Given such a high $p$-value we do not reject the hypothesis of normality at any reasonable significance level.

\appendix

\section{Appendix}

\begin{lemma}
\label{lemma:appendixGMeans} For $x\in \IR$ define $g(x) := F_\nu^{-1} ( \Phi (x))$ for $F_\nu$ and $\Phi$ being the cumulative distribution functions of $\nu$ and the standard normal distribution, and additionally define $\Phi_{0,\sigma}(x) = \Phi ( \frac{x}{\sigma} )$ for any $\sigma>0$. If $\Vert g' \Vert_\infty < \infty$, then there exist $K>0$ and $\sigma>0$ such that
\begin{itemize}
\item
for all $x<-K$ we have $F_\nu(x) \leq \Phi_{0,\sigma}(x) = \Phi ( \frac{x}{\sigma} )$ and
\item
for all $x > K$ we have $F_\nu (x) \geq \Phi_{0,\sigma}(x) = \Phi ( \frac{x}{\sigma} )$.
\end{itemize}

If additionally there exists a constant $c>0$ such that $0 < c \leq g'$ then there exist $K>0$ and $\sigma_1,\sigma_2>0$ such that 
\begin{itemize}
\item
for all $x>K$ we have \quad \
$\Phi_{0,\sigma_1}(x) = \Phi ( \frac{x}{\sigma_1} ) \leq F_\nu(x) \leq \Phi_{0,\sigma_2}(x) = \Phi ( \frac{x}{\sigma_2} )$ \quad and 
\item
for all $x < - K$ we have \ \
$\Phi_{0,\sigma_2}(x) = \Phi ( \frac{x}{\sigma_2} ) \leq F_\nu (x) \leq \Phi_{0,\sigma_1}(x) = \Phi ( \frac{x}{\sigma_1} )$.
\end{itemize}
\end{lemma}

\begin{proof}
Select $K,\sigma,\epsilon > 0$ such that for all $x > K$ we have $g( \frac{x}{\sigma}) \leq x$ and for all $ x < - K$ we have $g( \frac{x}{\sigma})- \epsilon \geq x$, which is possible since $0 \leq g' \leq C < \infty$. Then
\begin{align*}
\begin{array}{lc}
\text{for} \ x > K : \quad &
F_\nu (x) = F_\nu ( \frac{\sigma x}{\sigma} ) \geq F_\nu ( g( \frac{x}{\sigma} ) ) = F_\nu ( F_\nu^{-1} ( \Phi ( \frac{x}{\sigma} ) ) ) \geq \Phi ( \frac{x}{\sigma} ) = \Phi_{0,\sigma} (x), \\
\text{for} \ x < - K : \quad &
F_\nu (x) = F_\nu ( \frac{\sigma x}{\sigma} ) \leq F_\nu ( g( \frac{x}{\sigma} ) - \epsilon ) = F_\nu ( F_\nu^{-1} ( \Phi( \frac{x}{\sigma}) ) - \epsilon ) \leq \Phi ( \frac{x}{\sigma} ) = \Phi_{0,\sigma} (x).
\end{array}
\end{align*}
If additionally $0 < c \leq g'$ then we can choose $K_2 >0$ and some $\sigma_2>0$ such that for all $x > K_2$ we have $g( \frac{x}{\sigma_2}) - \epsilon \geq x$ and for all $ x < - K_2$ we have $g( \frac{x}{\sigma_2}) \leq x$. By an analogous argumentation as above we then obtain the remaining estimates. Setting $K$ as the maximum of $K$ from above and $K_2$ and furthermore $\sigma_1 := \sigma$ we have proved the statement.
\end{proof}

\bibliography{citations}

\end{document}